\documentclass[11pt,reqno,tbtags]{amsart}
\usepackage{amssymb}
\usepackage[utf8]{inputenc}   % om svenska bokstäver
\usepackage{url}
\usepackage[square,numbers]{natbib}
\bibpunct[, ]{[}{]}{;}{n}{,}{,}
\title%[]
{Phragm{\'e}n's and Thiele's election methods}

\date{27 November, 2016; 
revised 25 April, 2017 and 28 September, 2018.}
\author{Svante Janson}
\address{Department of Mathematics, Uppsala University, PO Box 480,
SE-751~06 Uppsala, Sweden}
\email{svante.janson@math.uu.se}
\newcommand\urladdrx[1]{{\urladdr{\def~{{\tiny$\sim$}}#1}}}
\urladdrx{http://www.math.uu.se/svante-janson/}

\subjclass[2010]{} 
\numberwithin{equation}{section}

\renewcommand\le{\leqslant}
\renewcommand\ge{\geqslant}

\allowdisplaybreaks

\newtheorem{thxmetod}{}
\newenvironment{metod}[1]%
{\renewcommand\thethxmetod{\textrm{\sc #1}}\begin{thxmetod}}{\end{thxmetod}}

{\begin{metod}{#1}\rm}{\end{metod}}

\newenvironment{metodx}%
{\par\smallskip\em\noindent\ignorespaces}{\par\smallskip}

\newenvironment{metod0}%
{\par\em}{\par}

\theoremstyle{plain}% default
\newtheorem{theorem}{Theorem}[section]

\newtheorem{corollary}[theorem]{Corollary}

\theoremstyle{definition}
\newtheorem{example}[theorem]{Example}

\newtheorem{remark}[theorem]{Remark}

\newtheorem*{ack}{Acknowledgement}

\theoremstyle{remark}

\newenvironment{romenumerate}[1][-10pt]{% optional argument changes indentation
\addtolength{\leftmargini}{#1}\begin{enumerate}% gives (i), (ii) etc.
 \renewcommand{\labelenumi}{\textup{(\roman{enumi})}}%
 \renewcommand{\theenumi}{\textup{(\roman{enumi})}}%
 }{\end{enumerate}}

\newcounter{oldenumi}
{\setcounter{oldenumi}{\value{enumi}}
\begin{romenumerate} \setcounter{enumi}{\value{oldenumi}}}
{\end{romenumerate}}

\newenvironment{PXenumerate}[1]{%  argument yields prefix
\begin{enumerate}% gives (#1 1), (#1 2) etc.
 \renewcommand{\labelenumi}{\textup{(#1\arabic{enumi})}}%
 \renewcommand{\theenumi}{\labelenumi}%
 }{\end{enumerate}}

\newenvironment{PXenumerateq}[1]{%  
\setcounter{oldenumi}{\value{enumi}}
\begin{enumerate} \setcounter{enumi}{\value{oldenumi}}
 \renewcommand{\labelenumi}{\textup{(#1\arabic{enumi})}}%
 \renewcommand{\theenumi}{\labelenumi}%
 }{\end{enumerate}}

\newenvironment{alphenumerate}[1][0pt]{% optional argument changes indentation
\addtolength{\leftmargini}{#1}\begin{enumerate}% gives (i), (ii) etc.
 \renewcommand{\labelenumi}{\textup{(\alph{enumi})}}%
 \renewcommand{\theenumi}{\textup{\alph{enumi}}}%
 }{\end{enumerate}}

\newcounter{thmenumerate}

\newcounter{xenumerate}   %no left indentation; thus wider lines

\newenvironment{val}[1][0pt]{% optional argument changes indentation
\addtolength{\leftmargini}{#1}%\smallskip
\begin{itemize}%
 }{\end{itemize}}

\newenvironment{itemizex}[1][-10pt]{% optional argument changes indentation
\addtolength{\leftmargini}{#1}%\smallskip
\begin{itemize}%
 }{\end{itemize}}

\newenvironment{xquote}[2]%
{\list{}{\rightmargin#2\leftmargin#1}\item[]}
{\endlist }

\newcommand\xfootnote[1]{\unskip\footnote{#1}$ $} %$ $ tycks eliminera fel
\newcommand{\refT}[1]{Theorem~\ref{#1}}
\newcommand{\refC}[1]{Corollary~\ref{#1}}

\newcommand{\refR}[1]{Remark~\ref{#1}}
\newcommand{\refS}[1]{Section~\ref{#1}}
\newcommand{\refSS}[1]{Section~\ref{#1}}

\newcommand{\refE}[1]{Example~\ref{#1}}
\newcommand{\refFn}[1]{Footnote~\ref{#1}}
\newcommand{\refApp}[1]{Appendix~\ref{#1}}
\newcommand\REM[1]{{\raggedright\texttt{[#1]}\par\marginal{XXX}}}

\begingroup
  \divide\count255 by 60
  \multiply\count255 by -60
  \advance\count255 by \time
  \ifnum \count255 < 10 \xdef\klockan{\the\count1.0\the\count255}
  \else\xdef\klockan{\the\count1.\the\count255}\fi
\endgroup

\newcommand\set[1]{\ensuremath{\{#1\}}}

\newcommand\bigpar[1]{\bigl(#1\bigr)}
\newcommand\Bigpar[1]{\Bigl(#1\Bigr)}

\newcommand\xcpar[1]{\{#1\}}

\newcommand\bigabs[1]{\bigl|#1\bigr|}
\newcommand\Bigabs[1]{\Bigl|#1\Bigr|}

\def\rompar(#1){\textup(#1\textup)}    % usage: \rompar(...)

\def\xexp(#1){e^{#1}}
\newcommand\ceil[1]{\lceil#1\rceil}
\newcommand\floor[1]{\lfloor#1\rfloor}
\newcommand\lrfloor[1]{\left\lfloor#1\right\rfloor}

\newcommand\punkt{.\spacefactor=1000}    % om problem!
    
\newcommand\ie{i.e\punkt}
\newcommand\eg{e.g\punkt}
\newcommand\viz{viz\punkt}
\newcommand\cf{cf\punkt}

\newcounter{CC}
\newcounter{cc}
\newcommand\ga{\alpha}
\newcommand\gb{\beta}

\newcommand\gam{\gamma}
\newcommand\gG{\Gamma}

\newcommand\gs{\sigma}

\renewcommand\phi{\xxx}  %% WARNING

\newcommand\cA{\mathcal A}

\newcommand\cC{\mathcal C}

\newcommand\cE{\mathcal E}

\newcommand\cS{{\mathcal S}}

\newcommand\cV{\mathcal V}

\newcommand\ett[1]{\boldsymbol1\xcpar{#1}}

\newcommand\qw{^{-1}}
\newcommand\qww{^{-2}}

\newcommand\phragmen{Phrag\-m{\'e}n}
\newcommand\andrae{Andr\ae}
\newcommand\StL{Sainte-Lagu\"e}

\newcommand\gsi{_{\gs\ni i}}
\newcommand\sumpii{\sum\gsi}
\newcommand\xx[1]{^{(#1)}}

\newcommand\tie{\relax}
\newcommand\subb[1]{\ensuremath{{}_{#1}}}
\newcommand\jfr{reduced vote}
\newcommand\EA{Elections Act}
\newcommand\sss[1]{${}_{#1}$}
\newcommand\bxi{\bar\xi}
\newcommand\fx{\hat f}
\newcommand\rmA{\textrm{A}}
\newcommand\rmB{\textrm{B}}
\newcommand\rmC{\textrm{C}}
\newcommand\rmD{\textrm{D}}
\newcommand\rmE{\textrm{E}}
\newcommand\rmx[1]{\textrm{#1}}
\newcommand\tA{\hat t_\rmA}
\newcommand\hcE{\widehat\cE}
\newcommand\MM{s}
\newcommand\ellx{k}

\newcommand\qEJR{\textit{EJR}}
\newcommand\qPJR{\textit{PJR}}
\newcommand\qJR{\textit{JR}}

\newcommand\VV{V}
\newcommand\opt{optimization}
\newcommand\Opt{Optimization}
\newcommand\optal{\opt{al}}
\newcommand\citex[1]{\texttt{[#1]}}

\newcommand\xref[1]{\texttt{(#1)}}

\begin{document}

\begin{abstract} 
The election methods introduced in 1894--1895 by Phragmén and Thiele, and
their somewhat later versions for ordered (ranked) ballots,
are discussed in detail. The paper includes definitions and examples and 
discussion of whether the methods satisfy some properties, including
monotonicity, consistency and various proportionality criteria.
The relation with STV is also discussed.
The paper also contains historical information on the methods.
% and their inventors. 
\end{abstract}

\maketitle

\begingroup
\renewcommand\footnote[1]{\relax}
\tableofcontents
\endgroup

\section{Introduction}\label{S:intro}

The purpose of this paper is to give a detailed presentation in English
of the election methods by Edvard Phragmén 
\cite{Phragmen1894, Phragmen1895, Phragmen1896,Phragmen1899}
and Thorvald Thiele \cite{Thiele}, originally
proposed in 1894 and 1895, respectively; we show also some properties of them.
(The presentation is to a large extent based on my discussion in Swedish in 
\cite[Chapters 13 and 14]{SJV6}.)
Both methods were originally proposed for unordered ballots (see below), but
ordered 
versions were later developed, so we shall  consider four different methods.
We also briefly discuss some other, related, 
methods suggested by \citet{Thiele}.

Versions of both Phragmén's and Thiele's methods have been used in Swedish
parliamentary elections (for distribution of seats within parties), and
Phragmén's method is still part of the election law, although in a minor
role, while Thiele's method is used for some elections in \eg{} city councils,
see \refApp{Ahistory}. 

Brief biographies of \phragmen{} and Thiele are given in \refApp{Abio}.

\subsection{Background}

The problem that \phragmen's and Thiele's methods try to solve is that of
electing a set of a given number $\MM$ of persons from a larger set
of candidates.
\phragmen{} and Thiele discussed this in the context of a parliamentary
election in a multi-member constituency; the same problem can, of course,
also occur in local elections, but also in many other situations such as
electing a board or a committee in an organization.

One simple method that has been used for centuries is the 
\emph{Block Vote}, where each voter votes for $\MM$ candidates, and the
$\MM$ candidates with the largest number of votes win. 
(See further \refApp{ABV}.) 
In the late 19th century, when 
(in Sweden)
political parties started to became organized,
it was evident that the Block Vote tends to give all seats to the largest party,
and as a consequence there was much discussion about
more proportional election methods that give representation also to
minorities.

Many proportional election methods have been constructed.
Among them, the ones that dominate today are 
\emph{list  methods}, where each voter votes for a party, and then each
party is given a number of seats according to some algorithm,
see  \refApp{Alist}. An important example is 
\emph{D'Hondt's method} (\refApp{ADHondt}) proposed by D'Hondt
\cite{DHondt1878,DHondt1882} in 1878.
\phragmen{} and Thiele  were inspired by D'Hondt's method, 
and \phragmen{} \cite{Phragmen1894, Phragmen1895}
called his method a generalization of D'Hondt's method,
but they
did not want a list method. They wanted to keep the voting method of the
Block Vote, where each voter chooses a set of persons, arbitrarily chosen
from the available candidates, without any formal role for parties.
Thus, 
a voter could
select candidates based on their personal merits and views on different
questions, and perhaps combine candidates from different parties and
independents. 
(This frequently happened, see Examples \ref{E1893a} and \ref{E1893b} from
the general election 1893.)
Then, an algorithm more complicated than the simple Block Vote would give
the seats to candidates in a way that, hopefully, would give a proportional
representation to minorities.
As we shall see in \refS{Sparty}, 
both \phragmen's and Thiele's methods achieve this at least in the
special case when there are parties with different lists, and every voter
votes for one of the party lists; then both methods give the same result as
D'Hondt's method. However, the methods were designed to cope also with more
complicated cases, when two different voters may vote for partly the same
and partly different candidates.
%(In such cases, \phragmen's and Thiele's  methods may give different
%results.)

\phragmen's and Thiele's  methods, especially the ordered versions, are thus
close in spirit to STV (\refApp{ASTV}), which also is a proportional
election method
where parties play no role, see \refS{SPhSTV} for a closer comparison.
(\phragmen{} knew about STV, at least in \andrae's version, 
and had  proposed a version of it before
developing his own method, see \refSS{SSPh1}.)

\subsection{Contents of the paper}

\phragmen's and Thiele's  election
methods are described in detail in Sections \ref{SPhru}--\ref{Sthio},
in both unordered (Sections \ref{SPhru}--\ref{Sdecap})
and ordered (Sections \ref{Sordered}--\ref{Sthio}) versions.
\refS{Sparty} shows that all the  methods reduce to D'Hondt's method
in the case of party lists. \refS{S1} treats the simple special case of
a single-member constituency, when only one candidate is elected.

\refS{Sex} contains a number of examples, many of them comparing the
methods; some of the examples are constructed to show
weak points of some method. (There are also examples in some other
sections.) 
Sections \ref{Smono2}--\ref{Sprop} discuss further properties of the methods
(monotonicity, consistency and proportionality), and \refS{SPhSTV} discusses
the relation
between \phragmen's method and STV.

Some variants of the methods are described in \refS{Svar}.

The purpose of this paper is not to advocate any particular method, but we
give a few conclusions in \refS{Sconclusions}.

The appendices contain further information, including biographies
and the history of the methods.
Furthermore,
\refApp{Aother} gives for the reader's convenience brief
descriptions of several
other election methods that are related to \phragmen's and Thiele's or occur
in the discussions.

\section{Assumptions and notation}\label{Snot}

For the election methods studied here, 
we assume, as discussed in the introduction, that each voter votes with a
ballot containing the names of one or several candidates. (Blank votes,
containing no candidates, may also be allowed, but in the methods treated
here they are simply ignored.) Parties, if they exist, have no formal role
and are completely ignored by the methods.

These election methods are of two different types, with different types of
ballots:

\begin{description}
\item [Unordered ballots]
The order of the names on a ballot does not matter.
In other words, 
each ballot is regarded as a set of names.
(Sometimes called \emph{approval ballots}, since the voter can be seen as
approving some of the candidates.)

\item [Ordered ballots]
The order of the names on a ballot matters.
Each ballot is an ordered list of names.
(Sometimes called \emph{ranked ballots}, or \emph{preferential voting}.)
\end{description}
Of course, the practical arrangemants may vary; for example,
the voter may write the names by hand, 
or there might be a printed or electronic list of all candidates where the
voter marks his choices by a tick (unordered ballots) or by 1, 2,
\dots in order of preference (ordered ballots).

In some election methods, there are restrictions on the number of candidates
on each ballot (see \refApp{Aotheru}--\ref{Aothero} for examples). 
We make no such assumptions for \phragmen's and
Thiele's methods; each voter can vote for an arbitrary number of candidates.
However, the unordered versions can be modified by allowing at most $\MM$ (the
number to be elected) candidates on each ballot (for philosophical reasons
or for practical convenience); this was for example
done in the version of Thiele's
method used in Sweden 1909--1922, see \refApp{Ahistory}.
\xfootnote{
I do not know whether \phragmen{} intended this restriction or not.
I cannot find anything stated explicitly about it in \phragmen's papers
\cite{Phragmen1894,Phragmen1895,Phragmen1899}, 
but all his examples are of this type.
Remember also that \phragmen{} intended his method as an alternative to the
then used Block Vote (\refApp{ABV}), where unordered ballots with this
restriction were used, and it is possible that he intended the same for
his method.
\citet{Thiele}, on the contrary, gives several examples with more names on
the ballots than the number elected.
}
For the ordered versions, there is no point in forbidding (or allowing) more
than $\MM$ names on each ballot, since only the $\MM$ first names can matter.
\xfootnote{
This does not hold for all methods using ordered ballots, for example neither
for STV (\refApp{ASTV}) nor for scoring rules (\refApp{Aborda}), as is
easily seen be considering the case $\MM=1$.
}

It is often convenient to consider the different types of ballots that
appear, and count the number of ballots of each type.
If $\ga$ is a type of ballot that appears (or might appear), let $v_\ga$
denote the number of ballots of that type, \ie, the number of voters
choosing exactly this ballot.
Furthermore, let $V=\sum_\ga v_\ga$ denote the total number of (valid) votes,
and let 
$p_\ga:=v_\ga/V$ be the proportion of the votes that are cast for $\ga$.

\begin{remark}\label{Rhomo}
    Many election methods are \emph{homogeneous}, meaning that the result
	depends only on the proportions $p_\ga$.
This includes \phragmen's and Thiele's methods.
\xfootnote{
Although homogeneity seems like a natural property, not all election methods
	used in practice are strictly homogeneous. 
In quota methods, see \eg{} \cite{Pukelsheim}, as well as in STV 
(\refApp{ASTV}),
a ``quota'' is calculated, and this is (perhaps by tradition) usually
rounded to an integer, 
see \cite[Section 5.8]{Pukelsheim} for several examples, meaning that the
election method is 
not  homogeneous.
Similarly, in \phragmen's method as used in Sweden, see
\refApp{APh-vallag},
rounding to two decimal places is specified for all intermediary
calculations;
again this means that the method is not strictly homogeneous.
However, in both cases the methods are 
asymptotically homogeneous, as the number of votes gets large,
and for practical purposes they can be regarded as homogeneous, at least for
public elections.
}
\end{remark}

\begin{remark}
It is possible to let different voters have different weights, which in
principle could be any positive real numbers. The only difference is that
$v_\ga$ now is the total weight of all voters choosing $\ga$, and that this
is a real number, not necessarily an integer.
\xfootnote{
This was the case in local elections in Sweden 1909--1918, when a voter
had 1--40 votes depending on income, and a modification of Thiele's method
was used, see \refApp{Ahistory}.
}
We leave the trivial modifications for this extension to the reader, and
continue to talk about numbers of votes.
\end{remark}

\begin{remark}
Every election method has to have provisions for the case that a tie occurs
between two or more candidates. Usually ties are resolved by lot, although
other rules are possible.
\xfootnote{Non-mathematical rules are occasionally used, for example giving
  preference to the oldest candidate. (Such rules have not been used in
  connection with \phragmen's or Thiele's methods as far as I know.)} 
\phragmen{} originally proposed a special rule for his method,
see \refApp{APhragmen1894}, but he seems to have dropped this later and we
shall do the same. We assume that ties are resolved by lot or by some other
rule, and we shall usually not comment on this.
\end{remark}

\subsection{Some notation}
We let throughout 
$\MM$ be the number of seats, \ie, the number of
candidates to be elected; we assume that $\MM$ is
fixed and determined before the election.
We use a variable such as $i$ for an unspecified candidate.
In examples, candidates are usually
denoted by capital letters A, B,\dots

The \emph{outcome} of the election is the set $\cE$ of elected candidates.
By assumption, $|\cE|=\MM$,
where $|\cE|$ denotes the  number of elements of the set $\cE$.
When there are ties, there may be several possible outcomes.

In the discussions below, 'candidate' and 'name' are synonymous. 
Similarly, we identify a voter and his/her ballot.

In numerical examples, $\doteq$ is used for decimal approximations
(correctly rounded).

\subsubsection{Unordered ballots}
In a system with unordered ballots, each ballot can be seen as a set of
candidates, so 
the different types of ballots are  subsets of the set of all candidates.
We denote such sets by $\gs$.
Note that candidate $i$ appears on a ballot of type $\gs$ if and only if $i\in
\gs$; hence, the total number of ballots containing $i$ is $\sum_{\gs\ni i}
v_{\gs}$. 

\subsubsection{Ordered ballots}
In a system with ordered ballots, 
the different types of ballots are  ordered list of some (or
all) candidates.
We denote such ordered lists by $\ga$.

\section{\phragmen's unordered method}\label{SPhru}
 
\subsection{\phragmen's formulation}\label{SPhru1}

\phragmen{} presented his method in a short note in 1894
\cite{Phragmen1894}, 
followed by further discussions, motivations and explanations in
\cite{Phragmen1895,Phragmen1896,Phragmen1899},  
see \refApp{APhragmen1894}.
His definition is as follows (in my words and with my notation).
\phragmen{} assumes
that the ballots are of the unordered type in \refS{Snot}, \ie, 
that each ballot contains a set of candidates,
without order and without other restrictions. 
A detailed example is given in \refSS{SSPhru-ex}; %, see \refE{EPhr1894}. %
further examples are given in \refS{Sex}.

\begin{metod}{\phragmen's unordered method}
Assume that each ballot has some \emph{voting power} $t$; this number is
the same for all ballots and will be determined later. 
A candidate needs total voting power $1$ in order to be elected.
The voting power of a
ballot may be used by the candidates on that ballot, and it may
be divided among several of the candidates on the ballot. 
During the
procedure described below, some of the voting power of a ballot may be already
assigned to already elected candidates; the remaining voting power of the
ballot
is free.

The seats are distributed one by one.

For each seat, each remaining candidate
may use all the free voting power of each ballot that includes
the candidate. (I.e., the full voting power $t$ except for the voting power
already assigned from that ballot to candidates already elected.)
The ballot voting power $t$ %for each ballot 
that would give the candidate voting power $1$ is
computed, and the candidate requiring the smallest voting power $t$ is
elected. All free (\ie, unassigned) 
voting power on the ballots that contain the elected
candidate is assigned to that candidate, and these assignments remain fixed
throughout the election. 

The computations are then repeated for the next seat, for the remaining
candidates, and so on.
\end{metod}

Ties are, as said in \refS{Snot}, 
broken by lot, or by some other supplementary rule.
(See \refApp{APhragmen1894} for \phragmen's original suggestion.)

Note that the required voting power $t$ increases for each seat, except
possibly in the case of a tie.

We consider the method defined above in some more detail.
We begin with the first seat. 
Since candidate $i$ appears on $\sumpii v_\gs$ ballots,
the smallest voting power $t$ that makes it possible for $i$ to get the
seat, provided it gets all voting power from each available ballot, is thus 
$t_i=1/\sumpii v_\gs$. 
\phragmen's rule is that the seat is given to the candidate $i$ that
requires the 
smallest voting power $t_i$. 
Hence the first seat goes to the candidate appearing on the largest number
of ballots.

Suppose that the first seat goes to candidate $i$, and that this requires voting
power $t\xx1=t_i$. The ballots containing $i$ have thus all used voting power
$t\xx1$ for the election of $i$; this allocation will remain fixed forever.
We now increase the voting power $t$ of all ballots, noting that the ballots
containing $i$ only have $t-t\xx1$ free voting power
available for the remaining candidates.
We again calculate the smallest $t$ such that some candidate may be given total
voting power 1; we give this candidate, say $j$, 
the second seat and let $t\xx2$ be the
required voting power. Furthermore, on the ballots containing $j$, we
allocate all available voting power to the election of $j$.

We continue in the same way.
In general, suppose that $n\ge0$ seats have been allocated so far, and that this
requires voting power $t\xx n$. Suppose further that on each ballot for the set
$\gs$, an amount $r_\gs$ of the voting power already is used, with $0\le
r_\gs\le t\xx n$. If the voting power of each ballot is increased to $t\ge
t\xx n$, 
then each ballot for $\gs$ has thus a free voting power $t-r_\gs$, which 
can be used by any of the candidates on the ballot.
The voting power available for candidate $i$ is thus
%$  \sum\gsi v_\gs(t-r_\gs)$, 
\begin{equation}
  \sum\gsi v_\gs(t-r_\gs) = t\sum\gsi v_\gs -\sum\gsi v_\gs r_\gs
\end{equation}
and for this to be equal to 1 we need $t$ to be
\begin{equation}\label{ti}
t_i= \frac{1+\sum\gsi v_\gs r_\gs}{\sum\gsi v_\gs}.
\end{equation}
The next seat is then given to the candidate with the smallest $t_i$;
if this is candidate $i$, the required voting power  $t\xx{n+1}$ is thus $t_i$,
so $r_\gs$ is updated to 
\begin{equation}\label{r'}
r'_\gs:=t\xx{n+1}=t_i  = \frac{1+\sum\gsi v_\gs r_\gs}{\sum\gsi v_\gs}
\end{equation}
for each $\gs$ such that $i\in\gs$.
($r_\gs$ is unchanged for $\gs$ with $i\notin\gs$.)

These formulas give an algorithmic version of \phragmen's method.

\begin{remark}
  \label{Rload}
In \cite{Phragmen1899}, \phragmen{} describes the method in an equivalent
way  using the term \emph{load} instead of voting power;
the idea is that when a candidate is elected, the participating ballots
incur a total load of 1 unit, somehow distributed between them. 
The candidates are elected sequentially. In each round, the loads  are
distributed and the candidates are chosen such that the
maximum load of a 
ballot is as small as possible.
(The same description is used by \citet{Cassel}.)
This is also a useful formulation, and it will sometimes be used below.
\end{remark}

\begin{remark}
\citet{Phragmen1899} 
illustrates also the method by imagining the different groups of
ballots as represented by cylindrical vessels, with base area proportional to
the number of ballots in each group.
The already elected candidates are represented by a liquid that is fixed in
the vessels, and 
 the additional voting power required to
elect another candidate is represented by pouring 1 unit of a liquid into the
vessels representing a vote for that candidate, distributed among these
vessels such that the height of the liquid will be the same in all of them.
This is to be tried for each candidate;
the candidate that requires the smallest height is elected, 
and the corresponding amounts of liquid are added to the vessels and fixed
there.
\end{remark}

\begin{remark}\label{Rtime}
  Sometimes it is convenient to think of the voting power as increasing
  continuously with
  time; at time $t$ each ballot has voting power $t$. 
The voting power available to each candidate thus also increases with time,
and as soon as some candidate reaches voting power 1, this candidate is
elected
and the free voting power on each participating ballot is permanently
assigned to this candidate (which reduces the free voting power to 0 for
these ballots, and thus typically reduces the available voting power for
other candidates). This is repeated until $\MM$ candidates have been elected.
\end{remark}

\subsection{An equivalent formulation}\label{SPhru2}
The numbers $t_i$ above will in practice be very small, and it is often more
convenient 
to instead use $W_i:=1/t_i$.
We also let $q_\gs:=v_\gs r_\gs$; this is the total voting power allocated
so far from the ballots of type $\gs$, and can be interpreted as the
(fractional) number of seats already elected by these ballots; $q_\gs$ is
called the \emph{place number} of this group of ballots.
Note that $\sum_\gs q_\gs$ always
equals the number of candidates elected so far. 

This leads to the following algorithmic formulation. To see that it really
is equivalent to the formulation in \refS{SPhru1}, it suffices to note that with
$W_i=1/t_i$ and $q_\gs=v_\gs r_\gs$, \eqref{wi} below is the same as
\eqref{ti}, and the update rule \eqref{q'} is the same as \eqref{r'}.

\begin{metod}{\phragmen's unordered method, equivalent formulation}
Seats are given to candidates sequentially, until the desired number
have been elected. During the process,
each type of ballot, \ie,
each group of identical ballots, 
is given a 
\emph{place number}, which is a rational non-negative number that can be
interpreted as the fractional number of seats elected so far by these
ballots; the sum of the place numbers is always equal to the number of seats
already allocated. 
The place numbers are determined recursively and the seats
are allocated by the following rules:
\begin{romenumerate}
\item 
Initially all place numbers are $0$.
\item \label{phrub}
Suppose that $n\ge0$ seats have been allocated.
Let $q_\gs$ denote the place number for the ballots with a set $\gs$
of candidates;
thus $\sum_\gs{q_\gs}=n$.
The total number of votes for candidate $i$ is $\sum_{\gs\ni i} v_\gs$, and 
the total place number of the ballots containing candidate $i$ is
$\sum_{\gs\ni i} q_\gs$. The \jfr\xfootnote{
The Swedish term is \emph{j\"amf\"orelsetal} (comparative figure).}
for candidate $i$ is defined as
\begin{equation}\label{wi}
  W_i:=\frac{\sum_{\gs\ni i} v_\gs}{1+\sum_{\gs\ni i} q_\gs},
\end{equation}
\ie, the total number of votes for the candidate divided by 
$1+$  their total
place number.

\item 
The next seat is given to the candidate $i$ that has the largest $W_i$.
\tie

\item \label{phrud}
Furthermore, if candidate $i$ gets the next seat, 
then the place numbers are updated
for all sets $\gs$ that participated in the election, \ie, the sets $\gs$
such that $i\in\gs$. For such $\gs$, the new place number is
\begin{equation}\label{q'}
 q'_\gs:=\frac{v_\gs}{W_i}
=
\Bigpar{1+\sum_{\gs\ni i} q_\gs}\frac{v_\gs}{\sum_{\gs\ni i} v_\gs}.
\end{equation}  
For $\gs$ such that $\gs\not\ni i$, $q'_\gs:=q_\gs$.
\end{romenumerate}
Steps \ref{phrub}--\ref{phrud} are repeated as many times as desired.
\end{metod}

We  see that the number $W_i$ may be interpreted as the total number of
votes for candidate $i$, reduced according to the extent to which the ballots
containing $i$ already have successfully participated in the election of other
candidates. (For that reason, we call $W_i$ the ``\jfr'' above.)
Cf.\ D'Hondt's method, see \refApp{ADHondt}, which as said above,
\phragmen{} tried to generalize.

\begin{remark}\label{Rc1}
  Let  $W\xx n$ be the winning (\ie, largest) \jfr{} when the
$n$-th seat is filled.
Then, by \eqref{q'}, during
the calculations above,
the current place number $q_\gs$ is $v_\gs/W\xx \ell$,
if  the last time that some candidate on the ballot (\ie, in $\gs$)
was elected was in round $\ell$. (Provided any of them has been elected;
otherwise $q_\gs=0$; in this case we may define $\ell=0$ and $W\xx0:=\infty$.) 
\end{remark}

\begin{remark}\label{Rplace}
  It is in practice convenient to use place numbers $q_\gs$
defined for groups of identical  ballots as above,
but it is sometimes also useful to consider the place number of an
individual ballot; for a ballot of type $\gs$ this is $r_\gs=1/W\xx\ell $,
with $\ell$ as in \refR{Rc1}.
\end{remark}

\begin{remark}\label{Rsheppard}
  The equivalence with the formulation in \refSS{SPhru1} shows that 
\begin{equation}\label{Wntn}
W\xx n=1/t\xx n  
\end{equation}
with $t\xx n$ as in
\refSS{SPhru1}. (And $t\xx0:=0$.)
\end{remark}

\subsection{An example}\label{SSPhru-ex}
\citet{Phragmen1894,Phragmen1895,Phragmen1896} illustrates his (unordered)
method with
the following example (using slightly varying descriptions in the different
papers).  
We present detailed calculations (partly taken from \phragmen)
using both formulations above.

\begin{example}[\phragmen's unordered method] \label{EPhr1894}
\quad 

Unordered ballots. 3 seats. \phragmen's method.
\begin{val}
\item [1034]ABC
\item [519]PQR
\item [90]ABQ
\item [47]APQ
\end{val}
The total numbers of votes for each candidate are thus
\begin{val}
\item [A]1171
\item [B]1124
\item [C]1034
\item [P]566
\item [Q]656
\item [R]519.
\end{val}
Using the formulation of \phragmen's method in \refSS{SPhru1},
we see that the smallest voting power that gives some candidate command of 
voting power
1 is $t\xx1=1/1171\doteq0.000854$, which gives A voting power $1171/1171=1$. 
Hence A (which has the largest number of votes) is elected to the first seat. 

If now the voting
power is increased to $t>t\xx1$, then each ballot ABC, ABQ or APQ has free
voting power $t-t\xx1=t-1/1171$, while each ballot PQR has free voting power
$t$.
Hence, the voting power that each of the remaining candidates can use is
\begin{val}
\item [B:] $1124(t-t\xx1)$
\item [C:] $1034(t-t\xx1)$
\item [P:] $519t+47(t-t\xx1)=566t-47t\xx1$
\item [Q:] $519t+137(t-t\xx1)=656t-137t\xx1$
\item [R:] $519t$.
\end{val}
In order for these values to be equal to 1, we need for B the voting power
$t$ to be, see \eqref{ti},
$$
t_\rmB:=\frac{1+1124t\xx1}{1124}=t\xx1+\frac{1}{1124}
%=\frac{2295}{1124\cdot1171}
=\frac{2295}{1316204}
\doteq0.001744,
$$
and for Q a voting power
$$
t_\rmx{Q}:=\frac{1+137t\xx1}{656}
=t\xx1+\frac{1-519t\xx1}{656}
%=\frac{327}{164\cdot1171}
=\frac{327}{192044}
\doteq0.001703,
$$
while the remaining candidates obviously require more than either B or Q.
Since $t_\rmx{Q}<t_\rmB$, $Q$ is elected to the second seat, and
$t\xx2=t_\rmx{Q}$.

For the third seat, if the voting power of each ballot 
is increased further to $t>t\xx2$, then each ballot ABC has free voting
power $t-t\xx1$, and each ballot PQR, ABQ or APQ has free voting power
$t-t\xx2$.
Hence, the voting power that each remaining candidate can use is
\begin{val}
\item [B:] $1034(t-t\xx1)+90(t-t\xx2)=1124t-1034t\xx1-90t\xx2$
\item [C:] $1034(t-t\xx1)$
\item [P:] $566(t-t\xx2)$
\item [R:] $519(t-t\xx2)$.
\end{val}
In order for this to equal 1, B needs the voting power $t$ to be 
$$
t_\rmB:=\frac{1+1034t\xx1+90t\xx2}{1124}
=\frac{195525}{107928728}
\doteq0.001812
$$
while P needs $t$ to be
$$
t_\rmx{P}:=
\frac{1+566t\xx2}{566}=t\xx2+\frac{1}{566}=\frac{188563}{54348452}
\doteq0.003470
$$
while C and R obviously require more than B and Q, respectively.
Since $t_\rmB<t_\rmx{P}$, B is elected to the third seat.
(Also, $t\xx3=t_B$.)

Hence, the elected by \phragmen's method are AQB.

\smallskip
Using ``load'' instead of ``voting power'', see \refR{Rload}, we do the same
calculations; we now say that first A is elected which
gives a load $t\xx1=1/1171$ to each ballot
ABC, ABQ or APQ (total load $1171t\xx1=1$); then Q is elected which gives
a load $t\xx2=327/192044$ to each ballot PQR and an additional load
$t\xx2-t\xx1=163/192044$ to each ballot ABQ or APQ 
(total new load $519\cdot 327/192044 + 137\cdot163/192044=1$);
finally, B is elected which gives an additional load
$t\xx3-t\xx1=103357/107928728$ to each ballot ABC
and
$t\xx3-t\xx2=11751/107928728$ to each ballot ABQ
(total new load $1034\cdot103357/107928728+90\cdot11751/107928728=1$).
The final loads on the ballots of the four types are
$(t\xx3,t\xx2,t\xx3,t\xx2)$, where $t\xx2=327/192044\doteq0.001703$ and 
$t\xx3=195525/107928728\doteq 0.001812$.

\smallskip
Using the formulation with place 
numbers and \jfr{s} in
\refSS{SPhru2}, we obtain the same result by similar but somewhat different
calculations, cf.~\eqref{Wntn}.
For the first seat, the \jfr{s} $W_i$ are just the number of votes
for each candidate; hence A is elected with $W_\rmA=1171$. This gives a place
number $1/1171$ for each participating
ballot, and thus $1034/1171\doteq0.8830$ for all ballots ABC
together, 
$90/1171\doteq0.0769$ for the ballots ABQ and $47/1171\doteq0.0401$ 
for the ballots APQ.

Hence, for the second seat, the \jfr{s} for B and Q are, by \eqref{wi},
\begin{equation*}
  W_\rmB=\frac{1124}{1+1124/1171}=\frac{1316204}{2295}\doteq573.51
\end{equation*}
(since the 1124 ballots containing B have a combined place number $1124/1171$)
and
\begin{equation*}
  W_\rmx{Q}=\frac{656}{1+137/1171}=\frac{192044}{327}=587.29
\end{equation*}
(since the 656 ballots containing Q have a combined place number $137/1171$),
while the remaining candidates have smaller \jfr{s}
($W_\rmC\doteq549.12$, $W_\rmx P\doteq544.16$,
$W_\rmx R=519$).
Hence Q is elected to the second seat. The place numbers for the four groups
of ballots are
$\frac{1034}{W_\rmA}=\frac{1034}{1171}\doteq 0.8830$,
$\frac{519}{W_\rmx Q}=\frac{169713}{192044}\doteq0.8837$,
$\frac{90}{W_\rmx Q}=\frac{14715}{96022}\doteq0.1532$,
$\frac{47}{W_\rmx Q}=\frac{15369}{192044}\doteq0.0800$, with sum 2.

For the third seat, we have the \jfr{s}
$$
W_\rmB=\frac{1124}{1+1034/1171+14715/96022}=\frac{107928728}{195525}
\doteq551.99
$$
and
$$
W_\rmx{P}=\frac{566}{1+169713/192044+15369/192044}=\frac{54348452}{188563}
\doteq288.22
$$
while the remaining candidates have smaller \jfr{s}
($W_\rmC\doteq549.12$ as for the second seat,
$W_\rmx R\doteq275.52$). Hence B is elected to the third seat.
\end{example}

\section{Thiele's unordered methods}\label{Sthiu}

 \citet{Thiele} praised \phragmen's contribution \cite{Phragmen1894}, 
but proposed a different method based on a different idea.
Thiele realized that his idea led to a difficult
optimization problem that was not practical to solve, so he also proposed two
approximations of the method.
There are thus three different methods by Thiele for unordered ballots; 
we may call them
\emph{Thiele's \opt{} method},  
\emph{Thiele's addition method}
and \emph{Thiele's elimination method}, 
\xfootnote{In \citet{Thiele} (in Danish), the \opt{} method has no name,
  the addition 
  method is called \emph{Tilf{\o}jelsesreglen} and the elimination method
is called \emph{Udskydelsesreglen}; in his examples, there are also captions
in French, with the names \emph{règle d'addition} and \emph{règle de rejet}.

When the addition method was used in
Sweden 1909--1922, it was called \emph{Reduk\-tions\-regeln} (the Reduction
Rule). 
}
but
since only the addition method has found practical use, 
we mainly consider this method and
we often call the addition method simply 
\emph{Thiele's method}.
(The addition method was also the only of the methods that was considered in
the discussions in Sweden in the early 20th century, see \refApp{Ahistory}
and \eg{} \cite{Cassel}, \cite{bet1913}.)
The three methods are defined below.
A detailed example is given in \refSS{SSthiu-ex}; %, see \refE{EPhr1894-Th} %
further examples are given in \refS{Sex}.

\subsection{Thiele's \opt{} method}\label{SSthiuopt}
Thiele's idea was that a voter that sees $n$ of the candidates on 
his\xfootnote
{In 1895, only men were allowed to vote, in Denmark as well as in Sweden.}
ballot elected, will feel a \emph{satisfaction} $f(n)$, for some increasing
function $f(n)$, and the result of the election should be the set of $\MM$
candidates that maximizes the total satisfaction, \ie, the sum of $f(n)$
over all voters.

In formulas, if a set $\cE$ of candidates is elected, a voter that has voted
for a set $\gs$ will feel a satisfaction $f(|\gs\cap\cE|)$.

As Thiele notes, we can without loss of generality assume $f(0)=0$ and
$f(1)=1$. We also let $w_n:=f(n)-f(n-1)$, the added satisfaction when a
voter sees the $n$-th candidate elected; thus 
\begin{equation}
  \label{fw}
f(n)=\sum_{k=1}^n w_k.
\end{equation}

Of course, the result of the method depends heavily on the choice of the
function $f(n)$. Thiele discusses this, and 
argues that the function depends on the purpose of the election. He claims
that for the election of a government or governing body, each new member is
as important as the first, so 
\begin{equation}\label{fstrong}
  f(n)=n
\end{equation}
(\ie, $w_n=1$); he calls this the \emph{strong} method. 
On the other hand, for the election
of a committee for comprehensive investigation of some issue, he sees no
point in having several persons of the same meaning, so he sets 
\begin{equation}\label{fweak}
  f(n)=\ett{n\ge1}:=
  \begin{cases}
	0,&n=0,
\\1,&n\ge1,
  \end{cases}
\end{equation}
(\ie, $w_1=1$, $w_n=0$ for $n\ge2$);
he calls this the \emph{weak} method.
Finally, Thiele claims that there are many cases between these two
extremes, in particular when electing representatives for a society, where
proportional representation is desired.
\xfootnote{
Whether this argument is convincing or not is perhaps for the reader to
decide. 
\citet{Tenow1912} argues that Thiele here
rather seems to evaluate the methods by the desirability of their outcome,
and that his ``satisfaction'' thus becomes a fiction and is chosen to
achieve the desired result.
However, while this philosophical question may be relevant for
applications of the methods, 
it is irrelevant for our main purpose, which is to present the
methods as algorithms and give some of their mathematical properties.
}

Thiele notes that the choice \eqref{fstrong}
gives the result that the $\MM$ candidates with the largest numbers of votes
will be elected; this is thus Approval Voting (\refApp{Aapproval}).
(If each ballot can contain at most $\MM$ names we obtain Block voting,
\refApp{ABV}.) 
Thiele  notes that this yields a method that is far from proportional.

On the other hand, Thiele notes that the choice 
\begin{equation}\label{f}
f(n)=1+\frac{1}2+\dots+\frac{1}n,
\end{equation}
(\ie, $w_n=1/n$), 
yields a proportional method in the case when no name appears on more than one
type of ballot, see \refS{Sparty}.
This is therefore Thiele's choice for his \emph{proportional} method.

In the sequel we shall use the choice \eqref{f}, 
except when we explicitly state otherwise. 
%(This is also the version that has found practical use.)
\xfootnote{
The strong version \eqref{fstrong} is as said above equivalent to
Approval Voting.

The weak version \eqref{fweak} seems to be more interesting
mathematically than for practical applications.
See \cite[Examples 5 and 6]{Thiele} for an example including
all three methods with the weak function \eqref{fweak},
in this case yielding different results.
} 
Thiele's \opt{} method for proportional elections is thus the following:

\begin{metod}{Thiele's \opt{} method}
  For each set $\cS$ of $\MM$ candidates, calculate the ``satisfaction''
  \begin{equation}
F(\cS):=\sum_\gs v_\gs f(|\gs\cap\cS|),	
  \end{equation}
where the function $f$ is given by \eqref{f}.
Elect the set $\cS$ of the given size
that maximizes $F(\cS)$.
\end{metod}

However, Thiele notes that the maximization over a large number of sets
$\cS$ is impractical. With $n$ candidates to the $\MM$ seats, there are  
$\binom n\MM$ sets $\cS$ that have to be considered, and as an example,
\citet{Thiele} mentions 30 candidates to 10 seats, 
  when there are more than 30 million combinations (30\;045\;015). 
% tryckfel 3 milj. i [Thiele]
\xfootnote{Using a concept not existing in Thiele's days, the
  problem of finding the maximizing 
  set(s) is NP-hard, see \cite[Theorem 1]{Aziz:computational}
or \cite[Theorem 3]{SkowronFL}.}
\citet{Thiele} thus for practical use proposes two approximation to his
\opt{} method, where candidates either are selected one by one (the
addition method), or eliminated one by one (the elimination method), in both
cases maximizing the total satisfication in each step.
These methods are described in detail in the following subsections.
\xfootnote{
\citet{Thiele} recommends using the elimination method, for the reason that
the addition method selects the elected sequentially, which might give the
first elected pretensions to be superior to their colleagues.
}
(Thiele was aware that the methods might give different results, and showed
this in
some of his examples; see Examples \ref{ETh}--\ref{ETh12} and
\refR{Rmono}.)

\begin{remark}
  Thiele's \opt{} method was reinvented by Simmons in 2001, 
under the name \emph{Proportional Approval Voting (PAV)}, see \cite{Kilgour}.
\end{remark}

\subsection{Thiele's addition method}\label{SSthiuadd}
Thiele's addition method is a ``greedy'' version of his \opt{} method,
where candidates are elected one by one, and for each seat, the candidate is
elected that maximizes the increase of the total satisfaction of the voters.
For a general satisfaction function $f$ satisfying \eqref{fw}, the
satisfication of a ballot $\gs$ containing a candidate $i$ is increased by
$w_{k+1}$ when $i$ is elected, where $k$ is the number of already elected on
this ballot.
This yields the following simple description, where we use $w_n=1/n$ to
obtain a proportional method. The general version is obtained by replacing
$1/(k+1)$ by some (arbitrary) numbers $w_{k+1}\ge0$.

\begin{metod}{Thiele's (addition) method}
Seats are given to candidates sequentially, until the desired number
have been elected.

%  \begin{itemize} \item 
For each seat, a ballot where $k\ge0$ names already have been
	elected is counted as $1/(k+1)$ votes for each remaining candidate on
	the ballot. The candidate with the largest vote count is elected.
%  \end{itemize}
\end{metod}

As said above, we usually call the method ``Thiele's method''.

\begin{remark}
  Thiele's addition method with the weak satisfaction function \eqref{fweak}
is studied in \citet{EJR} and is there called 
\emph{Greedy Approval Voting (GAV)}.
(In this case, each seat goes to the candidate with most votes among the
ballots that do not contain any already elected candidate.)
\end{remark}

\subsection{Thiele's elimination method}\label{SSthi-}

Thiele's elimination method works in the opposite direction. Weak candidates
are eliminated until only $\MM$ (the desired number) of them remain. Each
time, the 
candidate is eliminated that minimizes the decrease of the total satisfaction
of the voters caused by the elimination.
For a general satisfaction function $f$ satisfying \eqref{fw}, the
satisfication of a ballot $\gs$ with $k$ remaining candidates is decreased by
$w_{k}$ when one of them is eliminated.
This yields the following description, where again we use $w_n=1/n$ to
obtain a proportional method. The general version is obtained by replacing
$1/k$ by some (arbitrary) numbers $w_k\ge0$.

\begin{metod}{Thiele's elimination method}
Candidates are eliminated one by one, until only $\MM$ remain. The remaining
ones are elected.
%  \begin{itemize}  \item 
In each elimination step, a ballot where $k\ge1$ names remain
 is counted as $1/k$ votes for each remaining candidate on
	the ballot. The candidate with the smallest vote count is eliminated.
%  \end{itemize}
\end{metod}

Note the (superficial?) similarity with the Equal and Even Cumulative Voting in
\refApp{ACV}, 
but note that here the calculation is iterated, with the vote counts
changing as candidates are eliminated.

\begin{remark}
Thiele's elimination method has recently been reinvented 
(under the name \emph{Harmonic Weighting})
as a method for ordering alternatives for display for the electronic voting
system \emph{LiquidFeedback}  \cite{LiquidFeedback}.
\end{remark}

\subsection{An example}\label{SSthiu-ex}
We illustrate Thiele's unordered
methods by the same example as was used to illustrate
\phragmen's method in \refE{EPhr1894}.

\begin{example}[Thiele's \opt{}, addition and elimination methods]  
\label{EPhr1894-Th}
\quad

Unordered ballots. 3 seats. Thiele's three  methods.
\begin{val}
\item [1034]ABC
\item [519]PQR
\item [90]ABQ
\item [47]APQ
\end{val}

With Thiele's \opt{} method (\refSS{SSthiuopt}),
we note first that in this example,
we have dominations $\rmA>\rmB>\rmC$ and $\rmx Q>\rmx P>\rmx R$, in the
sense that, for example, replacing A by B or C always decreases the
satisfaction of a set of candidates; hence
it suffices to consider the four possible outcomes
ABC, ABQ, APQ, PQR instead of all $\binom63=20$ possible sets of three
candidates. 
These sets yield the satisfactions
\begin{val}
\item [ABC:]$1034\cdot\frac{11}6+519\cdot0+90\cdot\frac{3}2+47\cdot1=6233/3
\doteq2077.67$
\item [ABQ:]$1034\cdot\frac32+519\cdot1+90\cdot\frac{11}6+47\cdot\frac32
=4611/2=2305.5$
\item [APQ:]$1034\cdot1+519\cdot\frac32+90\cdot\frac{3}2+47\cdot\frac{11}6
=6101/3\doteq2033.67$
\item [PQR:]$1034\cdot0+519\cdot\frac{11}6+90\cdot1+47\cdot\frac{3}2
=1112$
\end{val}
Hence the largest satisfaction is given by ABQ, so ABQ are elected.

\medskip
With Thiele's addition method (\refSS{SSthiuadd}), 
the first seat goes to the candidate with the
largest number of votes, \ie, A (1171 votes, see \refE{EPhr1894}).

For the second seat, all ballots ABC, ABQ and APQ now are worth $1/2$ vote
each.
This gives the vote counts
\begin{val}
\item [B]$1034/2+90/2=562$
\item [C]$1034/2=517$
\item [P]$519+47/2=542.5$
\item [Q]$519+90/2+47/2=587.5$
\item [R]519.
\end{val}
Thus Q has the highest vote count and is elected to the second seat.

For the third seat, the ballots ABC and PQR have the value $1/2$, and the
ballots ABQ and APQ have the value $1/3$ each. Hence the vote counts are
\begin{val}
\item [B]$1034/2+90/3=547$
\item [C]$1034/2=517$
\item [P]$519/2+47/3\doteq275.17$
\item [R]$519/2=259.5$.
\end{val}
Hence B gets the third seat.

\medskip
With Thiele's elimination method (\refSS{SSthi-}),
in the first round, each ballot is counted as $1/3$ (since they all contain
3 names), so R has the smallest vote count (519/3=173) and is eliminated.

This increases the value of the ballots PQR to $1/2$ each for P and Q,
so the vote counts in the second round are
\begin{val}
\item [A]$1034/3+90/3+47/3=1171/3%=390\frac{1}3
\doteq390.33$
\item [B]$1034/3+90/3=1124/3 %374\frac{2}3
\doteq374.67$
\item [C]$1034/3%=344\frac{2}3
\doteq344.67$
\item [P]$519/2+47/3=1651/6\doteq275.17$
\item [Q]$519/2+90/3+47/3=1831/6\doteq305.17$.
\end{val}
P has the smallest vote count and is eliminated. 

In the third round, the ballots are worth $\frac13,1,\frac13,\frac12$, and
thus the vote counts are
\begin{val}
\item [A]$1034/3+90/3+47/2=2389/6\doteq398.17$
\item [B]$1034/3+90/3=1124/3 %374\frac{2}3
\doteq374.67$
\item [C]$1034/3 %=344\frac{2}3
\doteq344.67$
\item [Q]$519+90/3+47/2=1145/2=572.5$.
\end{val}
Thus C is eliminated, and the remaining ABQ are elected.

\medskip
We see that in this example, all three Thiele's methods yield the same
result, which also agrees with \phragmen's method (\refE{EPhr1894}).
\end{example}

\section{House monotonicity}\label{Shouse}

The \emph{Alabama paradox} occurs when the house size $\MM$ (\ie, the number
of seats) is increased, with the same set of ballots, and as a result
someone loses a seat; see \cite{BY} for the historical context. 
(See also \cite{SJalabama}.)
An election method is \emph{house monotone} if the Alabama paradox cannot
occur.
More formally, and taking into account the possibility of ties, an election
method is house monotone if whenever, for a given set of ballots, 
$\cS_\MM$ is a possible outcome of the
election for $\MM$ seats, then there exists a
possible outcome  
$\cS_{\MM+1}$ for $\MM+1$ seats such that $\cS_{\MM+1}\supset\cS_\MM$.

Obviously, any sequential method where seats are assigned one by one is
house monotone; this includes \phragmen's method and Thiele's addition method
in both the unordered versions above and the ordered versions
below. Similarly, Thiele's elimination method is house monotone.
However, Thiele's \opt{} method is not house monotone
(as noticed by \citet{Thiele}), see Examples \ref{ETh} and \ref{ETh12}.
\refE{ETh12}  can be generalized as follows to other satisfaction functions
$f$.

\begin{theorem}\label{Tmono}
  Let $f(n)$ be a satisfaction function with $f(0)=0$ and $f(1)=1$.
Then Thiele's \opt{} method is house monotone if and only if $f(n)=n$.
\end{theorem}
\begin{proof}
  First, if $f(n)=n$, then, as said above, Thiele's \opt{} method is
  equivalent to Approval Voting (\refApp{Aapproval}), which obviously is
  house monotone.

For the converse, consider an election with 4 votes, for AB, AC, B, C. For
$\MM=1$, a possible outcome is A, 
but if $w_2<1$, then the only outcome for $\MM=2$ is BC. 
In the opposite direction, in an election with the 2 votes A
and BC, a possible outcome for $\MM=1$ is A, 
but if $w_2>1$, then the only outcome for $\MM=2$ is BC. 
Hence, if the method is house monotone, then $w_2=1$.

More generally, for any $n\ge0$, consider the same two examples but add
$n$ further candidates D\subb1,\dots,D\subb{n} to every ballot. 
Then the only outcome for $\MM=n$ is D\subb1$\cdots$D\subb{n}.
For $\MM=n+1$, a possible outcome is D\subb1$\cdots$D\subb{n}A, but
for $\MM=n+2$, the outcome is D\subb1$\cdots$D\subb{n}BC in the first example if
$w_{n+2}< w_{n+1}$,  and in the second example if $w_{n+2}>w_{n+1}$.
Hence, if the method is house monotone, then $w_{n+2}=w_{n+1}$ for every
$n\ge0$, and thus $w_n=w_1=1$ for every $n\ge1$.
\end{proof}

The ties in the examples used in the proof can be avoided by modifying
the examples so that the different alternatives have slightly different
numbers of votes.

\begin{remark}\label{Rmono}
Thiele's \opt{} method and addition method coincide, for a general
satisfaction function $f$, if and only if the \opt{} method is house
monotone. (This is obvious, at least in the absence of ties, from the fact that
the addition method is a greedy version of the \opt{} method. The case with
ties is perhaps easiest seen by using \refT{Tmono}.)
Hence, the two methods coincide only for $f(n)=n$.

Similarly, 
Thiele's \opt{} method and elimination method coincide only for $f(n)=n$.

The examples in the proof above also show that the addition and elimination
methods do not coincide for any other function $f$.
This can be seen without further calculations
by noting that in these examples, for $\MM=n+1$, the addition
method and the \opt{} method coincides, while for $\MM=n+2$ (when only one
candidate is eliminated), the elimination method coincides with the \opt{}
method;
hence, if the addition method coincides with the elimination method, then the
\opt{} method coincides with the addition method in these examples, which
is impossible when it is not house monotone.

As said above, \citet{Thiele} noted that for the ``proportional''
satisfaction function \eqref{f},
the three methods may give different result, see \refE{ETh}.
He also gave an example showing the same for 
the ``weak'' satisfaction function \eqref{fweak}
\cite[Examples 5 and  6]{Thiele}.
\end{remark}
 
\section{Unordered ballots, principles}\label{Sunordered}
The unordered version of \phragmen's method and Thiele's methods 
satisfy the following general principles:
\begin{metod}{Principles for unordered versions of \phragmen's and thiele's methods}\quad
\begin{PXenumerate}{U}
\item \label{Pu}
The ballots are unordered, so that each voter lists a
number of candidates, but their order is ignored.
\item \label{Pufull}
In the allocation of a seat, 
each ballot  is counted fully for every  candidate on it
(ignoring the ones that already have been elected).
The value of a ballot is reduced (in different ways for the different
methods) when someone on it is elected, but the effective value of a ballot
does not depend on the number of unelected candidates on it.
\end{PXenumerate}  
\end{metod}
As a consequence of \ref{Pufull}, the methods have also the following
property,
which reduces the need for tactical voting.
\begin{PXenumerateq}{U}
  \item \label{Puadd}
\emph{%
If any number  of candidates are added to a ballot, but
none of them is elected, then the result of the election is not affected.
}
\end{PXenumerateq}

Note that \ref{Pufull} and \ref{Puadd} hold also for Approval Voting (which
can be seen as a special case of Thiele's general \opt{} method, see
\refSS{SSthiuopt}), but not for some other election methods with
unordered ballots, for example Equal and Even Cumulative Voting
(\refApp{ACV}), where a vote
is split between the names on it, and the value of a ballot for each
candidate on it decreases if more names are added.

\section{Unordered ballots and decapitation}\label{Sdecap}

Unordered ballots have some problems that are more or less independent of
the election method used to distribute the seats, and in particular apply to
both \phragmen's and Thiele's methods.

If there are organized parties, a party usually fields more candidates
than the number that will win seats, for example to be on the safe side since 
the result cannot be predicted with certainty in advance.
If all voters of the party are loyal to the party and vote for the party list, 
the result will be that all candidates on the party list will tie, and 
the ones that are elected will be chosen by lot from the party list.

The party can avoid this random selection, 
and for example arrange so that the party leader
gets a seat (if the party gets any seat at all), by organizing a scheme where
a small group of loyal members vote for specially selected subsets of the
party list, with the result that the party's candidates 
will get slightly different number of votes, in the order 
favoured by the party organization. However, this will still be sensitive to
a coup from a rather small minority in the party, that might 
secretly agree not to
vote for, say, the party leader. To protect against such coups, the party
can give instructions to vote on specific sets of candidates
to a larger number of voters, but it is certainly a serious drawback of an
election system to depend on complicated systems of tactical voting.
(This might also, depending on the system, 
make it more difficult for the party in the competition with other parties,
and the party might risk to get fewer seats.) 

Moreover, even if party $A$ is well-disciplined and there is no internal
opposition, it is possible that a small group from another party, say $B$, will 
cast their votes on some less prominent and perhaps less able candidates
from party $A$, instead of voting for their own party. Of course, that might
give party $A$ another seat, but with a small number of tactical votes, this
risk is small. On the other hand, there is a reasonable chance that the
extra votes will elect the chosen candidates from $A$ instead of the ones
preferred by the party. 
This tactical manoeuvre, called \emph{decapitation}, can thus prevent \eg{}
the party leader of an otherwise successful party to be elected.

Decapitation does not necessarily occur because of sinister tactics by some
groups; it can also occur by mistake, when too many voters believe that
their primary candidate is safe and therefore also vote for others.
One such  situation  is discussed in
\refE{Erank}. (There in connection with further complications caused by
an  election system with additional rules.)

I do not know to what extent  such decapitation  occurred in practice, but 
at least the unintentional type in \refE{Erank} occurred
\cite[p.~8]{bet1913}, and decapitation was
something that was feared and much discussed in the discussions about electoral
reform in Sweden around 1900;
see for example \cite[p.~25]{Cassel} and \cite[pp.~13, 20, 33, 35]{bet1913}.
%(For example in the parliament's decision on an election reform 1907.)
%\kolla % \emph{Riksdagens skrifvelse N:o 147, 1907}.
%Bihang till Riksdagens Protokoll 1907,
%10 Saml., 1 Afd., 1 Band.
As a result, versions using ordered ballots were developed
of both \phragmen's method and Thiele's addition
method, as described in the following sections.
\xfootnote{
  \citet{Phragmen1895} makes another suggestion to avoid the problem of
  decapitation (saying that it is just one
  possibility among many), while keeping unordered ballots. 
In this proposal, parties could  register ordered
  party lists. 
The ballots would still be regarded as unordered, but
a vote on a party list, say ABCDE, would in addition to the vote ABCDE also
be regarded as (for example) $\frac{1}{10}$ extra vote on shorter lists, in
this case ABCD, ABC, AB and A, split between them with $\frac{1}{40}$ extra
vote each.
It seems that everyone ignored this suggestion,
including  \phragmen{} himself in later writings, and it seems for
good reasons.
}
(We do not know of any ordered versions of 
Thiele's \opt{} and elimination methods,
except that \emph{Bottoms-up}, see \refApp{Abottom}, perhaps might be seen as
an ordered version of his elimination method.)

\section{Ordered ballots, principles}\label{Sordered}

The ordered versions of \phragmen's method and Thiele's (addition) method 
differ from  the
unordered versions in Sections \ref{SPhru} and \ref{Sthiu} in two ways:
\begin{metod}{Principles for ordered versions}
\quad\begin{PXenumerate}{O}
\item \label{Po}
The ballots are ordered, so that each voter lists a
number of candidates in order.
\item \label{Potop}
In the allocation of a seat, 
each ballot  is counted only for one candidate, 
\viz{} the first candidate on it that is not already elected.
\end{PXenumerate}
\end{metod}
This is detailed in the following sections.

Note that \ref{Potop} implies the following property, similar to 
\ref{Puadd} for the unordered versions.
  \begin{PXenumerateq}{O}
\item \label{Poadd}
\emph{%
If any number  of candidates are added after the existing names on a
ballot, then the result of the election is not affected unless all existing
candidates are elected.
}
\end{PXenumerateq}
Thus, a voter can add names after his or her favourite candidates without
risking to hurt their chances.

\section{\phragmen's ordered method}\label{SPhro}

\phragmen's ordered method is thus obtained 
by modifying the
version in \refS{SPhru} using the principles \ref{Po}--\ref{Potop}
in \refS{Sordered}.
This method was proposed in 1913 by a Royal
Commission on the Proportional Election Method \cite{bet1913}
(see \refApp{Ahistory});
\phragmen{} was one of the members of the commission, so it is natural to
guess that he developed also this version.
\xfootnote{
\citet{Phragmen1893} had already in 1893 used ordered ballots in one example
when he discussed
a version of STV.}
(A similar method had been proposed by \citet{Tenow1910} in 1910.)

The method has been used in Swedish elections for the distribution of seats
within parties since 1921,
although it now plays only a secondary role;
see Appendices \ref{Ahistory}
and
\ref{APh-vallag}.

\subsection{First formulation}\label{SPhro1}
The ordered version can thus be defined as follows, \cf{} \refS{SPhru1}.
A detailed example is given in \refSS{SSPhro-ex}; %, see \refE{EPhr1894-o}. 
further examples are given in \refS{Sex}.

\begin{metod}{\phragmen's ordered method, formulation 1}

Assume that each ballot has some \emph{voting power} $t$; this number is
the same for all ballots and will be determined later. 
A candidate needs total voting power $1$ in order to be elected.
%The voting power of a ballot may
%be divided among several of the candidates on the ballot. 
During the
procedure described below, some of the voting power of a ballot may be already
assigned to already elected candidates; the remaining voting power of the
ballot
is free, and is used by the ballot's current
top candidate, \ie, the first candidate on the ballot that is not already
elected.

The seats are distributed one by one.

For each seat, 
each ballot is counted for its current top candidate.
%, \ie, the first candidate on the ballot that is not already elected. 
(If all candidates on a ballot are elected, the ballot is ignored.)
The top candidate receives the free voting power of the ballot.
(I.e., the full voting power except for the voting power
already assigned from that ballot to candidates already elected.)
For each candidate,
the ballot voting power $t$ %for each ballot 
that would give the candidate voting power $1$ is
computed, and the candidate requiring the smallest voting power $t$ is
elected. All free (\ie{} unassigned)
voting power on the ballots that were counted for the elected
candidate is assigned to that candidate, and these assignments remain fixed
throughout the election. 

The computations are then repeated for the next seat, with the remaining
candidates, and so on.
\end{metod}

\begin{remark}
  \label{Rtimeo}
As for the unordered version in \refS{SPhru}, we can think of the voting
power of each ballot increasing continuously with time, see \refR{Rtime}.
Again, the voting power available to each candidate increases with time, and
when some candidate reaches voting power 1, that candidate is elected.
(A difference from the unordered version is that the available voting powers
of the other candidates do not change when someone is elected; in
particular, the voting powers of the candidates never decrease.)
\end{remark}

\subsection{Second formulation}\label{SPhro2}
The same calculations as in \refS{SPhru} show that the formulation above is
equivalent to the following more explicit algorithm,
\cf{} \refS{SPhru1}:

\begin{metod}{\phragmen's ordered method,  formulation 2}
Seats are given to candidates sequentially, until the desired number
have been elected. 
During the process, each type of ballot, \ie,
each group of identical ballots, 
is given a 
\emph{place number}, which is a rational non-negative number that can be
interpreted as the (fractional) number of seats elected so far by these
ballots; the sum of the place numbers is always equal to the number of seats
already allocated. 
The place numbers are determined recursively and the seats
are allocated by the following rules:
\begin{romenumerate}
\item 
Initially all place numbers are $0$.
\item \label{phro2b}
Suppose that $n\ge0$ seats have been allocated.
Let $q_\ga$ denote the place number for the ballots with a list $\ga$
of candidates;
thus $\sum_\ga{q_\ga}=n$.
For each candidate $i$, let $A_i$ be the set of lists $\ga$ such that $i$ is
the first element of $\ga$ if we ignore candidates already elected.
The total number of votes counted for candidate $i$ in this step is 
$\sum_{\ga\in A_i} v_\ga$, and 
the total place number of the ballots counted for candidate $i$ is
$\sum_{\ga\in A_i} q_\ga$. The \jfr{} for candidate $i$ is defined as
\begin{equation}\label{wio}
  W_i:=\frac{\sum_{\ga\in A_i} v_\ga}{1+\sum_{\ga\in A_i} q_\ga},
\end{equation}
\ie, the total number of votes counted for the candidate divided by 
$1$ $+$ their total
place number.

\item 
The next seat is given to the candidate $i$ that has the largest $W_i$.
\tie

\item \label{phro2d}
Furthermore, if candidate $i$ gets the next seat, 
then the place numbers are updated
for all lists $\ga$ that participated in the election, \ie, the lists
$\ga\in A_i$. For such $\ga$, the new place number is
\begin{equation}\label{q'o}
 q'_\ga:=\frac{v_\ga}{W_i}
=
\Bigpar{1+\sum_{\ga\in A_i} q_\ga}\frac{v_\ga}{\sum_{\ga\in A_i} v_\ga}.
\end{equation}  
For  $\ga\notin A_i$, $q'_\ga:=q_\ga$.
\end{romenumerate}
Steps \ref{phro2b}--\ref{phro2d} are repeated as many times as desired.
\end{metod}

Indeed, the commission report \cite{bet1913} that proposed the method in
1913 introduced it in
this form; the committee first discussed the current method and various
proposals to improve it,
and showed by examples that they all were unsatisfactory, and then presented the
method above gradually, through a sequence of examples of increasing
complexity,  
as the correct generalization of D'Hondt's method to
ordered lists.
\xfootnote{The commission report also noted that the method was essentially
  the same as the method for unordered ballots presented in 1894 by
one of the members of the commission, 
  \citet{Phragmen1894,Phragmen1895}.}

\begin{remark}\label{RWto}
  As in \refS{SPhru}, the connection between the two formulations above is
  that $W_i=1/t_i$, where $t_i$ is the voting power per ballot that would
  give $i$ voting power 1 in the current round.
In particular, \eqref{Wntn} still holds for the winning voting power and
\jfr{} in each round.
\end{remark}

\subsection{Third formulation}\label{SPhro3}
In practical applications of the method, 
it is not necessary to keep track
of and calculate place numbers for each individual type of ballot; since
each ballot counts only for its current top name, 
it suffices to group ballots according to their current top name.
This simplification 
was introduced already in the commission report \cite{bet1913} where the
method was introduced, and it
is used 
in the Swedish \EA, where this method has played a part (within
parties) since 1921,
see \refApp{Ahistory}, and where the method is defined in the following form (in
my words; for the actual text of the \EA, see \refApp{APh-vallag}):

\begin{metod}{\phragmen's ordered method,  formulation 3}
Seats are given to candidates sequentially, until the desired number
have been elected.
During the process,
the ballots are organized in groups; 
the ballots in each group have the same top candidate, ignoring candidates
already elected, and the group is valid for that candidate.
Each group is also given a 
\emph{place number}, which is a rational non-negative number that can be
interpreted as the (fractional) number of seats elected so far by these
ballots; the sum of the place numbers of the different groups
is always equal to the number of seats
already allocated. 

The groups are created, 
their place numbers are determined and the seats
are allocated by the following rules:
\begin{romenumerate}
\item 
Initially, the ballots are divided into groups according to their first
name.
Each group is given place number\/ $0$.

\item \label{phro3b}
Suppose that $n\ge0$ seats have been allocated.
Denote the current set of groups of ballots by $\gG$, and for each group
$\gam\in\gG$, let $v_\gam$ be its number of votes and  $q_\gam$ its place
number; furthermore, let $i_\gam$ be the candidate that the group is valid
for.
%, \ie, the top candidate on the ballots in the group when already elected
%candidates are ignored.
For each candidate $i$,
let $\gG_i:=\set{\gam\in\gG:i_\gam=i}$, the set of groups valid for
that candidate,
and 
define the \jfr{}  by
\begin{equation}\label{wig}
  W_i:=\frac{\sum_{\gam\in \gG_i} v_\gam}{1+\sum_{\gam\in \gG_i} q_\gam},
\end{equation}
\ie, the total number of votes currently valid for the candidate divided by\/ 
$1$ $+$ the total place number of the corresponding groups.

\item 
The next seat is given to the candidate $i$ that has the largest $W_i$.
\tie

\item \label{phro3d}
If candidate $i$ gets the next seat, then the groups valid for $i$ are
merged together and then divided into new groups according to the new top
candidate on the ballots, \ie, 
the first candidate after $i$ that is not already elected. (Ballots were all
candidates are elected are put aside and ignored in the sequel.)
Each new group $\gam$
is given a place number
\begin{equation}\label{q'g}
 q_\gam:=\frac{v_\gam}{W_i}
%=\Bigpar{1+\sum_{\gam\in \gG_i} q_\gam}
% \frac{v_\gam}{\sum_{\gam\in \gG_i} v_\gam}
.
\end{equation}  
All other groups remain unchanged, and keep their place number.
\end{romenumerate}
Steps \ref{phro3b}--\ref{phro3d} are repeated as many times as desired.
\end{metod}

It is easily seen that this version is equivalent to the version in
\refS{SPhro2}; each group %$\gam$
contains all ballots of some set %$A_\gam$ 
of types (each given by a list $\ga$ of names), and the
number of votes and the place number for the group are just the sums of the
number of votes and of the place numbers for these types.

\subsection{An example}\label{SSPhro-ex}
We illustrate Phragmén's ordered
method by his unordered example  in \refE{EPhr1894}, but now interpreting the
ballots as ordered (with the first candidate first).

\begin{example}[\phragmen's ordered method]  \label{EPhr1894-o}
\quad

Ordered ballots. 3 seats. \phragmen's method.
\begin{val}
\item [1034]ABC
\item [519]PQR
\item [90]ABQ
\item [47]APQ
\end{val}
We use the formulation of \phragmen's ordered method in \refSS{SPhro2}.

Since each ballot is counted only for its first name,
we have in the first round 1171 votes for A and 519 votes for P, so A is
elected. 

This gives a place number (= assigned voting power = load) $1/1171$ to each
ballot ABC, ABQ or APQ, and thus, in the second round,
the four groups of ballots have %(collectively) 
place numbers
$\frac{1034}{1171},0,\frac{90}{1171},\frac{47}{1171}$, see \eqref{q'o}.
(If we use the formulation in \refSS{SPhro3}, we group the ballots ABC and
ABQ together; this group has 1124 votes and place number $\frac{1124}{1171}$.)

In the second round, the top names are B and P. They have 1124 and 566 votes
with total place numbers $\frac{1124}{1171}\doteq0.9599$ and
$\frac{47}{1171}\doteq0.0401$
(with sum 1),
respectively, 
and thus \jfr{s}
$$
W_\rmB=\frac{1124}{1+\frac{1124}{1171}}=\frac{1316204}{2295}\doteq573.51
$$
and
$$
W_\rmx P=\frac{566}{1+\frac{47}{1171}}=\frac{331393}{609}\doteq544.16
$$
Thus B is elected to the second seat.

The place numbers for the four groups of ballots are now
$\frac{1034}{W_\rmB}=\frac{1186515}{658102}\doteq1.8029$,
$0$, $\frac{90}{W_\rmB}=\frac{103275}{658102}\doteq0.1569$,
$\frac{47}{1171}\doteq0.0401$ (with sum 2).

In the third round, the top names on the four types of ballots are C, P, Q,
P,
and their \jfr{s} are
\begin{align*}
  W_\rmC&=\frac{1034}{1+\frac{1186515}{658102}}
=\frac {680477468}{1844617}
\doteq 368.90
\\
W_\rmx P&=\frac{566}{1+\frac{47}{1171}}=\frac{331393}{609}
\doteq 544.16
\\
W_\rmx Q&=\frac{90}{1+\frac{103275}{658102}}=
\frac {59229180}{761377}
\doteq 77.79.
\end{align*}
Hence, P is elected to the third seat.

Elected: ABP.
\end{example}

\section{Thiele's ordered method}\label{Sthio}

Thiele's ordered method is obtained by modifying
his unordered addition method in \refS{SSthiuadd} 
using the principles in \refS{Sordered}.

\begin{metod}{Thiele's ordered method}
Seats are given to candidates sequentially, until the desired number
have been elected.

For each seat, a 
ballot 
is counted for the ballot's current top
	candidate, \ie, the first  candidate on the ballot that has not already
	been elected; if this is the $k$-th name on the ballot (so the preceding
    $k-1$ have been elected), then the ballot is counted as $1/k$ vote.
(Ballots where all candidates have been elected are	ignored.)
The candidate with the largest vote count is elected.
\end{metod}
A detailed example is given in \refSS{SSthio-ex}; % \refE{EPhr1894-Tho}. 
further examples are given in \refS{Sex}.

\begin{remark}
  This is based on Thiele's proportional satisfaction function \eqref{f}.
Of course, just as in \refS{Sthiu}, other functions $f$ may be used. 
We leave this to the reader to explore.
\end{remark}

The ordered version of Thiele's method was proposed in the Swedish 
parliament in 1912 (by Nilson in \"Orebro),
as an improvement of the then used unordered Thiele's method for
distribution of seats within parties 
(see \refApp{Ahistory}). 
It was not adopted, but it was later adopted, still within parties, for
elections inside city and county councils, 
and it is still used for that purpose, see \refApp{Athi-lag}.
%bet1913: p. 26 ff,  1912 Nilson. Ex 14-17

\subsection{An example}\label{SSthio-ex}
We illustrate also
Thiele's ordered method 
with the same example as in Examples \ref{SSPhru-ex}, \ref{SSthiu-ex} and
\ref{SSPhro-ex}. 

\begin{example}[Thiele's ordered method]  \label{EPhr1894-Tho}
\quad

Ordered ballots. 3 seats. Thiele's method.
\begin{val}
\item [1034]ABC
\item [519]PQR
\item [90]ABQ
\item [47]APQ
\end{val}

Since each ballot is counted only for its first name, we have in the first
round
(as in \refE{EPhr1894-o}) 
1171 votes for A and 519 votes for P, so A is elected.

For the second seat, all ballots ABC, ABQ and APQ now have a value $1/2$
each.
This gives the vote counts
\begin{val}
\item [B]$1034/2+90/2=562$
\item [P]$519+47/2=542.5$
\end{val}
Thus B has the highest vote count and is elected to the second seat.

For the third seat, the four types of ballots have the values
$\frac{1}3,1,\frac{1}3,\frac12$ each. Hence the vote counts are
\begin{val}
\item [C]$1034/3\doteq 344.67$
\item [P]$519+47/2=542.5$
\item [Q]$90/3=30$.
\end{val}
Hence P gets the third seat.

Elected: ABP.
(As with \phragmen's method in \refE{EPhr1894-o}.)
\end{example}

\section{\phragmen's and Thiele's methods generalize
  D'Hondt's}\label{Sparty}

The election methods considered in this paper are blind to the possible
existence of parties.
Nevertheless, it is important to see how they perform in the presence of
parties. We therefore consider the following special case:
\begin{metod}{party list case}
  There is a number of
parties such that:
\begin{itemizex}
\item Each party has a single list of candidates. (Ordered or unordered,
  depending on the election method.)
\item The lists are disjoint. (I.e., no person runs for more than one
  party.) 
\item Every voter votes for one of the party lists.
\end{itemizex}
\end{metod}

In this case, \phragmen's and Thiele's methods reduce to D'Hondt's method
(described in \refApp{ADHondt}). 
This was explicitly stated (and shown) by
\citet[pp.~38--40]{Phragmen1895} for his method;
\citet[p.~425]{Thiele} states explicitly only the special case when the
fraction  of votes 
for each party is an integer times $1/\MM$, so that a perfectly proportional
representation is possible, but his proof really shows the general result.

\begin{theorem}[\phragmen, Thiele]\label{TDHondt}
  In the party list case just described, \phragmen's unordered and ordered
  methods and Thiele's three unordered methods and his ordered method 
(assuming the proportional satisfaction function \eqref{f})
all yield the same distribution of seats between parties as D'Hondt's method.
\end{theorem}

\begin{proof}
More or less obvious, and left to the reader.
Note that for the unordered methods, all   candidates from the same party
will tie, 
and that for \phragmen's methods, the place number of a list equals the
number of seats already given to that list (since lists are disjoint).  
\end{proof}

\begin{remark}
For the unordered methods,  
there is a minor difference with D'Hondt's method in the case of a tie, if
  that is resolved by lot. In D'Hondt's method, a tie will be resolved by
  choosing one of the tying parties at random, while in \phragmen's and
  Thiele's methods, one of the tying candidates is chosen at random, and thus
  a party with many candidates will have a larger chance of getting the seat.
\end{remark}

\begin{remark}
  For the unordered methods, by the assumptions of the party list case, all
  candidates from the same party will tie, so which of them that will be
  elected to the seats won by the party will typically be determined by lot.
(See also \refS{Sdecap}.)
For the ordered methods, the seats of a party will be assigned in the order
the candidates appear on the list.
\end{remark}

\begin{remark}
  The theorem extends to Thiele's methods with a different satisfation
function $f(n)=\sum_{k=1}^nw_k$; in the party list case they all give the same
result as the 
divisor method with divisors $1/w_n$ (see \refApp{Alist}).
See \cite{BrillLS}, where also further related results are given.
\end{remark}

\section{Electing a single person}\label{S1}

Another simple special case is when $\MM=1$, so that only one person is to be
elected.
(A single-member constituency.)
Of course, in this case, there are simpler methods, but it is of interest to
see how the methods above perform.

\begin{theorem}\label{T1}
In the case $\MM=1$, the following holds.
  \begin{romenumerate}
  \item For unordered ballots, \phragmen's method and Thiele's \opt{} and
	addition method elect the candidate with the largest number of
	votes. I.e., they are equivalent to Approval Voting
 (\refApp{Aapproval}).
  \item 
For ordered ballots, \phragmen's and Thiele's methods elect the candidate
with the largest number of first votes. I.e., they are equivalent to
Single-Member  Plurality (First-Past-The-Post, see  \refApp{ABV}). 
(Names after the first name on a ballot are thus ignored.)
  \end{romenumerate}
\end{theorem}
\begin{proof}
  Obvious.
\end{proof}

\begin{remark}
  Note that Thiele's elimination method may give a different result, see
  \refE{ETh12}. 
\end{remark}

\section{Examples}\label{Sex}

We give several examples to illustrate and compare the different methods.
(Many of them are constructed by various authors to illustrate weaknesses of
some method.)
See also Examples 
\ref{EPhr1894},
\ref{EPhr1894-Th},
\ref{EPhr1894-o}, and
\ref{EPhr1894-Tho}.

%We begin with four examples giving detailed calculations. 
It is seen that, except in very simple cases, 
\phragmen's methods lead to rational
numbers with large denominators and that they thus
are not suitable for hand calculations;
this is not a problem with computer assistance.
\xfootnote{In the use of the method in
Sweden (\refApp{APh-vallag}), all numbers are calculated to two decimal
places (rounded down), so exact rational arithmetic is not used.
}\,
\xfootnote{For computational aspects, see also \cite{BrillEtAl}.}

In the examples below, ABC denotes a ballot with the names A, B and C; in the
case of ordered ballots, the names are given in order, with A the first name.

In several of the examples, there are obvious ties because of symmetries. We
then usually give only one, typical, outcome, sometimes without comment.

\begin{example}[\phragmen's unordered method]\label{E1893a}
\citet[pp.~19, 28, 44]{Phragmen1895} discusses the following example, taken
from the parliamentary election 1893 (Stockholm's second constituency). 
There were actually 2624 valid ballots, but for simplicity \phragmen{}
ignores 73 ballots of various types with less than 10 ballots each.
\xfootnote{The five remaining types, shown below, were all promoted by various
  political organizations; note that there are considerable overlaps.}

Unordered ballots. 5 seats. \phragmen's method.
\xfootnote{
T = Themptander, V = von Friesen, F = Fredholm, P = Palme, 
E = Eklund, L = Lovén, N = Nordenski{\"o}ld, B = Backman, Q = Pettersson,
J = Johansson, W = Wikman.
}\,
\xfootnote{The actual election used the Block Vote, and the elected were,
  counting also the ballots ignored here, 
TVFPE with 1995, 1470, 1458, 1316 and 1313 votes, respectively; the next
candidate had 1301 votes \cite[p.~29]{SCB1891-1893}.
} 
  \begin{val}
  \item[1233] TVFPE  %liberala valmansforeningen
  \item[585] TLNBQ   %de moderate
  \item [124] TVFLN  %frihandlarna
  \item [547] LNBQJ  %fosterlandska forbundet
  \item [62] VFPEW   %allmanna rosttratten
  \end{val}
\phragmen{} uses this example to discuss several conceivable methods, and
gives detailed calculations for his method which results in the election of, in
order (after resolving ties arbitrarily) TLVNF. 
\end{example}

\begin{example}[\phragmen's and Thiele's unordered methods]\label{E1893b}
  \citet[p.~49]{Phragmen1895} gives also another example from the
parliamentary election 1893 (Stockholm's fifth constituency). 
There were actually 1520 valid votes, of which 57 (of types with less than 10
ballots each) are ignored in the example.

Unordered ballots. 5 seats. \phragmen's and Thiele's methods.
\xfootnote{
H = H\"ojer, G = G. Eriksson, E = P.J.M. Eriksson, B = Bergstr\"om, O = Olsson,
X = Branting, Q = \"Ostberg, P = Palm, Y = Hellgren, Z = Billmansson, L =
Lindvall.
}\,
\xfootnote{The actually elected, using the Block Vote and
  counting also the ballots ignored here, were
HGEBO, with 1454, 1421, 1055, 1053, 715 votes, respectively, with 469 votes
for the 
next candidate
\cite[p.~29]{SCB1891-1893}.
}
\begin{val}
\item [680] HEOBG  %liberala valmansforeningen
\item [341] HEXBG  %socialisterna
\item [322] HYPQG  %de moderate
\item [49] HPXQG   %blandad moderat-socialistisk
\item [47] ZYPQL   %Nya Dagligt Allehandas
\item [14] HEOBX   %blandad socialistisk-liberal
\item [10] HPOQG   %Olssonsk spranglista
\end{val}
\phragmen{} shows that his method would elect HGEBQ (or HGEBP; the last seat
is tied, as is the order between E and B, but both are elected in any case).

\citet[Examples 1 and 2]{Thiele} uses the same example 
(omitting Z and L to show that all ballots do not have to have the same
number of names, and perhaps also to shorten the calculations) 
and shows that his addition and elimination methods both also elect HGEBQ (or
HGEBP).

Thiele's \opt{} method yields also the same result.
(\citet{Thiele} did not calculate this, presumably because
even with his omission of two candidates, there
are $\binom 95=126$ sets to consider.)
\end{example}

\begin{example}[Thiele's three unordered methods]
  \label{ETh}
The following example by \citet[Examples 3 and 4]{Thiele}, here slightly
modified to avoid ties, shows that the three methods by Thiele
(see \refS{Sthiu}) can give different results.

Unordered ballots.
2 seats. %to be elected.
  \begin{val}
  \item[960] ACD
  \item [3000]BCD
  \item [520]BC
  \item [1620]AB
  \item [1081]AD
  \item [1240]AC
  \item [360]BD
  \item [360]D
  \item [120]C
  \item [60]B
%  \item [1]A
  \end{val}
Thiele's \opt{} method elects AB, his addition method elects CA and his
elimination method elects BD.

This example also shows that the \opt{} method is not house monotone, see
\refS{Shouse}; if only 1 is elected, the \opt{} method elects C (as the
addition method does; the elimination method elects D).

\phragmen's method  elects CA, and thus gives the same result as
Thiele's addition method.

(With 3 seats, Thiele's \opt{} method and elimination method elect BCD,
while  the addition method and
\phragmen's method elect CAD.)
\end{example}

\begin{example}[Thiele's three unordered methods]\label{ETh12}
\citet{Tenow1912} gave a simpler example
showing that Thiele's three unordered methods are different
(and thus also that Thiele's \opt{} method is not house monotone).

Unordered ballots.
  \begin{val}
  \item [12]AB
  \item [12]AC
  \item [10]B
  \item [10]C
%\item [\strut{666}]
  \end{val}
For 1 seat,
Thiele's \opt{} method elects A, but for 2 seats, the method elects BC.

Thiele's addition method elects A for 1 seat, and AB or AC for 2 seats.

Thiele's elimination method
eliminates first A; hence it elects BC for 2 seats, and B or C for 1 seat.
\end{example}

In the examples above, \phragmen's method yields the same result as Thiele's
(addition) method.
The following examples are constructed to show differences between the
two methods (and to show weaknesses of Thiele's methods).
% (and at least partly trying to find faults with Thiele's method). 
The first three examples are rather similar, and have two parties
(or factions) that agree on one (or several) common candidates.

\begin{example}[\phragmen's and Thiele's unordered or ordered methods]
  \label{EPhr1899}
\citet{Phragmen1899} compares his method and Thiele's (addition) method 
for unordered ballots
(and argues in favour of his own) using the following example
(here slightly modified to avoid ties).
The example works also for ordered ballots, with the given order.

Unordered or ordered ballots.
\begin{val}
\item [2001] AB$_1$B$_2$B\subb3\dots
\item [1000] A$\rmC_1\rmC_2\rmC_3\dots$
\end{val}
This could be two parties B and C, and a highly respected independent
candidate A, or two factions that both accept the same leader.

In an example like this, with two types of ballots that have one common
candidate, which in the ordered case comes first,
it is obvious that both \phragmen's and Thiele's methods first
elect the common candidate A. 

After that, 
as is easily seen from the formulation in Sections \ref{SPhru1} and
\ref{SPhro1},
\phragmen's method treats the
remaining parts of the two
lists as two disjoint party lists, and thus (see
\refT{TDHondt}) the remaining seats are distributed as by D'Hondt's method.
Hence, 
%if for definiteness ties in this example are resolved in favour of B 
%(or the number of votes for the first list is increased to 2001),
\phragmen's method will elect candidates in the order
AB\subb1B\subb2C\subb1B\subb3B\subb4C\subb2B\subb5B\subb6C\subb3\dots.
\xfootnote{In the unordered case, all B's are tied, as are all C's, so the
  order among them is arbitrary.}

Thiele's method, on the other hand, reduces the votes of both lists by the
same factor after the election of A.
As \citet{Tenow1912} comments for a related example,
Thiele's method here treats A as two different
persons; if the two  lists had had two different persons A\subb1 and
A\subb2, then, when both were elected, 
Thiele's method would have reduced the votes in the same way.
(In that case, we would have had disjoint party lists and by \refT{TDHondt}
the elected would have been
A\subb1B\subb1A\subb2B\subb2B\subb3C\subb1B\subb4B\subb5C\subb2B\subb6B\subb
7C\subb3\dots.)
Consequently, the candidates are elected in order
AB\subb1B\subb2B\subb3C\subb1B\subb4B\subb5C\subb2B\subb6B\subb7C\subb3\dots

We thus see that Thiele's method favours the larger party. In particular,
with 4 seats, \phragmen's method elects AB\subb1B\subb2C\subb1 and Thiele's
method elects AB\subb1B\subb2B\subb3. 
\end{example}

\begin{example}[\phragmen's and Thiele's unordered or ordered methods]
  \label{EPhr1899C}
\citet{Phragmen1899} considers also  (with unordered ballots)
the following modification of the
previous example.

Unordered or ordered ballots.
\begin{val}
\item [2000] AB\subb1B\subb2B\subb3\dots
\item [1000] AC\subb1C\subb2C\subb3\dots
\item [550] C\subb1C\subb2C\subb3\dots
\end{val}
A calculation shows that \phragmen's method will elect candidates
in the order
AB\subb1C\subb1B\subb2C\subb2B\subb3C\subb3B\subb4C\subb4B\subb5B\subb6\dots,
while Thiele's method will elect in the order 
AC\subb1B\subb1B\subb2C\subb2B\subb3C\subb3B\subb4C\subb4B\subb5B\subb6\dots.
(The next seat that differs is no.~25, which is a tie for Thiele's method.) 

It is not clear that one of these results is ``better'' or ``more fair''
than the other, but
\phragmen{} points out that both methods generally give more seats to B than
to C (as they should, because B got more votes), 
but if there are 2 seats, then Thiele's
method will give these to AC; \phragmen{} concludes that at least it is not an
advantage of Thiele's method that for 2 seats it puts C ahead of B, while
for larger numbers of seats it puts B ahead of C.
\xfootnote{I do not know whether the same can happen with \phragmen's
  method in another situation.}

In this simple example, it is possible to analyse exactly what happens also
for a large number of seats (and candidates, \cf{} \refSS{SSMora-party} and
\cite{MoraO}).
Let $b=b(c)$ be the number elected from the B party before $c$ candidates
are elected from the C party, for $c\ge1$.

With \phragmen's method, the election of A gives a load $2/3$ to the first
group of ballots, and $1/3$ to the second. Thus, when C\sss{c} is elected,
the load on each ballot in the first group is $(b(c)+\frac{2}3)/2000$ and 
the load on each ballot in one of the two last groups is
$(c+\frac{1}3)/1550$.
Consequently, $b=b(c)$ is the largest integer such that
\begin{equation}
\frac{b+\frac{2}3}{2000}
\le
\frac{c+\frac{1}3}{1550},
\end{equation}
i.e.
\begin{equation}\label{ephr1899c-phr}
  b(c)=\lrfloor{\frac{2000}{1550}\Bigpar{c+\frac{1}{3}}-\frac{2}{3}}
=\lrfloor{\frac{40}{31}c-\frac{22}{93}}.
\end{equation}

With Thiele's method, C\sss{c} gets elected with $1000/(c+1)+550/c$ votes,
while B\sss{b} was elected with $2000/(b+1)$ votes. Thus $b=b(c)$ is the
largest integer such that
\begin{equation}
  \frac{2000}{b+1} \ge \frac{1000}{c+1}+\frac{550}{c},
\end{equation}
and thus
\begin{equation}\label{ephr1899c-th}
  b(c)=\lrfloor{\frac{2000}{\frac{1000}{c+1}+\frac{550}{c}}-1}
=\lrfloor{\frac{40}{31}c-\frac{161}{961}-\frac{8800}{961(31c+11)}}.
\end{equation}
The claims above are easily verified from \eqref{ephr1899c-phr} and
\eqref{ephr1899c-th}. 
Note that $\frac{22}{93}\doteq0.2366 >\frac{161}{961}\doteq0.1675$.
Furthermore, 
since the last fraction in \eqref{ephr1899c-th} tends to 0 as $c\to\infty$, 
it follows
that except for the first seats, 
both methods elect B and C in sequences that are
periodic with period 71, with 40 B and 31 C.
Moreover, it is easily verified that
for $c>1$,
\eqref{ephr1899c-phr} and \eqref{ephr1899c-th} differ if and only if 
$c\equiv11$ or $18\pmod{31}$, and hence the two methods differ
if and only if $\MM=2$ or $\MM\equiv 25$ or $41\pmod{71}$; when they
differ, except for $\MM=2$, \phragmen's method favours C and
Thiele's method favours B (except that for $\MM=25$, as said, Thiele's
method gives a tie for the last seat).
\end{example}

\begin{example}[\phragmen's and Thiele's unordered or ordered methods]
\label{ECassel}
\citet{Cassel} gave an example similar to \refE{EPhr1899} (using unordered
ballots). 

Unordered or ordered ballots.
9 seats.
\begin{val}
\item [4200]ABCDEFGHI
\item [1710]ABCUVWXYZ
\end{val}
With both \phragmen's method and Thiele's  method,
the 3 common candidates ABC are obviously
elected first (in arbitrary order in the case of unordered ballots).

With \phragmen's method, the 6 seats after the 3 first are distributed as
for two disjoint parties (again easily seen from the formulation in
Sections \ref{SPhru1} and \ref{SPhro1}), 
\ie, as with D'Hondt's method (see \refS{Sparty}). 
The result is 4 of these seats to the larger ``party'' and 2 to the smaller
(since $4200/4 > 1710/2 > 4200/5$); \ie, ABCDEFGUV. 

With Thiele's method, however, the last 6 seats go to DEFGHI; 
thus the largest party (faction) 
gets all 9 seats, again as by D'Hondt's method regarding ABC as different
persons on the two ballots.
(The last elected, I, is elected with $4200/9\doteq466.67$ votes against 
$1710/4=427.5$ for
U.)
\end{example}

% \citet{Cassel}
%  
% Om två partier (eller i användningen i Sverige 1909--1921, två
%   fraktioner inom   ett parti) 
% $A_1A_2A_3\dots$ och $B_1B_2B_3\dots$
% har lika många röster, så får de naturligvis lika många mandat, och ev.\
% lottas det sista mandatet mellan dem. Om nu ett tredje parti också stöder
% $B_1$, och röstar $B_1C_1C_2\dots$, så kommer $B_1$ att bli vald före $A_1$,
% men därefter blir resultatet jämt mellan $A$-partiet och $B$-partiet precis
% som förut (medan $C$-partiet får en större röstreducering); B-partiet får
% alltså ingen fördel av att fått partiellt stöd från $C$-partiet.
% 
% \end{example}

\begin{example}[\phragmen's and Thiele's unordered or ordered methods]
\label{E1913.5}
  An example from the commission report \cite{bet1913} (there with unordered
  ballots):

Unordered or ordered ballots. 2 seats.
\begin{val}
\item [21]AB
\item [20]AC
\item [12]D
\end{val}
Again there are two parties, with the first split into two factions, both
recognizing the same leader.
Note that since the first party gets more than three times as many votes as
the second, an election without splitting the vote between two lists would
give both seats (and also a third, if there were one)  to the first party by
any of the methods considered  here (since they would reduce to D'Hondt's
method by \refT{TDHondt}). 

\phragmen's method gives the first seat to A (41 votes), and then the second to
B (with C next in line);
the \jfr{s} are $W_\rmB=21/(1+\frac{21}{41})=\frac{861}{62}\doteq13.89$,
$W_\rmC=20/(1+\frac{20}{41})=\frac{820}{61}\doteq13.44$
and $W_\rmD=12$.

Thiele's unordered or ordered method gives the first seat to A, which
reduces the votes for the first two types of ballots so much that the second
seat goes to D, with 12 votes against 10.5 for B and 10 for C.
Hence, AD are elected, and the larger party gets only one seat.
\end{example}

\begin{example}[\phragmen's and Thiele's unordered or ordered methods]
\label{ECassel.p53}
\citet[p.~53]{Cassel} gave  a similar example (with unordered
ballots).

Unordered or ordered ballots.
\begin{val}
\item [90]A\subb1A\subb2A\subb3\dots
\item [90]B\subb1B\subb2B\subb3\dots
\item [90]B\subb1C\subb1C\subb2C\subb3\dots
\end{val}
There are three different parties (factions) of equal size, 
but the C party supports B\subb1.

With Thiele's method, evidently B\subb1 is elected first.
However, as Cassel notes,
once B\subb1 is elected, the B party
has no advantage at all of the partial support from the C party, 
and parties A and B will get equally many seats (possibly tying for the last
one),
while the C party pays for the support of B\subb1 as much as if the party
had elected one of its own, and will get one seat less than the others (up
to ties for the last place). 
Hence, the candidates are elected in order 
B\subb1A\subb1(A\subb2B\subb2C\subb1)(A\subb3B\subb3C\subb2)\dots, where the
parentheses indicate ties.
For example, with 5 seats, the elected are 
2 A, 2 B and 1 C, and with 6 seats, the additional seat is a tie between all
three parties.

With \phragmen's method, B\subb1 is still elected first, but this gives only
a place number $\frac12$ to each of the last two lists.
It follows, \eg{} by considering loads as in \refR{Rload}, that
the candidates are elected in order
B\subb1A\subb1(B\subb2C\subb1)A\subb2(B\subb3C\subb2)A\subb3\dots.
For example, with 5 seats, the elected are 
2 A, 2 B and 1 C, as by Thiele's method, but with 6 seats, the additional
seat is a tie between B and C only; thus B gets some advantage from the
partial support from C, and C does
not pay quite as much for it as with Thiele's method.
\end{example}

% \begin{example}\label{ETenowKLM}
%   An example by \citet{Tenow1912}:
% 
% Unordered ballots.
% 3 seats.
%   \begin{val}
%   \item [12] ABD
%   \item [12] ACE
%   \item [7] KLM
%   \end{val}
% Again there are two parties, with the first split into two factions, both
% recognizing the same leader.
% Note that since the first party gets more than three times as many votes as
% the second, an election without splitting the vote between two lists would
% give all 3 seats to the first party by any of the methods considered
% here (since they would reduce to D'Hondt's method by \refT{TDHondt}).
% 
% \phragmen's method gives the first seat to A, and then the second and third
% seats to, say, B and C. 
% 
% Thiele's (addition) method gives the first seat to A, but then K (say) gets
% the second seat and B the third seat. Thus the small party gets one seat.
% 
% The difference comes from the different methods of reducing the votes for
% the first two lists when A has been elected. With \phragmen's method, both
% lists get a place number of $\frac12$, and the vote 12 is reduced to
% $12/1.5=8$ in the next round, while with Thiele's method, the vote is
% reduced to $12/2=6$. 
% As \citet{Tenow1912} comments, Thiele's method here treats A as two different
% persons.
% %Tenow's conclusion is that Thiele's method reduces the votes too much in
% %cases like this.
% \end{example}

\begin{example}[\phragmen's and Thiele's unordered or ordered methods]
\label{ETenow96}
 \citet{Tenow1912} gave the following example (with unordered ballots):

Unordered or ordered ballots.
8 seats. 
  \begin{val}
  \item [21] ABCDH
  \item [21] ABCEI
  \item [21] ABCFJ
  \item [21] ABCGK
  \item [12] OPQRS
  \end{val}
In this example, there are two disjoint parties, but one party is split into
four different lists. The second party has $12/96 =1/8$ of the votes, so a
proportional representation between the parties would give 7 seats to the
first party and 1 to the second, electing for example ABCDEFGO (with many
equivalent choices, by symmetry, in the unordered case).

This is also the result of \phragmen's method. 

However,  Thiele's \opt{} method, addition  method and elimination method
in the unordered case
and Thiele's ordered method in the ordered case all elect
ABCDEFOP, with two seats to the second party.
\xfootnote{
It is natural that parties
that split their votes on several lists may be hurt by this, and that
happens also with \phragmen's method in other examples, 
see \eg{} \refE{Esplit}, so one should
perhaps be
careful with drawing conclusions from this example.
}
\end{example}

It seems that Thiele's methods reduce the votes for lists with common names
too much in the examples above. 
On the other hand, 
the opposite seems to  happen with Thiele's ordered method
when two list have a common name in a lower
position; in this case the method does not reduce the votes enough.

\begin{example}[\phragmen's and Thiele's ordered methods]\label{E1913.16}
    An example from the commission report \cite{bet1913}:

Ordered ballots. 3 seats.
\begin{val}
\item [34]AC
\item [34]BC
\item [32]D
\end{val}
With Thiele's ordered method, the three seats go to ABC (34 votes each,
since for the third seat, C has $34/2+34/2$ votes),
while D is not elected (32 votes).

\phragmen's method elects instead ABD;
for the third seat, the two first lists have place number 1 each, so C has
\jfr{} $(34+34)/3\doteq22.67$, while D has 32.
This seems to be a more proportional result.
\end{example}

\begin{example}[\phragmen's and Thiele's ordered methods]\label{E1913.17}
    Another example from the commission report \cite{bet1913}:

Ordered ballots. 4 seats.
\begin{val}
\item [33]AB
\item [32]AC
\item [18]DF
\item [17]EF
\end{val}
This example combines the features of Examples \ref{E1913.5} and \ref{E1913.16}.

With Thiele's ordered method, A is elected first (65 votes), and then
D, E, F are elected to the remaining seats (with 18, 17 and 17.5 votes,
against 16.5 for B and 16 for C).
Thus the ABC party, with 65\% of the votes, gets only 1 seat, while the
smaller DEF party gets 3.

\phragmen's method also elects A first (65
votes), but then B and C are elected with \jfr{s}
$W_\rmB=33/(1+\frac{33}{65})=\frac{2145}{98}\doteq21.89$ 
and $W_\rmC=32/(1+\frac{32}{65})=\frac{2080}{97}\doteq21.44$ against
$W_\rmD=18$ and $W_\rmE=17$; finally, D takes the last seat.
Elected: ABCD.
\end{example}

Furthermore, in the following  examples, Thiele's methods seem to
reduce the votes too little for candidates that appear both on a list where
someone else already has been elected, and on a list without any elected.
Moreover, and perhaps more seriously,
this yields possibilities for tactical voting, where a party may
gain seats by carefully splitting its votes on different lists.

\begin{example}[\phragmen's and Thiele's unordered methods]\label{Etactic}
  An example by \citet{Tenow1912}:

Unordered ballots. 3 seats.
\begin{val}
\item [37]ABC
\item [13]KLM
\end{val}
In this case, there are two disjoint party lists, 
so both \phragmen's and Thiele's
methods reduce to D'Hondt's method (see \refT{TDHondt})
and, say, ABK are elected. Thus the larger party gets 2 seats and the smaller 1.

However, with Thiele's method,
the larger party may cunningly split their votes on five different
lists as follows:
\begin{val}
\item [1]A
\item [9]AB
\item [9]AC
\item [9]B
\item [9]C
\item [13]KLM
\end{val}
Then A gets the first seat (19 votes), 
and the next two go to B and C (in some order)
with 13.5 votes each, beating KLM with 13. Thus the large party gets all seats.
The reason is obviously that only some of the ballots contain A, and thus
get their votes reduced for the following seats.

Nevertheless, if news of this scheme is leaked to the small party, they may
thwart it by letting two persons vote BKLM. The resulting election is
\begin{val}
\item [1]A
\item [9]AB
\item [9]AC
\item [9]B
\item [9]C
\item [2]BKLM
\item [11]KLM
\end{val}
Now, still with Thiele's method, B gets the first seat (20 votes); then
C gets the second (18 votes); finally, K, say, gets the third seat with 12
votes, beating A (10 votes). Thus, KLM have gained a seat by voting for the
enemy!
(Cf.~\refS{Smono2}.)

I do not know whether there are in this example even more cunning schemes
that can not be thwarted, and what the best strategies of the two parties
are if they do not know in advance how the others vote. (Possibly the best
strategies are mixed, \ie{} random.) Thiele's method thus leads to
interesting mathematical problems in combinatorics and game theory, but for
its practical use as a voting method, this example of sensitivity to
tactical voting is hardly an asset.

\phragmen's method seems to be much more robust in this respect. In the
example above, KLM can always get one seat by voting KLM, regardless of how
the voters in the other party vote, see \refT{TPhPC}.
\end{example}

\begin{example}[\phragmen's and Thiele's ordered methods]  \label{Etactic-o}  
%\label{E1913.14}
The opportunities for tactical voting shown by Thiele's unordered method in
\refE{Etactic} exist for
Thiele's ordered method as well, in a somewhat simpler and perhaps more
obvious form.
The following example is (partly) from the commission report \cite{bet1913}.

Ordered ballots. 2 seats.
\begin{val}
\item [61]AB
\item [39]CD
\end{val}
Again, there are two disjoint party lists, so both \phragmen's and Thiele's
methods reduce to D'Hondt's method (see \refT{TDHondt})
and AC are elected. Thus the parties get 1 seat each. (Since both have more
than a third of the votes.)

However, with Thiele's method,
the larger party may split their votes as follows:  
  
%Ordered ballots. 2 seats.
\begin{val}
  \item [41]AB
  \item [20]B
  \item [39]CD
  \end{val}
With Thiele's  method,
the first seat goes to A (41 votes, against 20 for B and 39 for C).
For the second seat, B  has $41/2+20=40.5$ votes, and beats C. Elected: AB.

In this example, C thus is not elected, in spite of being
supported by $39\%$ of the voters. (See further \refS{Sprop}.)
Unlike \refE{Etactic} with the unordered method, in this case there is
nothing that the CD party can do to prevent AB from being elected; 
even if they help B to be elected to the
first seat, A will still beat them to the second seat.

\phragmen's method is more robust
(as for the unordered methods); using it,
CD will always get (at least) one seat by voting CD,
regardless of how the others vote,
see \refT{TPhPC}.
In the example above, 
\phragmen's method elects A to the first seat. This gives place number
1 to the first list, and for the second seat, B has an \jfr{}
$(41+20)/2=30.5$ by \eqref{wio}. Hence C gets the second seat.
Elected: AC.
\end{example}

\begin{example}[Thiele's ordered method]\label{Etactic-o2}
  Consider the following variation of the second election in \refE{Etactic-o}.

Ordered ballots. 3 seats. Thiele's method.
\begin{val}
  \item [30]AB
  \item [15]B
  \item [55]CD
  \end{val}
C gets the first seat, but then D has only $55/2=27.5$ votes and the second
and third seats go to A and B with 30 votes each.  Elected: ABC.

Thus the CD party gets only 1 seat in spite of a majority of the votes.
\end{example}

Also \phragmen's method can behave strangely, as seen in the following example.
Further examples of non-intuitive behaviour 
are given in Sections \ref{Smono2} and \ref{Sconsistency}.

\begin{example}[\phragmen's ordered method]\label{ELanke}
This  example
was constructed by \citet{Lanke:egendomlighet}.

Ordered ballots. 2 seats. \phragmen's method.
\begin{val}
\item [15]AB
\item [12]BX
\item [14]CY
\item [3]ZW
\end{val}
Phragmén's method gives the first seat to A (\jfr{} 15) and the second to C
(\jfr{} 14, against  $27/2=13.5$ for B and 3 for Z).
Elected: AC.

Now suppose that the three ZW voters change their minds and vote AC:
\begin{val}
\item [15]AB
\item [12]BX
\item [14]CY
\item [3]AC
\end{val}
Then the first seat goes to A (\jfr{} 18).
For the second seat, the four groups of ballots have place numbers
$\frac{5}6,0,0,\frac16$, giving B and C the \jfr{s}
\begin{align*}
  W_\rmB=\frac{15+12}{1+5/6}=\frac{162}{11}\doteq14.73
\quad\text{and}\quad
  W_\rmC=\frac{14+3}{1+1/6}=\frac{102}{7} \doteq14.57;
\end{align*}
thus the second seat goes to B.

Thus C lost the seat because of 
the new votes for AC!
(Cf.~\refS{Smono2}, and note that
this example does not contradict \refT{TmonoPhro}, since the change here
involves also A and not only C.)

The explanation of this seems to be that \phragmen's method implicitly
regards A and B  (and  X) as a party.
The new votes AC help C, but they are also extra votes for A that help the
AB ``party'', and thus (since A already is elected) B.
(In both cases above, the first seat goes to A.
The extra votes for A in the second case, \ie, after the change,
means that the place numbers get smaller so each vote is reduced less, and
more votes are counted for B in the second round.)
The calculations above show that,  perhaps surprisingly, the indirect gain
for B is larger than the direct gain for C.

Nevertheless, it seems strange that more votes on
the winning combination AC will make someone else elected, and
\citet{Lanke:egendomlighet} regards this example is an argument  against
\phragmen's ordered method.
(On the other hand, as discussed
in \refS{Smono2}, it seems impossible for any election method 
with ordered ballots to avoid similar examples.)
\end{example}

The final example involves the combination of methods used in Sweden
1909--1921. %, see \refApp{Ahistory}.
\begin{example}[Ranking Rule + Thiele's unordered method]\label{Erank}
An example from the commission report \cite{bet1913}.

Ordered and unordered ballots.
\begin{val}
\item [570]ABC
\item [290]BC
\item [20]C
\end{val}

With ordered ballots, the candidates are obviously elected in order A, B, C
by any method considered here,
and with unordered ballots, they are elected in the opposite order C, B, A.
Both results are obviously fair by all standards, but a problem arises with the
combination of ordered and unordered rules used in Sweden 1909--1921, see
\refApp{Ahistory}. With those rules, A would get the first seat (because of
the ordered Ranking Rule in \refApp{Ahistory}), but then the unordered
Thiele's method would be used, and the remaining candidates would be elected
in order C, B. Thus, if 2 were elected, it would be AC, which hardly seems
fair by any interpretation of the ballots.

This example illustrates a situation that actually occurred in some cases
\cite{bet1913}.
In fact, it could easily happen if the party organization recommends a list
ABC, in order of priority, and most of the voters like this and vote as the
party says, but a sizeable minority prefers B and therefore omits A, but
still keeps C after B on the ballot (perhaps not fully aware of the
implications of keeping C).

Note also that if 20 of the BC voters had voted ABC instead, then the
Ranking Rule (\refApp{Ahistory})
would have applied to both A and B, and thus with 2 seats, AB
would have been elected. Hence, the voters favouring B have thus in reality
eliminated their candidate by voting BC instead of ABC!

In the opposite direction, suppose that there are also 300 votes on other
lists with the same party name but other candidates
(other factions, or another party in
electoral alliance, see \refApp{Ahistory});
then the Ranking Rule would not be applicable at all, and if the lists above
still would get 2 seats, they would go to C and B, while the party's top
candidate A would not be elected in spite of a majority of these voters
preferring A. And if the same list only appoints 1 candidate, it would be C
and not A. This can be seen as a form of decapitation, see \refS{Sdecap}.
A real example of this type from the general election in 1917
(Skaraborg County south)
is discussed in \cite{Tenow1918}.
  \end{example}

\section{Monotonicity}\label{Smono2}

It is well-known that no perfect election method can exist.
For example, the Gibbard--Satterthwaite theorem \cite{Gibbard,Satterthwaite}
says briefly that every deterministic
election method for ordered ballots
is susceptible to tactical voting, \ie, there exist situations where 
a voter can obtain a result that she finds better by not voting according to
her true preferences.
\xfootnote{Assuming only that there are at least three candidates that can
  win, and that the 
  election method is not a dictatorship where the result is determined by 
a single voter.
See also the related Arrow impossibility theorem \cite{Arrow}.}
A simple and well-known (also in practice)
example of this is when a single candidate is
elected by simple 
plurality (First-Past-The-Post, \refApp{ABV}); 
recall that this includes both \phragmen's and
Thiele's ordered methods when $\MM=1$, see \refT{T1}.
If candidates B and C have equally strong support, and candidate A a much
weaker, then a voter that  prefers them in order ABC may do better by
voting BAC, since this might help getting B elected instead of C, while the
voter's first 
choice A is without a chance in any case.

Nevertheless, some election methods are more imperfect than others.
Particularly disturbing are  cases of non-monotonicity when
a candidate can become elected if some votes are changed to less favourable for
the candidate; equivalently, changing some votes in favour of an elected
candidate 
A might result in A losing the election.
Unfortunately, this can happen in Thiele's ordered method.
\xfootnote{\label{f-monoSTV}%
This type of non-monotonicity is also a well-known problem with
  STV
(\refApp{ASTV}), where it may occur in particular because of the eliminations.
%(and also for other reasons, such as rounding of the quota).
% Ex. 50000 ABCDE, 9999 XYZ; 5 seats. With the Droop quota rounded up to
% 10000, ABCDE get all 5 seats. With another vote for ABCDE, and the Droop 
%quota rounded up to 10001, X will get one seat.
This can happen even in a single-seat election, when STV 
reduces to Alternative Vote (regardless of the
version used for transfering surplus, since there is no surplus when
$\MM=1$).
A simple example is 6~ABC, 6~BCA, 5~CAB, where C is eliminated and A is
elected;
however, if A attracts more support from two voters that change from BCA to
ABC,
so that the votes are 8~ABC, 4~BCA, 5~CAB, then B is eliminated and C is
elected instead of A, see further \cite{Woodall:monotonicity}.
See also the impossibility theorem by \citet{Woodall:impossibility}, but
note that \phragmen's and Thiele's methods do not satisfy property (4) there.
}

\begin{example}[Non-monotonicity for Thiele's ordered method]\label{E-monoTh}
\quad

Ordered ballots. 2 seats. Thiele's method.
\begin{val}
\item [5]ABC
\item [2]ACB
\item [9]BCA
\item [8]CAB
\end{val}
With Thiele's method, first B is elected (9 votes), and then C is elected to
the second seat
(with 12.5 votes against 7 for A), so the result is BC. 
However, suppose that the two ACB voters move A to the second place
(although they really favour A):
\begin{val}
\item [5]ABC
\item [9]BCA
\item [10]CAB
\end{val}
Then C will be elected first (10 votes), followed by A (10 votes, against 9
for B), so the result is AC. 
Hence, these voters have got A elected %(instead of B)
by moving A to a lower place.
\end{example}

The type of non-monotonicity in Example \ref{E-monoTh} cannot happen with
\phragmen's ordered method, or by any of \phragmen's and Thiele's unordered
methods. 
For the unordered methods, 
this was shown by \citet{Phragmen1896} for a class of methods that includes
his method, see \refT{TPh1896}; he did not consider Thiele's
methods, but the same proof applies to them too, and yields the following
result for the unordered methods considered so far. (The proof applies also
to Thiele's methods with other satisfaction functions, 
except that a minor complication may occur with ties if some $w_n=0$ as for
the weak method.)
However, note that there may be other types of non-monotonicity; 
for Thiele's unordered method, one example is given in \refE{Etactic}.

\begin{theorem}[partly \citet{Phragmen1896}]
  \label{Tmono-u}
\phragmen's unordered method and Thiele's three unordered methods
satisfy the following:
Consider an election and let A be one of the candidates.
Suppose that the votes are changed such that A gets more votes, either
from new voters that vote only for A, or from 
voters that add A to their ballots (thus changing the vote from some set
$\gs$ to $\gs\cup\set A$), but no other changes are made
(thus all other candidates receive exactly the same votes as before).
Then this cannot hurt A; if A would have been elected before the change,
then A will be elected also after the change.
\xfootnote{There is no problems with ties. The proof shows that there are
  strict improvements, so if A was tied before the change, 
then A will certainly win after the change.
(This is called \emph{positive responsiveness}, see \eg{}
\cite{May1952}.)
}\,
\xfootnote{Note that the elected set is not necessarily the same. Adding
  A to some ballots can affect, for example, whether B is elected or not
  together with A. This  applies to all the methods.
For example, if we modify Examples \ref{Enonmono1} and \ref{Enonmono2} below 
by adding only a vote on A\sss1, 
then this results in replacing A\sss2 by C or B.}
\end{theorem}

\begin{proof}
For Thiele's \opt{} method, note that the total satisfaction is unchanged
for every set $\cS$ of candidates that does not include A, but it is increased
for every set of candidates that includes A. Hence, if one set including A
maximized the total satisfaction before the change, then the same holds
after the change.
\xfootnote{The maximum set is not necessarily the same.}
Thus A will still be elected.

For \phragmen's unordered method and Thiele's addition method, we note that
as long as A is not elected, then all other candidates have exactly the same
voting power or vote count before and after the change, while this has
increased for A.
Hence, if A would have been elected in round $\ell$ before the change, then
after the change either A has already been elected, or A will be elected in
the same round $\ell$.

Finally, for Thiele's elimination method, note that if there are $N$
candidates, the elimination method can be seen as repeatedly using the
\opt{} method to select $\MM=N-1, N-2, \dots$ candidates, until the desired
number remains, after each round eliminating the non-selected candidate from
all further consideration. The result just proved for the \opt{} method
shows that 
if A survives all rounds before the change, then A survives all rounds also
after the change.
\end{proof}

For \phragmen's ordered method we have the following corresponding result.
Note that, as shown by \refE{E-monoTh}, this theorem fails for Thiele's
ordered method.

\begin{theorem}
  \label{TmonoPhro}
\phragmen's ordered method satisfies the following:
Consider an election and let A be one of the candidates.
Suppose that some votes are changed in favour of A,
either by adding a new ballot containing only A, 
by adding A to an existing ballot 
at an arbitrary place (if A did not already exist on the ballot),
or by moving A to a higher position (if A already existed on the ballot); in
all cases without changing the other candidates and their  order.
Then the change cannot hurt A; if A would have been elected before the
change, then the same is true after the change.
\end{theorem}

\begin{proof}
  We use the formulation in \refS{SPhro1} and \refR{Rtimeo} with increasing
  time; thus, at time $t$, each ballot has voting power $t$.
Let $f_i(t)$ and $\fx_i(t)$ be the total voting  power for candidate
$i$ at time $t$ for the original set of ballots and for the changed ballots,
respectively. 
Note that $0\le f_i(t)\le 1$ and that $i$ has been elected (with the
original ballots) at time $t$ or earlier if and only if $f_i(t)=1$, and
similarly for $\fx_i(t)$.
We claim the following:
{\em
\begin{equation}\label{cl}
  \text{If\/ }\fx_\rmA(t)<1, \text{ then } f_\rmA(t)\le \fx_\rmA(t)
\text{ and } f_i(t)\ge \fx_i(t)
\text{ for } i\neq \rmA.
\end{equation}}
To prove this, we let $\set{t_\ell}_{\ell=1}^L$ 
be the set of times when some candidate is elected for either the original
or the changed ballots, ordered in increasing order. We let $\tA$ be the
time that A is elected for the changed ballots, so the assumption
$\fx_\rmA(t)<1$, meaning that A is not yet elected, is equivalent to $t<\tA$.
We also let $t_0=0$. We prove the claim \eqref{cl}
by induction, assuming that it is
true for $0\le t\le t_\ell$, for some $\ell\ge0$ with $t_\ell<\tA$, and show
that \eqref{cl} holds 
also for $t\in I_\ell:=[t_\ell,t_{\ell+1}]$.

To show this induction step, note that by the inductive assumption, 
for
every candidate $i\neq \rmA$ 
that has not yet been elected at $t_\ell$ for the original ballots, we have
$1>f_i(t_\ell)\ge \fx_i(t_\ell)$, and thus $i$ is also not elected for the
changed ballots. 
Furthermore, $t_\ell<\tA$, so A has not been elected  at $t_\ell$ for the
changed ballots. Consequently, every candidate that at $t_\ell$
has been elected for the changed ballots has also been elected for the
original ballots.

Consider now a candidate $i\neq \rmA$. If $i$ has been elected for the original
ballots at $t_\ell$, then
$f_i(t_\ell)=1$ and thus $f_i(t)=1\ge\fx_i(t)$ for all $t\ge
t_\ell$.
On the other hand, suppose that $i$ has not been elected yet for the original
ballots at $t_\ell$, so $f_i(t_\ell)<1$. Consider a ballot that, after the
change, is valid
for $i$ immediately after $t_\ell$;
this means that $i$ is on the ballot and any candidate $j$ before $i$ on the
ballot has been elected, \ie, $\fx_j(t_\ell)=1$. 
In particular, $j\neq \rmA$.
Hence, before the change, this ballot had the same names before $i$, 
%except possibly that A might have been absent, 
and we have seen that they have been
elected for the original ballots too at $t_\ell$; thus the ballot is valid
for $i$ also in the original set of ballots.
If we increase time (\ie, voting power) from $t_\ell$ to $t_{\ell+1}$, 
then, because there are no elections inside this interval,
the voting power of $i$ increases linearly, with a derivative equal to the
number of ballots valid for $i$. Hence we have shown that 
$f_i'(t)\ge \fx_i'(t)$ for $t\in(t_\ell,t_{\ell+1})$, and 
we see that $f_i(t)\ge \fx_i(t)$ for $t\in[t_\ell,t_{\ell+1}]$ because, by
the induction hypothesis, this holds for $t=t_\ell$.

We have verified the last part of the claim \eqref{cl}.
Now, the total voting power assigned or available from a ballot $b$ at time $t$
equals $\min(t,t_b)$, where $t_b$ is the time the last candidate on the ballot
is elected. We have seen that if the last candidate on the ballot is elected
at some time $\hat t_b$ for the set of changed ballots, and $\hat
t_b\le t_{\ell}<\tA$, 
then at time $\hat t_b$ all candidates on the ballot have been elected also for
the original set of ballots, so $t_b\le\hat t_b$. Consequently, for each
ballot, and each time $t<t_{\ell+1}$, the total voting power that is used from a
ballot is at least as large after the change as before. Summing over all
ballots (and noting that some ballots may have been added, but none
removed),
we find
that the total voting power used for all candidates is at least as large
after the change as before, \ie, $\sum_i \fx_i(t)\ge\sum f_i(t)$, 
$t<t_{\ell+1}$. 
For $t\in[t_\ell,t_{\ell+1})$, we have also shown that $\fx_i(t)\le f_i(t)$
for every $i\neq A$, and it follows that $\fx_\rmA(t)\ge f_\rmA(t)$.
By continuity, this holds for $t=t_{\ell+1}$ too, which completes the
inductive step verifying \eqref{cl}.

Finally, \eqref{cl} says that the conclusion of \eqref{cl}
holds when $\fx_\rmA(t)<1$, \ie,
when $t<\tA$. By continuity, the conclusion holds also for $t=\tA$.
Now suppose that $\rmA$ is not elected with the changed ballots; then 
a set $\hcE$ of $\MM$ 
other candidates have been elected at time $\tA$. 
This together with the claim \eqref{cl} implies that for each $i\in \hcE$,
$f_i(\tA)\ge\fx_i(\tA)=1$, and thus $i$ has at time $\tA$ been elected also
for the original ballots. On the other hand, for any $t<\tA$, by \eqref{cl},
$f_\rmA(t)\le\fx_\rmA(t)<1$, so $\rmA$ has not been elected earlier.
Consequently, $\rmA$ is not in the set $\cE$ of the $\MM$ first 
elected with the original
ballots,
which completes the proof.
(In the case of A tying the last place both for the original and changed
ballots, we have to 
assume that the tie is resolved in the same way.)
\end{proof}

However, also \phragmen's methods are liable to 
non-monotonicity of other types. One such case for \phragmen's ordered method
is shown by \refE{ELanke}.
We give  two examples for \phragmen's unordered method,  
taken from \citet[Section 7.5]{MoraO},
which show that \refT{Tmono-u} does not extend to sets of more than one
candidate, even if they always appear together (and thus are tied).

\begin{example}[Non-monotonicity for \phragmen's unordered method]
\label{Enonmono1}
An example from \citet[(43)]{MoraO}.

Unordered ballots. 3 seats. \phragmen's method.
  \begin{val}
  \item [10]A\sss1A\sss2    
  \item [3]B
  \item [12]C
  \item [21]A\sss1A\sss2B
  \item [6]BC
  \end{val}
A calculation shows that \phragmen's method elects A\sss1, C, A\sss2.
(For the second seat, C has \jfr{} 18 against
$30/(1+\frac{21}{31})=465/26\doteq17.88$ for B and $31/2=15.5$ for A\sss2.
Then A\sss2 is elected with \jfr{} $31/2=15.5$ against $2790/187\doteq14.92$
for C.)

Now add another vote for A\sss1A\sss2:
  \begin{val}
  \item [11]A\sss1A\sss2    
  \item [3]B
  \item [12]C
  \item [21]A\sss1A\sss2B
  \item [6]BC
  \end{val}
The first seat still goes to A\sss1, but its load
on each ballot A\sss1A\sss2B is changed from $1/31$ to $1/32$,
and as a consequence, for the second seat, B's \jfr{} is increased to 
$30/(1+\frac{21}{32})=960/53\doteq18.11$, and thus B gets the second seat.
This, in turn, increases the load on each ballot A\sss1A\sss2B to
$53/960$, and thus for the third seat, the \jfr{} of A\sss2 is only
$10240/801\doteq12.78$ against $960/71\doteq13.52$ for C, so C gets the
third seat. Thus the elected are A\sss1, B, C.

Consequently, the extra vote on A\sss1A\sss2 means that one of them loses a
seat.
\end{example}

\begin{example}[Non-monotonicity for \phragmen's unordered method]
\label{Enonmono2}
An example from \citet[(42)]{MoraO}.
  
Unordered ballots. 3 seats.  \phragmen's method.
  \begin{val}
  \item [4]A\sss1A\sss2    
  \item [7]B
  \item [1]A\sss1A\sss2B
  \item [16]A\sss1A\sss2C
  \item [4]BC
  \end{val}
A calculation shows that \phragmen's method elects A\sss1, B, A\sss2.

Now suppose that one of the voters for B adds A\sss1 and A\sss2 to the
ballots.
Thus the ballots are:
  \begin{val}
  \item [4]A\sss1A\sss2    
  \item [6]B
  \item [2]A\sss1A\sss2B
  \item [16]A\sss1A\sss2C
  \item [4]BC
  \end{val}
A\sss1 still gets the first seat, but now the second seat goes to C, and
then the third seat goes to B; the elected are thus A\sss1, C, B.

Consequently, also in this example,
the extra vote on A\sss1A\sss2 means that one of them loses a
seat.
\end{example}

\section{Consistency}\label{Sconsistency}

\emph{Consistency} \cite{Young}
of an election method is
the property that if there are two disjoint sets $\cV_1$ and $\cV_2$ of voters
voting such that the votes from only $\cV_1$ would elect a set $\cE$ of
candidates, and  
the votes from only $\cV_2$  would elect the same set $\cE$, then
the combined set of votes from $\cV_1\cup\cV_2$ elects the same set $\cE$.
To be precise (to allow for ties),
we require that
if $\cE$ is a possible outcome for the votes from only $\cV_1$,
and also a possible outcome for the votes from only $\cV_2$,
then $\cE$ is a possible outcome for 
the combined set of votes from $\cV_1\cup\cV_2$.
\xfootnote{In the case of ties, \cite{Young} also requires
that, conversely, if
there exists a set  $\cE$ that is a possible outcome for both $\cV_1$ and
$\cV_2$, then any other 
possible outcome $\cE'$ of $\cV_1\cup\cV_2$ (so $\cE'$ ties with $\cE$)
 is a possible outcome for both $\cV_1$ and $\cV_2$.
We will not consider this, more technical, condition.
}

\begin{remark}\label{Rconvex}
  Consistency implies, in particular, that multiplying the number of ballots
  of each type by some integer will not affect the outcome of the election;
  in other words, the election method is homogeneous, see \refR{Rhomo}.
We assume that all voters are equal, so that only the number of ballots of
each type matters.\xfootnote{A property called \emph{anonymity}.}
Then the input to the election method, for a given set $\cC$ of candidates,
can be seen as a vector of non-negative integers, indexed by subsets of
$\cC$ in the unordered case and by ordered subsets of $\cC$ in the ordered
case. Homogeneity means that we can extend the method (uniquely) to vectors 
of non-negative rational numbers, such that the set $C_\cE$ of such
vectors that (may) elect a  set $\cE$ is a cone (or empty), for any set
$\cE$ of $s$ candidates.
Furthermore, as is easily seen, consistency is equivalent to the cone
$C_\cE$ being
convex, for every possible outcome $\cE$. 
(Hence the property is also called \emph{convexity} \cite{Woodall:properties}.)
\end{remark}

Consistency is a very natural property. Unfortunately, it is not compatible with
some other natural properties, and it is not satisfied by most
election methods.
In fact, \citet{Young} showed that (assuming some technical
conditions) the only consistent election methods for ordered ballots are the
scoring rules (Borda methods) (\refApp{Aborda}).
Similarly,  \citet{Lackner} show that 
(again assuming some technical conditions)
every consistent method for unordered ballots
equals Thiele's \opt{} method for some satisfaction function.
We record the easy part of this as a theorem.
\begin{theorem}\label{TTconvex}
  Thiele's \opt{} method is consistent (for any satisfaction function).
\end{theorem}
\begin{proof}
  Obvious from the definition of the method in \refS{SSthiuopt}.
\end{proof}

By the theorems just mentioned, \phragmen's unordered and ordered methods
and Thiele's unordered (addition) and ordered methods are not consistent.
We give some explicit counterexamples.

\begin{remark}
In the case $\MM=1$, \refT{T1} says that
\phragmen's method and Thiele's \opt{} and addition
methods for unordered ballots reduce to Approval Voting, 
while \phragmen's and Thiele's methods for
ordered ballots reduce to Single-Member-Plurality; 
Approval Voting and Single-Member-Plurality are both obviously consistent, so 
all of \phragmen's and Thiele's methods (in the latter case for any
satisfaction function), except possibly Thiele's elimination method, are
consistent for $\MM=1$.
\end{remark}

\begin{example}[\phragmen's and Thiele's unordered methods are not consistent]
\label{Econs-u}
\quad

Unordered ballots. 2 seats. \phragmen's or Thiele's method.

First set of voters:
  \begin{val}
  \item [12]AB
  \item [1]B
  \item [9]C
  \end{val}
For both \phragmen's and Thiele's (addition) methods,
first B is elected and then C. Elected: BC.

Second set of voters:
  \begin{val}
  \item [12]AC
  \item [9]B
  \item [1]C
  \end{val}
This is the same with B and C interchanged, so the elected are again BC.
(They are elected in different order, but that is irrelevant for the outcome
of the election.)

Combined election:
  \begin{val}
  \item [12]AB
  \item [12]AC
  \item [10]B
  \item [10]C
  \end{val}
This is the same as \refE{ETh12}.
\phragmen's and Thiele's method both elect first A, and then B or C (a tie,
by symmetry). Elected: AB or AC.

Thus neither of the methods is consistent.
\end{example}

\begin{example}[Thiele's ordered method is not consistent]\label{EconsTh-o}
\quad
  
Ordered ballots. 2 seats. Thiele's method.

First set of voters:
  \begin{val}
  \item [12]A
  \item [10]BC
  \item [9]C
  \end{val}
%Both \phragmen's and Thiele's methods elect A and B.
 Thiele's (ordered) method elects A and B.

Second set of voters:
  \begin{val}
  \item [11]B
  \item [10]AC
  \item [9]C
  \end{val}
 Thiele's  method elects B and A.

Combined election:
  \begin{val}
  \item [12]A
  \item [10]AC
  \item [11]B
  \item [10]BC
  \item [18]C
  \end{val}
Thiele's method elect A first (22 votes), 
but then C gets the second seat (23 votes, against 21 for B). 
Elected: AC.

Thus Thiele's ordered method is not consistent.
%\xfootnote{\phragmen's method would elect AB for all three elections.}
\end{example}

In these examples, the same candidates were elected by the two sets of
voters $\cV_1$ and $\cV_2$, but in different order.
It is also possible to find counterexamples where the candidates are elected
in the same order.

\begin{example}[\phragmen's ordered method is not consistent]
	We modify \refE{ELanke}.

Ordered ballots. 2 seats. \phragmen's method.

First set of voters:
\begin{val}
\item [15]AB
\item [12]BX
\item [14]CY
%\item [3]ZW
\end{val}
Phragmén's method gives, just as in \refE{ELanke}, the first seat to A and
the second to C. 
Elected (trivially): AC.

Second set of voters:
\begin{val}
\item [3]AC
\end{val}
Elected: AC.

Combined election:
\begin{val}
\item [15]AB
\item [12]BX
\item [14]CY
\item [3]AC
\end{val}
As seen in \refE{ELanke}, now the elected are A and B.

Hence, \phragmen's ordered method is not consistent.
%\xfootnote{Thiele's method would elect AB in the first election.}
\end{example}

\subsection{Ballots with all candidates}\label{SSfull}

Consider unordered ballots and
the case when the voters in a subset $\cV_1$ of all voters $\cV$ vote for all
candidates. Will their votes affect the outcome at all, or will the result
of the election be the same as if only $\cV_2:=\cV\setminus\cV_1$ had voted?
We call a ballot containing all candidates a \emph{full ballot}, and
we say that an election method \emph{ignores full ballots} if it always
gives the same result (including ties) if we add or remove full ballots.
\xfootnote{
Whether this is a natural, and perhaps important, property or not 
is a philosphical question
rather than a mathematical, and is left to the reader; however,\textsc{}
it seems natural to argue that full ballots with all candidates do not express
any real 
preference, and therefore ought to be ignored just as blank ballots with no
names. 
%On the other hand, it might perhaps also be argued that such voters
%express 
%support for all candidates, which increases the relative amount of support
%for a weak candidate. 
}

This is a special case of consistency; if only $\cV_1$ had voted, with 
full ballots, then the
result would have been a tie between any set of $s$ candidates.
Hence any consistent method will ignore full ballots.
In particular, Thiele's \opt{} method ignores full ballots by
\refT{TTconvex}. Moreover, the following holds.

\begin{theorem}
Thiele's unordered methods all ignore full ballots (for any satisfaction
function). 
\end{theorem}

\begin{proof}
  Obvious from the definitions in \refS{Sthiu}, since a ballot with all
  candidates gives the same contribution to each candidate at every step of
  the calculations.
\end{proof}

On the other hand, \phragmen's method does \emph{not} ignore full ballots.
This follows from the study of party versions by \citet[Section 7]{MoraO}
(discussed in \refSS{SSMora-party} below);
we give two explicit examples below.
(See \cite{MoraO} for further strange behaviour of \phragmen's method in
situations where some voters vote for all candidates.)

\begin{example}[\phragmen's method does not ignore full ballots]
\label{EfullAB}
This example is based on general results in \citet[Section 7.7]{MoraO}.

Unordered ballots. 4 seats. \phragmen's method.
\begin{val}
\item [5] A\sss1A\sss2A\sss3A\sss4
\item [3]B\sss1B\sss2B\sss3B\sss4
\end{val}
This is a case with party lists, so by \refT{TDHondt}, the candidates are
elected as by D'Hondt's method, which gives A\sss1B\sss1A\sss2A\sss3. 
(For example, the last seat
goes to A\sss3 because $5/3>3/2$.

Now add a large number of votes for all candidates:
\begin{val}
\item [5] A\sss1A\sss2A\sss3A\sss4
\item [3]B\sss1B\sss2B\sss3B\sss4
\item [10] A\sss1A\sss2A\sss3A\sss4B\sss1B\sss2B\sss3B\sss4
\end{val}
A calculation shows that the elected will be
A\sss1B\sss1A\sss2B\sss2.
(For example, for the fourth seat, the three 
types of ballots have already loads 
$\frac{34}{195},\frac{25}{195},\frac{34}{195}$
and thus the three
groups of ballots have
place numbers $\frac{34}{39},\frac{15}{39},\frac{68}{39}$, 
and the \jfr{s}
of the A's and B's (which of course tie among themselves) are
$\frac{195}{47}\doteq4.1489$ and $\frac{507}{122}\doteq4.1557$,
so the last seat goes to B.)
\end{example}

\begin{example}[\phragmen's method does not ignore full ballots]
\label{EfullABC}
Another example, showing that already the second seat can differ.

Unordered ballots. 2 seats. \phragmen's method.
  \begin{val}
  \item [10]A
  \item [3]AB
  \item [2]C    
  \end{val}
The first seat goes to A (13 votes).
The second seat goes to B (\jfr{} $3/(1+\frac{3}{13})=39/16=2.4375$ against
2 for C).
Elected: AB.

Now add a large number of votes for all candidates:
  \begin{val}
  \item [10]A
  \item [3]AB
  \item [2]C    
  \item [10]ABC
  \end{val}
A still gets the first seat.
A calculation shows that for the second seat, 
B has \jfr{} $13/(1+\frac{13}{23})=299/36\doteq8.31$ while
C has \jfr{} $12/(1+\frac{10}{23})=92/11\doteq8.36$; 
thus C gets the second seat.
Elected: AC.
\end{example}

Since the result of \phragmen's method thus can be affected by adding
full ballots, it is (mathematically) interesting to ask for the result if we
add a very large number $N$ of full ballots.

\begin{theorem}
  Consider an election with \phragmen's unordered method.
If we add $N$ full ballots, then for all sufficiently large $N$, the outcome
of the new election will be the same as for the original ballots with the
following election method, at least providing that there are no ties for the
latter:
\end{theorem}

\begin{metod}{A limit of \phragmen's unordered method}
Seats are given to candidates sequentially, until the desired number have
been elected.
In round $k$ (\ie, for electing to the $k$-th seat), a ballot $\gb$
is counted as
$1-\ell_\gb/k$ votes for each candidate on the ballot (except the ones
already elected), where $\ell_\gb$ is the last round where some candidate
from the ballot was elected, with $\ell_\gb=0$ if no candidate on the ballot
has been elected yet.
The candidate with the largest number of votes is elected.  
\end{metod}

\begin{remark}
  Equivalently, we may count ballot $\gb$ as $k-\ell_\gb$ votes.
We can thus also describe the method as follows:
\begin{metod0}
Elect the candidate with the largest number of votes. Discard all ballots
with the elected candidate; then add a new copy of each original ballot.
Repeat $s$ times.
\end{metod0}
\end{remark}

\begin{proof}
We consider an election with $s$ seats and $\VV $ 'real' votes, to which we add
$N$ full ballots. 
We use the formulation in \refSS{SPhru1}; in particular, $t\xx n$ is the
woting power required to elect $n$ candidates.
Since an additional voting power $1/N$ gives each candidate voting power 1
just from the full ballots, which is enough to elect any of them,
$t\xx{n+1}\le t\xx n+1/N$. In particular, for every $n\le s$, $t\xx n\le
n/N\le s/N$.
This implies that at time $t\xx{n}$, the elected candidate gets voting
power at most $\VV s/N$ from the real votes, and thus at least $1-\VV s/N$
from the full votes. Hence, $1/N-\VV s/N^2\le t\xx{n}-t\xx {n-1}\le 1/N$,
and thus
\begin{equation}\label{full}
t\xx{n}-t\xx {n-1}=N\qw+O\bigpar{N\qww}.  
\end{equation}
When the $k$-th candidate is elected, a ballot  where the last round that
someone was elected was $\ell$ thus has a free voting power 
$t\xx k-t\xx\ell=(k-\ell)N\qw+O\bigpar{N\qww}$.
The total voting power available for candidate $i$ is thus
\begin{equation}\label{fullm}
T_i=\sum_{\gb\ni i} (k-\ell_\gb)N\qw+O\bigpar{N\qww}+N(t\xx k-t\xx{k-1}),
\end{equation}
summing over all real ballots containing $i$. Since the total voting power
$T_i$ is $1$ for the elected candidate, and at most 1 for every other, 
we see that if $i_k$ is the elected in round $k$, then
for every $i$, $T_i\le T_{i_k}$, and thus by \eqref{fullm}
\begin{equation}
\sum_{\gb\ni i} (k-\ell_\gb)N\qw
\le 
\sum_{\gb\ni i_k} (k-\ell_\gb)N\qw
+O\bigpar{N\qww}
\end{equation}
for every $i$, and thus
\begin{equation}
\sum_{\gb\ni i} (k-\ell_\gb)
\le 
\sum_{\gb\ni i_k} (k-\ell_\gb)
+O\bigpar{N\qw}.
\end{equation}
It follows by induction on $k$, that for all large $N$, 
$i_k$ is the winner of the $k$-th seat by the
method defined above, provided there are no ties in that method.
\end{proof}

This limit method seems more interesting theoretically than for practical
use.
In particular, the following example shows that
in the party list case, the limit method does \emph{not}
reduce to D'Hondt's method.

\begin{example}[The limit method in a party list case]\label{Elimit}
  Consider the simplest case with two parties.

Unordered ballots. The limit of \phragmen's method defined above.
\begin{val}
\item [$a$] A\sss1A\sss2A\sss3\dots
\item [$b$]B\sss1B\sss2B\sss3\dots
\end{val}

Suppose that $a>b$, and let $k:=\floor{a/b}$.
Furthermore, to avoid ties, suppose that $a/b$ is not an integer, so
$kb<a<(k+1)b$.
It is easy to see that the limit method above will elect in the periodic
order
A\sss1,\dots,A\sss{k},B\sss1,
A\sss{k+1},\dots,A\sss{2k},B\sss2,\dots
with $k$ A's followed by a single B, repeatedly.

In particular, if $b<a<2b$, 
and the number of seats is
even, then the two parties get an equal number of seats, in spite of the
fact that one
party has more votes than the other.
For example, with $a=5$, $b=3$ and $s=4$ as in \refE{EfullAB}, 
the elected are A\sss1B\sss1A\sss2B\sss2, while D'Hondt's method would
give
A\sss1B\sss1A\sss2A\sss3.
\end{example}

\section{Proportionality}\label{Sprop}

There is no precise definition of proportionality for an election method,
and whether a method is proportional or not is partly a matter of degree
(and perhaps taste).
The general idea of proportionality is that a sufficiently large group of
voters gets a ``fair'' share of the elected. When there are no formal
parties, as in the election methods considered here, we should here consider
arbitrary groups of voters, possibly with some restriction on how they vote.
This can be made precise in many different ways. We study in this section
some such properties relevant for \phragmen's and Thiele's methods.
%In \refSS{SPropCrit} we discuss one type of proportionality criterion, which
%is satisfied by \phragmen's methods (in slightly different versions in the
%unordered and ordered cases) but not by Thiele's.
%In \refSS{SPropThr}, we take a related but
%more quantitative approach and define
%several measures of proportionality in the form of different
%thresholds for representation.

%\subsection{Criteria}\label{SPropCrit}

\phragmen's methods satisfy the following proportionality criteria.

\begin{theorem} \label{TPhPC}
\phragmen's methods satisfy the following  properties,
where $V$ is the number of votes, $\MM$ the number of seats and $\ell$
is an arbitrary integer with $1\le \ell\le \MM$.
  \begin{romenumerate}
  \item \label{TPhPCu}
For \phragmen's unordered method:
If a set of more than $\frac{\ell}{\MM+1}V$ voters vote for the same list
containing at least $\ell$ candidates, then at least $\ell$ candidates from this
list are elected.
  \item \label{TPhPCo}
For \phragmen's ordered method:
If a set of more than $\frac{\ell}{\MM+1}V$ voters 
vote with the same $\ell$ candidates first on their
ballots, but not necessarily in the same order,
then these $\ell$ candidates  are elected.
In particular, if at least $\frac{\ell}{\MM+1}V$ voters 
vote for the same list
(containing at least $\ell$ candidates),
then the $\ell$ first candidates on this list will be elected.
 \end{romenumerate}
\end{theorem}

\begin{proof}
Let $\cA$ be such a set of %more than
$|\cA|>\frac{\ell}{\MM+1}V$ voters, and
let $\cC$ in case \ref{TPhPCu} be the set of  candidates on the
list voted for by $\cA$ (thus $|\cC|\ge\ell$),
and in case \ref{TPhPCo} the set of the $\ell$ candidates that are first on the 
ballots from $\cA$ (thus $|\cC|=\ell$).
Suppose, to obtain a contradiction, that only  $\ellx<\ell$ of the
candidates in $\cC$
are elected.

We use the formulation with voting power
in Sections \ref{SPhru1} and \ref{SPhro1}, and let $t=t\xx \MM$ be the final
voting power of each ballot.

In the unordered case \ref{TPhPCu}, the $\ellx$ elected candidates from $\cC$
are together assigned voting power $\ellx$, 
of which some part may come from
voters not in $\cA$. The total free (unassigned) voting power of the ballots in
$\cA$
is at most 1, since otherwise another candidate from $\cC$ would have been
elected. (This free voting power may equal 1 if there was a
tie for the last seat.)
Hence, the total voting power of the ballots in $\cA$ is at most $\ellx+1\le
\ell$,

In the ordered case \ref{TPhPCo}, 
since at least one candidate in $\cC$ is not elected, each ballot in $\cC$
has when the election finishes a current top
candidate that belongs to $\cC$. Hence,
if we regard the free voting power of a ballot as assigned to its current
top candidate, then
the voting power of each ballot in $\cA$
is fully assigned to one or several candidates in $\cC$ (elected or not).
Since each candidate is assigned a total voting power at most 1,
and $|\cC|=\ell$,
the total voting power of the ballots in $\cA$ is at most $\ell$ in this case
too.

Since there are $|\cA|$ ballots in $\cA$ and each has voting power $t$, this
shows that in both cases
\begin{equation}
  \ell \ge |\cA|t>\frac{\ell}{\MM+1}Vt
\end{equation}
and thus 
\begin{equation}\label{tv}
  Vt < {\MM+1}.
\end{equation}

On the other hand, $\MM-\ellx$ candidates not in $\cC$ have been
elected,
and thus a total voting power $\MM-\ellx\ge \MM+1-\ell$ has been assigned to
them. 
Since the ballots in $\cA$ do not contribute to this, all this voting power
comes from the $V-|\cA|$ other ballots, and thus
\begin{equation}
  \MM+1-\ell \le (V-|\cA|)t <\Bigpar{V-\frac{\ell}{\MM+1}V}t = \frac{\MM+1-\ell}{\MM+1}Vt.
\end{equation}
However, this contradicts \eqref{tv}.
\end{proof}

\begin{corollary}\label{CPhPC}
\phragmen's unordered and ordered methods satisfy:
If  more than half of the voters vote for the same list
containing at least $\MM/2$ candidates, 
%or (more generally) vote with the same $\ceil{\MM/2}$ candidates first on their
%ballots, but not necessarily in the same order,
then at least $\MM/2$ of these  are elected.
\end{corollary}
\begin{proof}
  By \refT{TPhPC} with $\ell=\ceil{\MM/2}\le{(\MM+1)/2}$.
\end{proof}
In particular, if $\MM$ is odd, a majority of the voters will, if they vote
coherently, get a majority of the seats.

Note that \refT{TPhPC} 
does \emph{not} hold for Thiele's
unordered and ordered methods, as is seen by \refE{Etactic}
(the second election, where KLM has $13>V/(\MM+1)=50/4$ votes but does not
get any seat) 
and
\refE{Etactic-o}
(the second election, where CD has $39>V/(\MM+1)=100/3$ votes but does not
get any seat).
Also  \refC{CPhPC}  does not hold for Thiele's unordered and ordered
methods,
as is seen by the Example \ref{Etactic-o2} and the following example:

\begin{example}\label{Ega3-5}
Consider the following election:

Unordered ballots. 5 seats. Thiele's method.
\begin{val}
\item [9]A\sss1A\sss2A\sss3
\item [2]X\sss1X\sss2
\item [2]X\sss1X\sss3
\item [2]X\sss2
\item [2]X\sss3
\end{val}
The first two seats go to A\sss1 and A\sss2 (say).
Then X\sss1, X\sss2 and X\sss3 tie for the third seat (4 votes each), 
and  if X\sss1 is elected,  then A\sss3,
X\sss2 and X\sss3 tie for the remaining 2 seats (3 votes each).
Hence, depending on tie-breaks,
X\sss1, X\sss2 and X\sss3 may be elected to the last 3 seats, and thus
the A party, which has a majority of the votes (9 of 17)
 may end up with only 2 seats.
(In this and similar examples, it is easy to modify the example to avoid ties
by multiplying all votes by some large number and then adding a few votes on
suitable ballots, giving an arbitrarily small change of the proportions.)
\end{example}

\begin{remark}
  \label{RPhPCbest}
  The proportion of (more than) $\frac{\ell}{\MM+1}$ of the voters is exactly
  the proportion required to get $\ell$ seats of $\MM$ in an election with
  two parties by D'Hondt's method, as is easily seen.
%(See the proof of \refT{TpiParty} below.)
Hence, by \refT{TDHondt}, this proportion  in
\refT{TPhPC} is the best possible.
\end{remark}

\begin{remark}\label{RDroop}
The proportion $\frac{\ell}{\MM+1}$
is less that $\frac{\ell}{s}$, which often is regarded
as a proportion that ought to guarantee at least $\ell$ seats. 
The proportion $\frac{\ell}{\MM}$
might be more intuitive, but as pointed out already by
\citet{Droop}, in an election of $\MM$ representatives
using simple plurality
\xfootnote{E.g.~by SNTV (\refApp{ASNTV});
Droop really discussed Cumulative Voting (\refApp{ACV}), but SNTV can be seen
as a special case, and the difference between the methods does not matter
here.
},
if some candidate A has more than $\frac{1}{\MM+1}$ of the
votes, 
then there cannot be $\MM$ other candidates that have at least as many  votes, 
%and thus having better claims to be elected, 
and thus A is elected. The proportion
$\frac{1}{\MM+1}$ of the votes (called the \emph{Droop quota}, 
see \refApp{Alist}) is
thus a  
reasonable
requirement for having the ``right'' to be elected. And a party with
more than $\frac{\ell}{\MM+1}V$ votes could in principle split them on
$\ell$ candidates that would be elected.
Note also that in the case $\ell=\MM=1$, $\frac{1}{\MM+1}=\frac12$, so the
criterion, and \refT{TPhPC}, reduces to saying that a candidate supported by
a majority is elected.
\end{remark}

\begin{remark}\label{RDPC}
STV (\refApp{ASTV}) satisfies a property similar to \refT{TPhPC},
called the \emph{Droop Proportionality Criterion} by \citet{Woodall:properties}.
The version in \cite{Woodall:properties} assumes that the Droop
quota $Q_D:=V/(\MM+1)$ is used without rounding, 
\xfootnote{
For STV with a quota $Q> Q_D$
(recall that $Q\ge Q_D$ almost always holds in practice),
% (which almost always is  used), 
the property still holds 
  in the case $\ell=1$, but not in general; however, it holds for sets of
  at least $\ell Q$ voters. See \eg{} \cite[Section 12.5.1]{SJV6}.
}
and then says:

\emph{If $m\ge\ell$ and
a set of more than
$\ell V/(\MM+1)$ voters put the same $m$ names first on their ballots, but
not necessarily in the same order, then at least $\ell$ of these will be
elected.} 

Note that \refT{TPhPC}\ref{TPhPCo} is this property in the special
case $m=\ell$. (So the Droop Proportionality Criterion is stronger.) 
The fact that the property holds for STV also for arbitrary
$m>\ell$ depends on the eliminations of weak candidates, which will
concentrate the votes of this set of voters to smaller sets of candidates
until at
least $\ell$ of them are elected. 
For \phragmen's ordered method there is no similar
mechanism,
and splitting the votes on more than $\ell$ candidates can lead to less than
$\ell$ of them being elected; one example is given in \refE{Esplit} below.
This is in contrast to the unordered method, see \refT{TPhPC}\ref{TPhPCu},
where splitting the vote on several
candidates does not decrease the support of any of them (except when one of them
is elected).
\end{remark}

\begin{example}[\phragmen's ordered method]
  \label{Esplit}
\quad

Ordered ballots. 3 seats. \phragmen's method.
\begin{val}
\item [9]ABC
\item [9]ACB
\item [9]BAC
\item [9]BCA
\item [9]CAB
\item [9]CBA
\item [46]XYZ
\end{val}
With \phragmen's ordered method, the first two seats go to X and Y (\jfr{s}
46 and 23), while the third to any of A, B or C (\jfr{} 18).
Elected: AXY.

The ABC party has together 54 of the 100 votes, and if they all voted, say,
ABC, then they would be guaranteed two seats by \refT{TPhPC} (with $\ell=2$
and $\MM=3$) or by \refC{CPhPC}. (In fact, it suffices that they all vote
with ballots beginning with AB or BA.) However, by splitting their votes,
they lose one seat, and thus the majority.
%
%The same example, but with only 2 seats, give an example with $\ell=1$.
\end{example}

%\subsubsection{Justified representation, $\pi\EJR$ and $\pi\PJR$}

For election methods with unordered ballots, 
\citet{EJR} defined two properties \emph{\qJR{} (justified representation)}
and (stronger) \emph{\qEJR{} (extended justified representation)};
\citet{SanchezEtAl} then defined a related property
\emph{\qPJR{} (proportional justified representation)}
such that \qEJR$\implies$\qPJR$\implies$\qJR.
These are defined as follows, where $\VV $ is the total number of votes, and
$\cA$ is a set of voters.

\emph{\qPJR: Let $1\le\ell\le\MM$ and suppose that
$|\cA|\ge \frac{\ell}{\MM} \VV $ and that 
$\Bigabs{\bigcap_{\gs\in\cA}\gs}\ge\ell$,
\ie, the voters in $\cA$ all vote for a common set $\cC$ of at least $\ell$
candidates, but may also vote for other candidates.
Then
$\ell$ candidates that some voter in $\cA$ has voted for are elected,
\ie{} $\Bigabs{\cE\cap\bigcup_{\gs\in\cA}\gs}\ge\ell$.
}

\emph{\qEJR: Let $1\le\ell\le\MM$ and suppose that
$|\cA|\ge \frac{\ell}{\MM} \VV $ and that 
$\Bigabs{\bigcap_{\gs\in\cA}\gs}\ge\ell$,
\ie, the voters in $\cA$ all vote for a common set $\cC$ of at least $\ell$
candidates, but may also vote for other candidates.
Then some voter in $\cA$ has voted for 
$\ell$ candidates that are elected,
\ie{} $\bigabs{\cE\cap\gs}\ge\ell$ for some $\gs\in\cA$.
}

\emph{\qJR: The condition above for PJR (or EJR) holds for the special case $\ell=1$.
}
(Note that the difference between $\qPJR$ and $\qEJR$
disappears
for $\ell=1$.)

We state some results from the papers above (and \cite{BrillEtAl}) without
proofs.

\begin{theorem}[\cite{EJR}]\label{Tth-EJR}
 Thiele's \opt{} method satisfies \qEJR.    

However, with any  other satisfaction function instead of the proportional 
\eqref{f}, \qEJR{} fails. 

Thiele's \opt{} method with 
a general satisfaction function \eqref{fw} with $w_1=1$ satisfies \qJR{} if
and only if $w_j\le 1/j$ for every $j>1$.
\end{theorem}

\begin{theorem}[\cite{SanchezEtAl}]
 Thiele's addition method satisfies \qJR{} for $\MM\le5$, but not for $\MM\ge6$.
\end{theorem}

\begin{theorem}[\cite{EJR}]
 Thiele's addition method with a general satisfaction function $f(n)$ satisfies
 \qJR{} if and only if  $f(n)$ is the weak satisfaction function \eqref{fweak}.
\end{theorem}

\begin{theorem}[\cite{BrillEtAl}]
  \phragmen's unordered method satisfies \qPJR, but not \qEJR.
\end{theorem}

\section{\phragmen's ordered method and STV}\label{SPhSTV}
\phragmen's and Thiele's ordered methods use ordered ballots just as the
more well-known election method STV (\refApp{ASTV}).
Both methods are clearly different from STV, which is seen already in the
case $\MM=1$ of electing a single person, when they reduce to simple plurality
(\refT{T1}), while STV reduces to Alternative Vote.
Nevertheless, there are strong connections between \phragmen's method and
STV; both methods can be regarded as giving each ballot a certain voting
power that can be used to elect the candidates on the ballot (one at a time,
in order), and
when the voting power of a candidate is more than enough to elect the
candidate, the surplus is transfered to the next candidate.
 Thiele's method seems to be founded on different principles and will not be
 considered further in this section.

To see the connection in more detail, 
consider an election of $\MM$ seats by \phragmen's ordered method, using the
formulation in \refSS{SPhro1} (and the notation in \refSS{SPhru1}).
We thus give each ballot voting power $t\xx \MM$ (found at the end of the
calculations).
We obtain an equivalent formulation of the method where each ballot has
voting power 1 if we define
$Q:=1/t\xx \MM$ and multiply every voting power by $Q$.
This means that every ballot has voting power 1, and that a candidate needs
voting power $Q$ to be elected, just as in STV if the quota is $Q$.
Note that by \eqref{Wntn}, $Q=W\xx \MM$ is the \jfr{} for the last elected
candidate if we use the formulation in Sections \ref{SPhro2}--\ref{SPhro3}.
 
Let us use the algorithm in \refS{SPhro2}, and let $i_j$ be the candidate
elected in round $j$, $j=1,\dots,\MM$.
Consider a ballot $\gb$ in some round $k\ge1$, and
let $\ell<k$ be the latest round before $k$ when the ballot 
participated in the election of a candidate
(meaning that the elected candidate $i_\ell$ was the current top name on the
ballot  that round); if the ballot has not yet participated in the election
of any candidate let $\ell=0$.
Then the place number of the ballot in round $k$,
which is the voting power assigned to the already elected candidates, is
$t\xx\ell$, with $t\xx0=0$
(see Remarks \ref{Rplace}--\ref{Rsheppard}).
In the new scale, the ballot has used
$Qt\xx\ell$ of its value to the already elected, and its remaining value
is $1-Qt\xx\ell$. (Note that $Qt\xx\ell=t\xx\ell/t\xx \MM\le1$.)

In round $k$, $i_k$ is elected. Suppose that then there are $n_k$ ballots
currently valid for $i_k$ (\ie, with $i_k$ as the current top name), and that
of these $n_{k\ell}$ have participated in the election of $i_\ell$, but not in
the election in any later round, $\ell=1,\dots,k-1$, while $n_{k0}$ of them
have not participated in the election of any candidate before $i_k$.
Thus, the number of ballots valid for $i_k$ is
\begin{equation}\label{nk}
  n_k=\sum_{\ell=0}^{k-1} n_{k\ell},
\end{equation}
and their total place number is
\begin{equation}\label{pk}
  P_k:=\sum_{\ell=1}^{k-1} n_{k\ell}t\xx \ell
=\sum_{\ell=0}^{k-1} n_{k\ell}t\xx \ell.
\end{equation}
By the formulation of \phragmen's method in \refS{SPhro1},
if each ballot has voting power $t\xx k$, then the elected candidate $i_k$
will have voting power 1 from the $n_k$ ballots valid for $i_k$, \ie,
subtracting the voting power assigned to previously elected candidates,
\begin{equation}\label{p12}
  1 = n_kt\xx k - P_k
.%=\sum_{\ell=0}^{k-1} n_{k\ell}(t\xx k-t\xx\ell).
\end{equation}

On the other hand, if we consider the values of the ballots
in the new scale defined above, 
then the total remaining value of the $n_k$ ballots valid for
$i_k$ is, using \eqref{nk} and \eqref{pk},
\begin{equation}\label{p12b}
 \sum_{\ell=0}^{k-1} n_{k\ell}(1-Qt\xx\ell)
=
 \sum_{\ell=0}^{k-1} n_{k\ell} - Q  \sum_{\ell=0}^{k-1} n_{k\ell}t\xx\ell
=n_k-QP_k.
\end{equation}
Let us, as in STV, subtract $Q$ (which is used to elect $i_k$);
then there remains a surplus,
using \eqref{p12},
\begin{equation}
  \label{p13}
n_k-QP_k-Q=n_k-Q(P_k+1)
=n_k-Qn_kt\xx k
=n_k(1-Qt\xx k).
\end{equation}
(Since $Qt\xx k=t\xx k/t\xx \MM\le1$, this surplus is $\ge0$, which shows that
the total value \eqref{p12b} is at least $Q$.)   After the election
of $i_k$, these ballots have, as we have seen, the value $1-Qt\xx k$ each,
which equals the surplus \eqref{p13} divided equally between the $n_k$
ballots that participated in the election of $i_k$.  This is exactly as in
the inclusive Gregory method for STV (\refApp{ASTV}\ref{STV-inclusive}).

We summarize in a theorem.
\begin{theorem}\label{TPhSTV}
Consider an election with ordered ballots.
and  let $Q:=W\xx \MM$, the \jfr{} for the last elected using
  \phragmen's method.
Then \phragmen's method yields the same result as STV with the inclusive Gregory
method and this quota $Q$, provided that the surpluses are transfered in the
order the 
candidates are elected by \phragmen's method.
\end{theorem}

We see also that every elected will reach the quota $Q$, so that there will
be no eliminations; in fact, the final surplus $n_\MM(1-Qt\xx \MM)=0$, and thus  
$Q$ is the largest quota that allows $\MM$ candidates to be elected without
eliminations (still
provided the surpluses are transfered in the prescribed order).

Consequently,
\phragmen's ordered method can be seen as a variant of STV without
eliminations, and with a quota $Q$ that is calculated dynamically instead of
as a fixed proportion of the number of votes. 
(The latter difference is the same as the
difference between divisor and quota methods for list election methods,
see  \refFn{fDivisor} in \refApp{Alist}. Thus \phragmen's method can
be seen as a variant of STV  related to divisor methods, which is not
surprising since 
\phragmen's method was conceived as a generalization of D'Hondt's method.)

However, this correspondence with STV is not perfect, since STV is sensitive
to the order the surpluses are transfered, and standard versions of STV may
do this in a different order than \phragmen's method yields.

\begin{example}[\phragmen's method and STV]\label{EPhragmen-stv}
\quad

Ordered ballots.
  4 seats. 
  \begin{val}
  \item [22]ABCD
  \item [11]ABE
  \item [11]CE	
  \end{val}
Using \phragmen's method, candidates are elected in the following order:
\begin{enumerate}
\item 
A is elected with  \jfr{} $W_\rmA=33$.
\item 
B is elected with \jfr{}  $W_\rmB=\frac{33}2$. 
\item 
The three groups of ballots now have the place numbers
$\frac43$, $\frac23$, $0$ (with sum 2), which give C and E the \jfr{s}
\begin{align*}
  W_\rmC=\frac{22+11}{1+4/3}=\frac{99}7
\qquad\text{and}\qquad
  W_\rmE=\frac{11}{1+2/3}=\frac{33}5.
\end{align*}
Thus C is elected.
\item 
The three groups of ballots now have the place numbers
\begin{align*}
  \frac{22}{W_\rmC}
=   \frac{22}{99/7}
=\frac{14}9,
\qquad
\frac{11}{W_\rmB}=
\frac{11}{33/2}=
\frac23=\frac{6}9,
\qquad
\frac{11}{W_\rmC}
= \frac{11}{99/7}
=\frac{7}9
\end{align*}
(with the sum 3), which gives D and E the \jfr{s}
\begin{align*}
  W_\rmD=\frac{22}{1+14/9}=\frac{22\cdot 9}{23}
\qquad\text{and}\qquad
  W_\rmE=\frac{22}{1+13/9}=\frac{22\cdot 9}{22}=9.
\end{align*}
Thus E is elected.
\end{enumerate}
Elected: ABCE.

Consider now instead STV with the inclusive Gregory method and
quota  $Q=W_\rmE=9$ as in \refT{TPhSTV}. 
($Q$ happens to be the Droop quota rounded upwards in this
example. We could also obtain the Droop quota without rounding by adding a
single 
vote on F. In general, there is no direct relation between $Q$ in
\refT{TPhSTV} and the Droop quota.) 
By \refT{TPhSTV}, we will obtain the same elected ABCE if the surpluses are
transfered in this order.
This is also easy to see directly:
A has surplus 24, which is transfered to B. B has surplus 15 which is
transfered with 10 to C and 5 to E. C has surplus 12 of which is 8 is transfered
to D and 4 to E. E now has exactly $Q=9$ votes and gets the fourth seat,
while D only has 8 votes.
\xfootnote{
In this example, and assuming that the surpluses are transfered in this
order,
the result will be the same for every quota $Q\le99/7=14\frac{1}7$,
although if  $Q>9$ then $E$ will not reach the quota and if $Q\le
198/23\doteq8.61$ 
then both  $\rmD$ and $\rmE$ will reach the quota, although E wins.}

However, in standard versions of STV, such as the inclusive Gregory method 
used for the Australian Senate, %\cite[Section 273]{AUvallag},
the surpluses are transfered in a different order:
$\rmA$ and $\rmC$ reach the quota  on their first-preference
votes and are elected.
Their surpluses  24 and 2 are transfered to B and E, respectively.
This makes B reach the quota, so B is elected to the third seat,
with a surplus 15 that is transfered with 10 to D and 5 to E.
Consequently, D gets the final seat with 10 votes against 7 for E.
Elected: ABCD.
\end{example}

\begin{remark}
\refT{TPhSTV} yields also the following 
(rather complicated) description of \phragmen's method purely in
  terms of STV:
With $\MM$ seats, do $\MM$ distributions of seats by STV with the inclusive
Gregory method, the first time for 1 seat, the second for 2 seats, and so
on.
In the $k$-th distribution, the surpluses of the $k-1$ already elected shall
be transfered in the order that these have been elected, and the quota $Q_k$
shall be chosen such that these $k-1$ will reach the quota, and that (after
all transfers), there will be another candidate that reaches exactly $Q_k$,
so this candidate is elected to seat $k$.
\end{remark}

\section{Some variants of \phragmen's and Thiele's methods}\label{Svar}
\subsection{Unordered ballots with two groups}\label{SSline}
 As said in \refS{Sdecap}, a problem with unordered ballots is the risk of
 decapitation, which led to the introduction of the ordered versions of
 \phragmen's and Thiele's methods. 

Another modification to avoid decapitation 
is to let the names on each ballot be separated into two groups 
(instead of a complete ranking), with preference to the first group but no
ordering inside each groups.
(This modification too 
applies to both \phragmen's and Thiele's methods, as well as to some
other similar methods with unordered ballots.)
\begin{metod}{Unordered ballots with two groups}
  Each ballot has two groups of names. 
The orders inside the two groups do not matter. 
A ballot is regarded as a vote for only the names in the first group
as long as at least one of these names remains unelected; when all names
in the first group are elected, the ballot is regarded as a vote for all
remaining names.
(The second group might be empty; equivalently, ballots with a single group are
also accepted.)
\end{metod}
The names in the second group can thus be seen as reserves, that are not used
until required.

This modification  was described 
for \phragmen's method
(and attributed to \phragmen)
by \citet{Cassel} in 1903, 
and
 was proposed for Thiele's method in parliament in 1906 
(by Petersson in P\aa boda) %; it was not adopted)
\cite[p.~20]{bet1913}.

In practice, the groups might be shown on the ballot by, for example
drawing a line between the two groups of names \cite{Cassel}, 
or by underlining the
names in the first group \cite{bet1913}.

\subsection{Weak ordering}\label{SSweak}

To have two groups on a ballot is an intermediary between unordered and
ordered voting. More generally, a ballot could be allowed to have a
set of candidates with an arbitrary \emph{weak ordering} 
(= \emph{total preorder}), i.e., a
list of groups of candidates, with the different groups in order of preference
but ignoring the order inside each group.
\xfootnote{This version was, as far as I know, not proposed by \phragmen,
  Thiele or any of their contemporaries.
}

\begin{metod}{Weakly ordered ballots}
  Each ballot contains a weakly ordered set of names.
In practice, this could be a list of names 
separated into groups by one or
  several lines (or a single group without a line); 
the orders inside the  groups are ignored.
A ballot is regarded as a vote for only the names  
in the first group where there is someone that is not yet elected.
\end{metod}

Note that this version includes (as extreme cases)
both the unordered version (every ballot
contains only a single group) and the ordered version (every ballot contains
a total ordering), as well as the version with two groups in \refSS{SSline}.
The commission report \cite[p.~20]{bet1913} briefly mentions (and
dismisses as impractical) also the possibility
of having some names on a ballot ranked, and the remaining names coming
after these but unordered among themselves.

Ballots with weakly ordered sets of candidates have been discussed in the
context of STV (\refApp{ASTV}), see e.g.\ 
\citet{Meek2} and \citet{Hill:equality}, and
such ballots have been used in STV elections in some organizations
 \cite{Hill:equality}.
However, the method discussed in
\cite{Meek2} and \cite{Hill:equality}
to handle weak orderings differs from \phragmen's:
they regard a weak ordering as a vote split equally between all total
orderings compatible with the weak ordering.
On the other hand, \phragmen's version in \refSS{SSline}, extended to weak
orderings as above, means that
each ballot is counted as a full vote for each candidate in its first group,
as long as none of them is elected, cf.\ the
  principle in \ref{Pufull}.
For example, a ballot beginning with the group (AB) is, as long as neither A
nor B is elected, regarded as
$\frac12$ vote for A and $\frac12$ for B
by \cite{Meek2} and \cite{Hill:equality}, but as 1 vote for A and 1 vote for
B by \phragmen.
(If A or B is elected, then the value of the ballot is reduced for the
remaining candidate in both systems, by different mechanisms.)

It would be interesting to compare the two ways to handle weakly ordered
list.
It seems that \phragmen's principle can be applied to STV elections too, and
it might have some advantages over splitting the vote between total
orderings
as described above.
(Cf.~\refSS{SSPh1} for the  case of unordered ballots.)

\subsection{\phragmen's method recursively for alliances and factions}
\label{SSrecur}

In Swedish elections 1924--1950, 
ballots could (but did not have to) contain not only a party name but also
an alliance name and a faction name (in addition to the names of the
candidates),
see \refApp{Ahistory}; a ballot could thus
be labelled in up to three levels.
(Only the party name was compulsory.) 
When ballots with alliance and faction names were used, \phragmen's method
was used recursively in up to three steps:
(D'Hondt's method was used to distribute seats between alliances and parties
outside alliances.)

\begin{metod}{\phragmen's method with alliances and factions}
Ballots are ordered, and contain a party name and possibly an optional
alliance name and an optional faction name.

  First, \phragmen's method is applied to each faction separately.
Thus, for each faction name, the method is used to determine an ordering of
the candidates on the ballots with that faction name; call this ordering the
faction list.
In the sequel, all ballots with this faction name are regarded as containing
the faction list instead of their original lists of names.

Secondly, \phragmen's method is applied to each party.
For each party name, the method is used to determine an ordering of
the candidates on the ballots with that party name;
call this ordering the party list.
In the sequel, all ballots with this party name are regarded as containing
this party list instead of their original lists of names (or the faction list).

Thirdly, \phragmen's method is applied to each alliance.
For each alliance name, the method is used to determine an ordering of
the candidates on the ballots with that alliance name:
call this ordering the alliance list.

Seats given to an alliance or a separate party are assigned to candidates
according to the alliance list or party list.
\end{metod}

This means that if no candidate appears on the lists for more than one party
or faction, then (by \refT{TDHondt})
the seats given to an alliance are distributed between the
participating parties according to D'Hondt's method, 
and similarly the seats given to a party are distributed between the
factions according to D'Hondt's method; finally \phragmen's method is
applied for each faction (if there are different ballots within the faction).
(Since D'Hondt's method also was used for the distribution of seats between
alliances and separate parties, this gives a nice consistency.)
However, the method above in an elegant way handles also case where the same
name appears on ballots from different factions, or even different parties.
Moreover, the method above handles cases where some but not all ballots have
a faction name.

\subsection{Party versions}\label{SSMora-party}
\citet{MoraO} discuss a variant of \phragmen's unordered method where each
ballot contains a set of \emph{parties} instead of candidates.
The seats are distributed to the parties one by one as in \phragmen's
method, with the difference that a party can receive more than one seat, and
thus parties that have received a seat 
are not ignored in the sequel.

\begin{metod}{\phragmen's method for parties}
  Each ballot contains a set of parties.
The seats are distributed as in Section \ref{SPhru1} or \ref{SPhru2}, but
parties that have received seats 
continue to participate.
\end{metod}

It is easily seen that this extension of \phragmen's method also can be seen
as a special case of it: if 
each party has a set of (at least) $\MM$ candidates,
with these sets  disjoint,
%party A has candidates A\sss1,\dots, A\sss{\MM}
%party B candidates B\sss1,\dots,B\sss{\MM}, and so on, 
and we on each ballot
replace each
party by its set of candidates, then an election by
\phragmen's (unordered) method will give the same number of seats to each
party as the party 
version above for the original ballots with parties. (All candidate within
the same party would obviously tie, so the choice of elected within each
party would be uniformly random.) 
Hence the party version is equivalent to assuming that each party has a list
of candidates (with at least $\MM$ names), and that each voter votes for some
union of party lists, \ie, if the voter votes for one candidate from some
party, he or she also votes for all other from the same party.
(We ignore here that in the case of a tie, the probabilities for the
different possible outcomes may be different.)

A party version of Thiele's method can be defined in the same way.

The party version of \phragmen's method
 has interesting and surprising mathematical properties, see \citet{MoraO};
 these will be 
 further studied elsewhere.
(The corresponding party version of Thiele's method
is much better behaved.)

\subsection{\phragmen's first method (Eneström's method) -- 
STV with unordered ballots}
\label{SSPh1}

\citet[pp.~47--50]{Cassel} describes what he calls ``\phragmen's first
method'', which is a version of STV (\refApp{ASTV}), but much resembles
\phragmen's later method described above (which is called ``\phragmen's
second method'' in \cite{Cassel}).
The same method was earlier described by \citet{Enestrom} in 1896.
\xfootnote{\label{fEnestrom}%
The history is murky.
\phragmen{} seems to have invented this method c.~1893;
it is similar to his discussion and examples in the newspaper article
  \cite{Phragmen1893} from 1893, 
which clearly describes the weighted inclusive Gregory method,
but of his two examples in \cite{Phragmen1893}
one uses unordered
  ballots and the other ordered ballots (without comment), and at least one
  uses the Droop quota, so the method was presumably not completely
  developed yet. 
As far as I know, \phragmen{} never published anything about his ``first
method''.
On the other hand, the method is described in detail 
(including examples)
by \citet{Enestrom}
in 1896, together with
\phragmen's and Thiele's methods;
Eneström calls it ``my method'', and does not mention \phragmen{} in
connection with this method.
He had also a few months earlier
briefly described the method in a letter to a newspaper
\cite{Enestrom-AB1896}, as a simpler alternative to \phragmen's method.
\citet[pp.~29--31]{Flodstr} calls it  ``Eneström's method''.
Nevertheless, in 1903, \citet{Cassel} attributes the method to \phragmen,
without mentioning or giving a reference to Eneström.
It seems improbable that this is a mistake by Cassel, since \cite{Cassel} is
an appendix to the commission report \cite{bet1903}, and
\phragmen{} was 
one of the members of the commission.
Moreover, it seems obvious that \phragmen{} and Cassel must have seen
\cite{Enestrom}, which was published in the proceedings of
the Royal Academy of Science in the same volume as \cite{Phragmen1896}.
It is perhaps possible that Cassel regarded \phragmen{} as having invented
the method first and
deliberately ignored \citet{Enestrom}.
Note also that \citet{Enestrom} and \citet{Cassel} illustrate the
method by the same example (\refE{EPhr1} below), earlier
used by \citet{Phragmen1894,Phragmen1895,Phragmen1899} for his method;
\citet{Enestrom} and \citet{Cassel}  
even round the numbers in the calculations in the same way, but they use
different labels for the candidates.

Maybe the method ought to be called \emph{Eneström's method}?
At least \citet{Enestrom} seems to be the first publication of it.

Gustaf Eneström (1852--1923) was a Swedish mathematician.
He did (as both \phragmen{} and Thiele, see \refApp{Abio}) 
work in Actuarial Science, 
but his main contributions are to the history of mathematics,
where he published 
\emph{Bibliotheca mathematica},
initially an appendix to Mittag-Leffler's \emph{Acta Mathematica}, but
1888--1913 his own  independent journal.
(Eneström was assistant  to Mittag-Leffler, helping with
\emph{Acta Mathematica}, but there occured a break between them in 1888, and
Eneström was replaced  by \phragmen{}, who became a coeditor.)
Eneström 
 did not have a university position and worked as a librarian.
I do not know anything about his personal relations with \phragmen,
but they were possibly strained.
Ene\-ström was an outsider in Swedish mathematics and
was disappointed that Mittag-Leffler and other established
mathematicians looked down upon his work on the history of mathematics.
\cite{Domar}, \cite[Gustaf Hjalmar Eneström]{SBL}.
}
Note that the method, unlike all other versions of STV that I know of, older
and newer, uses unordered ballots; it follows the principles
\ref{Pu}--\ref{Pufull} in \refS{Sunordered}.

\begin{metod}{\phragmen's first method (Eneström's method)}
Unordered ballots.
  Each ballot has initially voting power $1$.
Each ballot is counted fully, with its present voting power, for each
unelected candidate on the ballot.
Let $Q$ be the Hare quota (see \refApp{Alist}).
For each seat, the candidate is elected that has the largest sum of voting
powers (from all ballots that contain the candidate's name).
If this total voting power is $v$, and $v>Q$, then each of these ballots has
its voting power multiplied by $(v-Q)/v$. If $v\le Q$, then these ballots
all get voting power $0$ (and are thus ignored in the sequel).
This is repeated until the desired number of candidates are elected.
\end{metod}

Note that the total voting power of all ballots is decreased by $(Q/v)\cdot
v=Q$ each time, as long as someone reaches the quota.

The method is presented in \cite{Cassel} as \phragmen's improvement of
\emph{\andrae's method} (\refApp{ASTV}), resolving two major
problems with the latter:

First, since \andrae's method uses ordered ballots, and only counts the
first unelected candidate on each (except possibly at the end), 
as in \ref{Po}--\ref{Potop} above,
a party that
gets many votes (maybe several times the quota)
but with its candidates in different orders may not get any seat at all.
\phragmen{} thus solves this by using unordered ballots.
(In modern versions of STV, with ordered ballots, this problem is solved by
eliminations, see 
\refApp{ASTV}.)
\xfootnote{There is no point in adding eliminations to \phragmen's version
  with unordered ballots, since every candidate on a ballot gets its full
  remaining   voting power, regardless of whether there are other remaining
  candidates on the ballot or not; hence eliminations would not change the
  sum of voting powers for the remaining candidates.}

Secondly, \andrae's method uses effectively a random selection of the
ballots that are transferred from a successful candidate,
so the outcome may be random. (This is still true for some versions of STV,
see \refApp{ASTV}\ref{STV-Cincinnati}--\ref{STV-Ireland}.)
\phragmen{} resolves this by transferring all ballots but reducing their
voting power proportionally, by what is now known as the \emph{weighted
  inclusive Gregory method}, see \refApp{ASTV}\ref{STV-weighted}, 
which thus was invented by
\phragmen{} (and then forgotten for until reinvented almost a century later).

\phragmen{} seems to have invented this method c.~1893, see \refFn{fEnestrom},
but then  he instead
developed the ideas further 
to the method described in \refS{SPhru}
by eliminating the fixed quota $Q$ and
instead using a variable quota (or, equivalently, as in the description in
\refS{SPhru}, a variable voting power); note also that there are other
modifications, and that the weighted inclusive Gregory method is gone, and
replaced by something similar to the unweighted inclusive Gregory
method in  \refApp{ASTV}\ref{STV-inclusive}, although the mechanism is
different, see \refS{SPhSTV}.

\phragmen{} seems to have returned to new versions of the quota-based method
(with a simplified 
method for reducing the votes) in 1906, according to  some unpublished
notes and drafts,
see \cite[Appendix B.1]{MoraO}.

\begin{example}\label{EPhr1}
\phragmen's first method (Eneström's method)
was illustrated by both \citet{Enestrom} and 
\citet[p.~49]{Cassel} by the
example used by 
\citet{Phragmen1894,Phragmen1895,Phragmen1896} for his method
described in
\refS{SPhru}, and used above in Examples \ref{EPhr1894}, \ref{EPhr1894-Th}, 
\ref{EPhr1894-o}, \ref{EPhr1894-Tho}.
We follow \cite{Enestrom} and \cite{Cassel}, 
and present the calculations using decimal approximations.

Unordered ballots. 3 seats. \phragmen's first method (Eneström's method).
\begin{val}
\item [1034]ABC
\item [519]PQR
\item [90]ABQ
\item [47]APQ
\end{val}
1690 votes. The Hare quota $Q=1690/3\doteq563.3$.

The total numbers of votes for each candidate are
\begin{val}
\item [A]1171
\item [B]1124
\item [C]1034
\item [P]566
\item [Q]656
\item [R]519.
\end{val}

The first seat goes to A, who has the largest number of votes.
Since A has 1171 votes, which exceed the quota $Q$, each ballot containing A
has its voting power (value) reduced from 1 to 
$1-\frac{Q}{1171}=1-\frac{563.3}{1171}\doteq0.519$.

The total voting power of each group of ballots is now 
\begin{val}
\item [ABC] $1034\cdot0.519\doteq536.6$
\item [PQR] 519
\item [ABQ] $90\cdot0.519\doteq46.7$
\item [APQ] $47\cdot0.519\doteq24.4$
\end{val}
(The sum is $\VV-Q=2Q\doteq1126.7$.)

By summing these values for the ballots containing a given candidate,
the voting power that each candidate can collect is:
B 583.3; C 536.6; P 543.4; Q 590.1; R 519.
Since Q has the highest voting power, Q is elected to the second seat.

The voting power 590.1 of Q exceeds the quota Q=563.3, and thus the ballots PQR,
ABQ and APQ get their voting power multiplied by 
$1-\frac{563.3}{590.1}\doteq0.045$.
The total voting power of each group of ballots is now 
\begin{val}
\item [ABC] $536.6$
\item [PQR] $519\cdot0.045\doteq23.5$
\item [ABQ] $46.7\cdot0.045\doteq2.1$
\item [APQ] $24.4\cdot0.045\doteq1.1$
\end{val}
(The sum is $\VV-2Q=Q\doteq563.3$.)

The voting powers available to the remaining candidates are thus:
B 538.7; C 536.6; P 24.6; R 23.5. Thus B is elected to the third seat.

Elected: ABQ.

This is the same result as produced by both \phragmen's and Thiele's
methods, see Examples \ref{EPhr1894} and \ref{EPhr1894-Th}.
\end{example}

See further examples in \cite{Enestrom}.

\subsection{Versions of \phragmen's method based on optimization criteria}
\label{SSopt}
\citet{Phragmen1896} discusses election methods (for unordered ballots) from
the general point of view that each voter should, as far as possible, obtain
the same representation as everyone else.
\xfootnote{It seems likely that this was a kind of response
to \citet{Thiele}, which also tries to derive an election method from
general principles by some \opt{} criterion.
\phragmen{} notes that his classification according to ``inequality''
is different from
from Thiele's based on ``satisfaction'', 
but does not compare the two approaches.
}
More precisely, each elected candidate is counted as one unit, which is
divided between the voters that have voted for that candidate.
This gives each voter a measure of the voter's representation,
and the goal is to keep these as equal as possible.

This ``amount of representation'' is the same as \emph{load}
in \refR{Rload} 
and for convenience we use the latter term here.
(Recall that load also equals  \emph{voting power} in \refS{SPhru}.)
The methods obtained by this approach can thus be described as:
\emph{For each set of $\MM$ candidates, compute the loads for each voter.
Elect the set of candidates that minimizes the ``inequality'' of the load
distribution.} 

In the paper, \phragmen{} discusses several alternatives, as follows.
We use the notation that there are $\VV$ voters, and that the load of voter
$k$ is $\xi_k$; 
thus $\sum_{k=1}^\VV\xi_k=\MM$, the number of elected candidates, and we define
$\bxi:=\sum_{k=1}^\VV\xi_k/\VV = \MM/\VV$, 
the average load per voter; note that $\bxi$ is the same for all sets of
elected (with $\MM$ fixed)
and for all distributions of their loads.
\begin{alphenumerate}[-10pt]
\item 
The total load 1 of each candidate can be:
\begin{enumerate}
\item \label{a1}
divided equally between all voters voting for that candidate,
\item \label{a2}
divided arbitrarily between the voters voting for that candidate
(in a way that minimizes the final inequality).
\end{enumerate}
\item 
The ``measure of inequality'' (that is to be minimized)
for a set $\set{\xi_k}$ of loads can be taken as:
\begin{enumerate}
\item \label{b1}
The sum of squares $\sum_{k,l}(\xi_k-\xi_l)^2$.
(This choice is thus an instance of the general 
method of \emph{least squares}.)
Equivalently, as is well-known (and noted by \citet{Phragmen1896}),
we can minimize the variance
$\VV\qw\sum_{k}(\xi_k-\bxi)^2$
(or just $\sum_{k}(\xi_k-\bxi)^2$),
or simply (because $\bxi$ is fixed) 
the sum of squares $\sum_k \xi_k^2$.

\item \label{b2}
The maximum difference $\max_k(\xi_k-\bxi)$.
This is equivalent to minimizing 
the maximum load  $\max_k\xi_k$.
\end{enumerate}
\item 
Furthermore, \phragmen{} discussed two versions for the optimization 
(the same as two of the three versions discussed by \citet{Thiele}, see
\refS{Sthiu}): 
\begin{enumerate}
\item \label{c1}
The set of $\MM$ candidates that minimizes the inequality is elected.
We call these methods \emph{\opt{}} methods.
\item \label{c2}
The candidates are elected sequentially.
In each round, the existing loads for the previously elected are kept fixed
and for each remaining candidate, the loads for that candidate, if elected,
are added to the previous loads. The candidate that minimizes the
inequality of the resulting set of loads is elected.
(This is thus  a greedy version of the \opt{} method.)
We call these methods \emph{sequential}.
\end{enumerate}
\end{alphenumerate}

This gives $2\cdot2\cdot2=8$ possible combinations
(we shall see below that they all yield different methods), but
\citeauthor{Phragmen1896} does not discuss all of them.

\phragmen{} first says that it is natural to divide the load of a candidate
equally 
between the voters (\ie, \eqref{a1} above), 
but that this cannot be regarded as the definitive answer; he gives the
following example.

\begin{example}%[Unanimity is not enough with \eqref{a1}\eqref{c1}]
\label{EABAC}
 From \citet{Phragmen1896}.
%The number of voters is unspecified there; only that they are equal.

Unordered ballots. 2 seats.
  \begin{val}
\item [100]AB
\item [100]AC
\end{val}
An \opt{} method \eqref{c1} is assumed.
%If two are to be elected, then 
BC will give a uniform distribution of the
load, while with \eqref{a1}, this is impossible if A is elected. Thus the
outcome is BC (for any inequality measure), in spite of the fact that 
everyone has voted for A.

We may also note that \eqref{a2} would enable all three choices AB, AC and BC to
have uniform load distributions, so although this would enable A to be
elected, the result would be a tie which hardly is satisfactory in this
case,
see also \refE{EABAC+} below.
Finally, note that any sequential method \eqref{c2} would elect A first.
\end{example}

\phragmen{} nevertheless continues to study \eqref{a1}, and says that it is
natural to measure the inequality of the loads by the sum of the squared
differences, \ie{} \eqref{b1} above. He shows that this is equivalent to
minimizing the sum $\sum_k\xi_k^2$, and that if there are $v_i$ votes for
candidate $i$, of which $v_{ij}$ also are for $j$, then, 
if we elect a set $\cS$,
\begin{equation}\label{a1b1}
  \sum_k\xi_k^2 =\sum_{i\in \cS}\frac{1}{v_i}
+\sum_{i,j\in \cS,\; i< j}\frac{2v_{ij}}{v_iv_j}.
\end{equation}

\phragmen{} then continues to study the sequential version
\eqref{a1}\eqref{b1}\eqref{c2}, which avoids the problem in \refE{EABAC},
but shows by another example 
that this too suffers from undesirable non-monotonicity, see
\refE{EPhr1896b} below.

We may note (although \phragmen{} did not do so) that the election method 
\eqref{a1}\eqref{b1}\eqref{c2} can be given an algorithmic description
similar to 
the one in \refSS{SPhru2}.
The proof follows from \eqref{a1b1} by noting that the increase of the sum
in \eqref{a1b1} if candidate $i$ is added the elected set $\cS$ equals $1/W_i$
in \eqref{jfra1b1}.

\begin{metod}{Sequential version with equipartitioned loads and least
	squares criterion}
Seats are given to candidates sequentially.
Let $v_i$ be the number of votes for candidate $i$.
If a set $\cS$ of candidates already has been elected, 
then each ballot
with a set $\gs$ of candidates is given a place number
$q_\gs':=\sum_{i\in \gs\cap \cS} 1/v_i$, \ie{} each elected candidate $i$
contributes 
$1/v_i$ to the place number on each ballot containing $i$.
Let further $v_\gs$ be the number of ballots with the set $\gs$, and let
$q_\gs:=v_\gs q_\gs'$, their total place number.
The \jfr{} for candidate $i$ is then defined as
\begin{equation}\label{jfra1b1}
  W_i:=\frac{v_i}{1+2\sum_{\gs\ni i} q_\gs}
=\frac{\sum_{\gs\ni i} v_\gs}{1+2\sum_{\gs\ni i} q_\gs}.
\end{equation}
The candidate with the largest \jfr{} is elected.  
\end{metod}

As said above, 
\eqref{a1}\eqref{b1}\eqref{c2} suffers from non-monotonicity, 
as shown by the following example.
(See \refE{E-monoTh} for a similar example for Thiele's method.)

\begin{example}%[Non-monotonicity for \eqref{a1}\eqref{b1}\eqref{c2}]
\label{EPhr1896b}
This example is in principle from \citet{Phragmen1896},
but \phragmen's numerical example is incorrect; the following corrected
version is due to Xavier Mora (personal communication).

Unordered ballots. 2 seats.
\begin{val}
\item [1145] A
\item [885] B
\item [900]C
\item [55]AB
\item [100]AC
\end{val}
Using \eqref{a1}\eqref{b1}\eqref{c2}, 
the first elected is A (who has most votes), and
a calculation, \eg{}  using \eqref{jfra1b1}, 
shows that the second place is tied between B and C.
However, if some of the voters for A change their mind and add B,  
so that the numbers of votes instead are
\begin{val}[20pt]
\item [$1145-x$] A
\item [885] B
\item [900]C
\item [$55+x$]AB
\item [100]AC
\end{val}
for some $x>0$ (with $x<360$ to ensure that A still is elected first), then
C will be elected; conversely, with a change in the opposite direction
($x<0$), B is elected. The election of B can thus be prevented by more
voters voting  for B!
\end{example}

\phragmen{} draws the conclusion that \eqref{a1} has to be abandoned, and
replaced by \eqref{a2}, which does not have this kind of non-monotonicity
as shown by the following result.
(See also \refS{Smono2} and in particular the overlapping
\refT{Tmono-u}.)

\begin{theorem}[\citet{Phragmen1896}]\label{TPh1896}
Every  election method based on \eqref{a2},
for any measure of inequality, satisfies the following:
Consider an election and let A be one of the candidates.
Suppose that the votes are changed such that A gets more votes, either
from new voters that vote only for A, or from 
voters that add A to their ballots (thus changing the vote from some set
$\gs$ to $\gs\cup\set A$), but no other changes are made
(thus all other candidates receive exactly the same votes as before).
Then this cannot hurt A; if A would have been elected before the change,
then A will be elected also after the change.
\end{theorem}

\begin{proof}
  Consider first an \opt{} method.
Each set of candidates not containing $\rmA$ has the same possible
distributions of loads, and thus the same minimum inequality as before
the change.
On the other hand, a set containing $\rmA$ still has all possible
distributions of loads that existed before the change, and possibly some
new ones; hence 
the minimum inequality for such a set is at most the same as before.
If $\rmA$ would be elected before the change, then a set of the latter type had
smaller (or possibly equal) minimum inequality than every set of the former
type, and then the same holds after the change, so $\rmA$ will still be
elected.

Consider now a sequential method.
\xfootnote{\phragmen{} did not mention this case explicitly.}
Consider the round when $\rmA$ would have been elected, if there was no change.
After the change, either $\rmA$ already has been elected before this round, and
we are done,  
or the preceding rounds have elected the same candidates as before the
change,
and then $\rmA$ will be elected in this round, by the same argument as in the
\opt{} case.
\end{proof}

Having thus argued for \eqref{a2}, \phragmen{} says that the election method
is determined by the measure of inequality.
\xfootnote{
 \phragmen{} did not explicitly discuss the difference between
\opt{} and sequential methods, or the computational problems with the former.
}
He notes that this measure can be chosen in different ways,
and continues that in order to obtain as simple calculations as possible,
it is recommended to use 
the difference between maximum and average load \eqref{b2} and 
the sequential version \eqref{c2}.

\phragmen{} thus has finally come to the election method
\eqref{a2}\eqref{b2}\eqref{c2}, and notes that this is the method that he
earlier has proposed in \cite{Phragmen1894,Phragmen1895}, which is
described in \refS{SPhru}; it is easily seen that the optimal distribution
of loads is the same as the voting powers assigned in the formulation in
\refSS{SPhru1}.
\xfootnote{
This was obviously the intention of the paper \cite{Phragmen1896}.
It seems that one of the main objections to \phragmen's method
(from its proposal in 1894 to the adoption of the ordered version 27 years
later) 
was that, regardless of whether it had mathematical advantages or not,
it was too complicated to be understood and to be used in practice.
\phragmen{} tried varying formulations and motivations in  different
papers, and  in \cite{Phragmen1896} he thus tries to present the method as 
simple by giving even more complicated alternatives.
(He does not explicitly say that the method is the simplest possible 
among the acceptable alternatives, but he possibly wanted to give that
impression.)
However, in retrospect, 
this attempt to present the method as simple  does not
seem succesful, since
\phragmen{} later did not use this argument again.
}

Although \phragmen{} thus favoured the method
\eqref{a2}\eqref{b2}\eqref{c2}, one might consider also other combinations
than the ones he studied. 
In particular, the other three versions with \eqref{a2}
also seem interesting, at least from a
mathematical point of view;
they are 
	\eqref{a2}\eqref{b2}\eqref{c1},
	\eqref{a2}\eqref{b1}\eqref{c1} 
and 	\eqref{a2}\eqref{b1}\eqref{c2} above:

\begin{metod}{\Opt{} version of \phragmen's method}
  Elect the set of $\MM$ candidates such that, with loads distributed
  optimally,
the maximum load is minimal.
\end{metod}

\begin{metod}{\Opt{} least squares version of \phragmen's method}
  Elect the set of $\MM$ candidates such that, with loads distributed
  optimally,
the sum of the squares of the loads is minimal.
\end{metod}

\begin{metod}{Sequential least squares version of \phragmen's method}
Seats are given to candidates sequentially.
When a candidate is elected, each ballot with this candidate is given a load
for that candidate, in addition to any load that might exist from previously
elected candidates; the additional loads are chosen such that their sum is\/ $1$
and the sum of the squares of the total loads of the ballots is minimal. 
In each round, the candidate is elected such that the resulting sum of
squares of loads after electing the candidate is smallest.  
\end{metod}

The two \opt{} methods are studied 
by \citet[Section 8]{MoraO} (together with their party versions as in
\refSS{SSMora-party}) 
and
in \cite{BrillEtAl}
(under the names \emph{max-\phragmen} and \emph{var-\phragmen}).
The sequential least squares version is studied by Mora \cite{Mora-var}.

The \opt{} methods, although optimal in the mathematical sense of
optimizing some function, do not always
yield results that seem optimal from other points of view. 
This was seen in \refE{EABAC}, and we
can modify the example to make it more striking.

\begin{example}\label{EABAC+}
\quad

Unordered ballots. 2 seats.
  \begin{val}
\item [100]AB
\item [100]AC
\item [1]B
\item [1]C
\end{val}
For any \opt{} method \eqref{c1},
using either \eqref{a1} or \eqref{a2}, and either \eqref{b1} or \eqref{b2}
or any other method of inequality, BC will be elected, since this gives a
perfectly uniform load distribution, while any other combination leaves  
one voter with load 0.

Since every voter but two votes for A, this result seems questionable.
Indeed, any sequential method (of the types considered here) begins by
electing the candidate with most votes, so A is elected (and the second seat
is a tie between B and C).

This example also shows that the \opt{} methods \eqref{c1} considered in this
subsection differ from the sequential ones \eqref{c2}. 

Cf.\ the similar \refE{ETh12}, showing the difference between the \opt{}
and the sequential versions of Thiele's method. That example would work here
too, but not conversely. (In this type of example, with the four ballot types
above and symmetric in B and C, Thiele's \opt{} and addition
 methods yield the same
result unless at least a third (and at most half)
of the votes are on B or C only, while for the methods
studied here, the fraction of such votes can be arbitrarily small.)
\end{example}

\citet{Mora-var} found a  technical problem with 
the sequential least squares method \eqref{a2}\eqref{b1}\eqref{c2},  
shown in the following example.

\begin{example}[\citet{Mora-var}]
The sequential least squares method \eqref{a2}\eqref{b1}\eqref{c2}.  

Unordered ballots. 3 seats.
\begin{val}
\item [9]AB
\item [1]ABC
\item [3]CD
\end{val}
The first seat goes to A, and the loads on the three types of ballots are
$(0.1,0.1,0)$.
The next seat goes to C, and then the loads are
$(0.1, 0.275, 0.275)$.
The third seat goes to B. However, if we would distribute the total loads of
the ballots containing B (including one unit for the election of B)
uniformly, we would get the load distribution
$(0.2175, 0.2175, 0.2750)$, where the load  on the ballot has decreased.
This is not allowed by the formulation above, and instead we have to
distribute
the load of B on the 9 ballots AB only, giving the loads
$(0.2111,0.275,0.275)$.
So either we have to accept this behaviour, which complicates the
implementation and reduces the usual advantage of least squares methods,
or the method has to be modified by allowing the load of a ballot to
decrease.
\xfootnote{
An analogous phenomenon occurs in the version of STV that uses the inclusive
Gregory method
(\refApp{ASTV}\ref{STV-inclusive}), where the voting value of
a ballot can increase when the surplus is transferred,
see \cite{Farrell:1983}.
}
\end{example}

Consider now the party list case (see \refS{Sparty}).
In this case, there is no
difference between \eqref{a1} and \eqref{a2}, since an equidistribution of
loads is optimal; if a party has $v$ votes and  $m$ candidates
elected, then each of its $v$ ballots thus has load $m/v$.
The problem of optimizing these quotients for party lists
by the method of least squares was
considered in 1910 (thus  after \phragmen's paper, but presumably
independent of it)
by \citet{StL}, who showed that this leads to the method now named after
him,
see \refApp{AStL}. In particular, since \StL's method is sequential, 
the optimization problem can in this case be solved sequentially (greedily).

\citet{StL} also more briefly,  again for party lists, 
considered minimizing the maximum load, \ie, \eqref{b2} above, and showed
that this leads to D'Hondt's method (\refApp{ADHondt});
this had also earlier been shown by
\citet{Rouyer} and \citet{Equer}, see \citet{Mora:Jefferson}.
\xfootnote{As said in \refS{Sparty}, \phragmen{} 
showed this for his sequential method \eqref{c2}; it is easy to see that in
the party list case, there is no difference between the \opt{} and
sequential versions, so both yield D'Hondt's method, but I do not know
whether \phragmen{} did observe this; 
he did not consider the \opt{} version at all except
  implicitly   in \cite{Phragmen1896} as discussed above.
}

We thus can conclude the following,
\cf{} \refT{TDHondt}.
\begin{theorem}\label{Tchu}
  In the party list case, the following holds:
  \begin{romenumerate}
  \item 
The least squares methods using \eqref{b1}, with either \eqref{a1} or
\eqref{a2} and either \optal{} \eqref{c1} or sequential \eqref{c2},
all yield the same result as Sainte-Laguë's method.
  \item 
The maximum load methods using \eqref{b2}, with either \eqref{a1} or
\eqref{a2} and either \optal{} \eqref{c1} or sequential \eqref{c2},
all yield the same result as D'Hondt's method.
  \end{romenumerate}
\end{theorem}

\begin{example}\label{Echu}
\quad

Unordered ballots. 2 seats.
  \begin{val}
  \item [5]AB
  \item [2]CD
  \end{val}
This is a party list case, and by \refT{Tchu}, every least squares method
\eqref{b1}
(of the types considered above) will elect (e.g.)\ AC,
while every maximum load method \eqref{b2} will elect AB.

This example shows that the least squares methods differ from the maximum
load methods.
\end{example}

\begin{remark}
  \citet{StL} also considered (for party lists)
maximizing the minimum load
(we might call this (b3)); this
  yields \emph{Adams' method} (see \cite{BY}), which is equivalent to giving
  every party 1 seat and distributing the rest by D'Hondt's method.
In the context of general unordered ballots (instead of party lists), 
in typical cases there will always be some ballots that do not vote for any
of the elected candidates, so their load is 0 and the minimum load is 0,
which seems to make (b3) less interesting.

\citet{StL} also considered optimizing the distribution of votes per elected
candidate, again using the least squares method, showing that (at least
assuming that each party gets at least
one seat), this yields a method that is now known as 
\emph{Huntington's method} (proposed by Huntington in 1921),
which since 1941 is used for the allocation of seats in the US House of
Representatives among the states, see \cite{BY}.
This version of optimization could perhaps be extended to general
unordered ballots, using a dual of \eqref{a2} where
each voter has $1$ vote which is divided between the names on the ballots in
an arbitrary way, and the goal is to make the resulting total votes on the
elected candidates high and close to equal. However, it is not clear how to
treat candidates that are not elected, and how to define the quantity that
should be optimized.
\end{remark}

Examples \ref{EABAC+} and \ref{Echu} thus show that of the 8 election
methods obtained by combining the alternatives above, the only ones that
possibly could be equal (in the sense of always giving the same outcome)
would be two that differ only in using \eqref{a1} or \eqref{a2}. 
For the combination \eqref{b1}\eqref{c2} (sequential least squares),
\refE{EPhr1896b} and \refT{TPh1896} show that the two methods differ.
For completeness, we give another example to show that also in the other
cases, \eqref{a1} and \eqref{a2} give different methods, and thus all 8
methods are different.

\begin{example}
\quad

Unordered ballots. 2 seats.
  \begin{val}
  \item [1]AB
  \item [1]A
  \item [1]B
  \item [$c$]C
  \end{val}
Here $c$ is a  positive rational number with $c<2$; 
we can obtain  integer numbers of votes by multiplying all numbers by
the denominator of $c$.

By symmetry, there will be ties; we may suppose that these are resolved
lexicographically, with A before B.
In the sequential versions \eqref{c2}, then A will be elected first (because
$c<2$); thus for the second seat we consider AB and AC.
In the \opt{} versions \eqref{c1}, by symmetry we also consider AB and AC.

For AC, the loads on the four types of ballots will in all cases (\eqref{a1}
or \eqref{a2}, \eqref{b1} or \eqref{b2}, \eqref{c1} or \eqref{c2}) be 
$\bigpar{\frac12,\frac12,0,\frac1c}$, with maximum $\frac{1}c$ and sum of
squares 
(recalling that we have $c$ ballots of the last type) $\frac12+\frac1c$.

For AB, we obtain 
using \eqref{a1}, in both the \opt{} and sequential versions,
the loads $\bigpar{1,\frac12,\frac12,0}$ with maximum 1 and sum of squares
$\frac{3}2$; thus, for all four methods using \eqref{a1}, C will get a seat if
$c>1$.

Using \eqref{a2}, we instead find that in the \opt{} version \eqref{c1},
the loads for AB will be $\bigpar{\frac23,\frac23,\frac23,0}$, with
maximum $\frac23$ and sum of squares $\frac43$.
Hence, 
C wins a seat if $c>\frac32$ 
for method \eqref{a2}\eqref{b2}\eqref{c1}, 
and if $c>\frac65$ 
for method \eqref{a2}\eqref{b1}\eqref{c1}.

Using \eqref{a2} and the sequential version \eqref{c2},
the loads for AB will be $\bigpar{\frac34,\frac12,\frac34,0}$, with
maximum $\frac34$ and sum of squares $\frac{11}8$.
Hence, 
C wins a seat if $c>\frac43$ 
for method \eqref{a2}\eqref{b2}\eqref{c2}, 
and if $c>\frac87$ 
for method \eqref{a2}\eqref{b1}\eqref{c2}.

Together with the examples above, this shows that all 8 combinations 
discussed above yield different methods.
\end{example}

\subsection{Recent STV-like versions of \phragmen's method}
Olli Salmi \cite{Salmi,Salmi2} has proposed a modification of \phragmen's
method to something similar to STV by introducing the Droop quota as a
criterium for election and eliminations when no-one reaches the quota.
This has been further developed by  \citet{Woodall:QPQ} and, in several
versions,  \citet{Hyman}.

However, it seems doubtful whether it is possible to modify \phragmen's
method in some way without losing 
some of its advantages.

  \section{Some conclusions}\label{Sconclusions}
As said in the introduction, our purpose of is not to advocate any
particular method. 
Nevertheless, we draw some conclusions for practical applications.

It seems that 
Thiele's unordered (addition) and ordered methods
both have serious problems in some situations,
shown by several of the examples in \refS{Sex} and by further results in
\refS{Sprop}, and that these methods therefore are not satisfactory in
general.
(This was also the conclusion of  \citet{Cassel},
%(who calls Thiele's unordered method 
% ``a false generalization of D'Hondt's rule'')
\citet{Tenow1912} and the commission reports
\cite{bet1913} and
\cite[pp.~213--220]{bet1921}.)
On the other hand, these methods are simple and may be useful in some
situations, and Thiele's ordered method has been used for a long time inside
local councils in Sweden, see \refApp{Awithin}, as far as I know without any
problems. 

Thiele's \opt{} method is computationally difficult, but might be used in
some (small) situations. It is perhaps not sufficiently investigated, but
note the property in \refT{Tth-EJR}.
On the other hand, this method is likely to have at least some of the same
problems with tactical voting as Thiele's other methods.

\phragmen's unordered and ordered methods seem quite robust in many
situations, and they have good proportionality properties (see
\refS{Sprop}),
but they have the disadvantage of 
leading to rather complicated calculations that only in simple cases can be
made by hand. 

\phragmen's unordered method does not ignore full ballots, but that is
perhaps more a curiosity than a real problem in practice.

Note also the general problems with unordered methods discussed in
\refS{Sdecap}. Nevertheless, unordered ballots seem to work well in
practice in many situations without organized parties, for example in
elections in non-political associations and organizations.

\begin{ack}
I thank Markus Brill, Rupert Freeman, Martin Lackner and Xavier Mora
for helpful comments,
Xavier Mora also for help with references, 
and Steffen Lauritzen for help with Thiele's biography.

This work was partly carried out in spare time during a visit to the 
Isaac Newton Institute for Mathematical Sciences
%during the programme 
%Theoretical Foundations for Statistical Network Analysis in 2016
(EPSCR Grant Number EP/K032208/1), 
partially supported by a grant from the Simons foundation, and
a grant from
the Knut and Alice Wallenberg Foundation.
  \end{ack}

\appendix

\section{Biographies}\label{Abio}
\subsection{Edvard Phragmén}\label{Abio-Phr}
Lars Edvard Phragmén (1863--1937)
was a Swe\-dish mathematician, actuary and insurance executive.
He was born in Örebro 2 October 1863,
and died in Djursholm (outside Stockholm) 13 March 1937.

Edvard Phragmén began his university studies in Uppsala in 1882, but
transferred in 1883 to Stockholm, where Gösta Mittag-Leffler had become
professor in 1881.
\phragmen{} obtained a  Licentiate degree in 1889, and was awarded a Ph.D.\
h.c.\ by 
Uppsala University in 1907.

In 1888, Edvard Phragmén was appointed coeditor %editorial secretary 
of
Mittag-Leffler's journal \emph{Acta Mathematica}, where he immediately made
an important contribution by finding an error in a paper by Henri Poincaré on the
three-body problem. The paper had been awarded a prize in a competition that
Mittag-Leffler had persuaded King Oscar II to arrange, but Phragmén found a
serious mistake when the journal already had been printed; the copies that had
been released were recalled and a new corrected version was printed.
\phragmen{} continued to be an editor of Acta Mathematica until his death in
1937. 

In 1892, Edvard Phragmén became professor of Mathematics at Stockholm
University
(at that time Stockholm University College).
%(succeeding Sonja Kovalevski
In 1897, he became also actuary in 
the private insurance company Allmänna {Lif}\-försäk\-rings\-bolaget.
His  interest in Actuarial Science and insurance companies seems to have grown,
and in 1904 he left his professorship to become the first 
head of the Swedish  Insurance Supervisory Authority. 
%http://www.fi.se/Folder-EN/Startpage/About-FI/Who-we-are/History/ 
He left that post in 1908 to become 
director of %the private insurance company 
Allmänna Lif{}försäkringsbolaget, a
post that he held until 1933.

Edvard Phragmén
became a member of the
Royal Swedish Academy of Sciences in 1901.
He was President of the Swedish Society of Actuaries 1909--1934.

His best known mathematical work is the Phragmén-Lindelöf principle in complex
analysis, a joint work with the Finnish mathematician Ernst Lindelöf which
was published in 1908.

His interest in election methods is witnessed by his publications
\cite{Phragmen1893,Phragmen1894, Phragmen1895, Phragmen1896, Phragmen1899}.
 %(and other writings). 
Moreover,
he was a member of the Royal Commission on a Proportional Election Method
1902--1903 \cite{bet1903}, 
and of a new Royal Commission on the Proportional Election Method
1912--1913 \cite{bet1913}.
He was also a local politician and
chairman of the city council of Djursholm 1907--1918.

See further \cite[L  Edward Phragmén]{SBL}, \cite{Domar} and 
\cite[Lars Edvard Phragmén]{MT}.

%led av komm ang proportionellt valsätt för riksdagens AK okt 02–okt 03, 
%av komm ang vissa försäkringsanstalter juli 03–dec 05,
%led av utredn ang proportionella valmetoden juli 12–dec 13, 

%http://www2.math.su.se/matematik/Historia/Histeng.pdf

%https://sok.riksarkivet.se/sbl/Presentation.aspx?id=7271

% http://www-history.mcs.st-andrews.ac.uk/Biographies/Phragmen.html

\subsection{Thorvald Thiele}
Thorvald Nicolai Thiele (1838--1910) was a Danish astronomer and
mathematician.
He was born in Copenhagen 24 December 1838, and died there 26 September 1910.

Thorvald Thiele 
studied Astronomy at the University of Copenhagen, where he obtained his
Master's 
degree in 1860 and his Ph.D. %in Astronomy 
in 1866, with a thesis on a double star. He was
Professor of Astronomy at the University of Copenhagen
from  1875 until his retirement in 1907.
He became a
member of the Royal Danish Academy of Sciences and Letters in 1879.
%  Videnskabernes Selskab 

Thiele had an increasingly poor eye-sight (severe astigmatism) and turned to
theoretical
work and mathematics instead of observational astronomy;
he published papers in both Astronomy and Mathematics.
Among his mathematical 
contributions are ``Thiele's interpolation formula'' for finding a
rational function taking given values at given points, published in 1909.
He made also contributions to Statistics, 
where he introduced the half-invariants
(later known as cumulants) as well as many other fundamental concepts in his
book from 
1889.
Moreover, he was interested in Actuarial Science, where he both did
theoretical work and was a founder
of the insurance company Hafnia where he was Mathematical Director 
from 1872 until 1901; from 1903 he was the chairman of the board of the
company.
% enl DBL. Ngt annorlunda enl SL.
Thiele was an original thinker, and his ideas were often ahead of his time.

Thorvald Thiele 
was one of the founders of the Danish Mathematical Society in 1873,
and
one of the founders of 
the Danish Society of Actuaries in 1901; he was the president of the latter
until his death.
In 1901, he also became a foreign member of the Institute of Actuaries
in London.

See further 
\cite{Lauritzen},
\cite[Thiele, Thorvald Nicolai]{DBL} and
\cite[Thorvald Nicolai Thiele]{MT}.

% http://www-history.mcs.st-andrews.ac.uk/Biographies/Thiele.html

%Dansk biografisk Lexikon / XVII. Bind. Svend Tveskjæg - Tøxen /
% p. 188--190 (1887-1905)
%http://runeberg.org/dbl/17/0190.html

%Nordisk familjebok / Uggleupplagan. 28. Syrten-vikarna - Tidsbestämning /
%sp. 1104 
%http://runeberg.org/nfch/0582.html

\section{\phragmen's original formulation}\label{APhragmen1894}
\phragmen{} first presented his method in a short note
\cite{Phragmen1894}, dated 14 March 1894, 
in the Proceedings of the Royal Swedish Academy of Sciences.
The note is written in French, and the method is defined as follows, using
the term \emph{force électrice}:

\smallskip

{\it
Désignons par $k$ une quantité variable, et
\begin{enumerate}
 \renewcommand{\labelenumi}{\textup{\arabic{enumi}:o)}}%
 \renewcommand{\theenumi}{\labelenumi}%

\item 
commen\c cons par donner à la force électrice de tous les bulletins cette
valeur $k$.
\item 
Avant de proclamer l'élection d'un premier représentant, nous étab\-lirons
entre les candidats un certain ordre de préférence, en calculant pour chacun
d'eux la valeur de $k$ qui donne la valeur un à la somme de force électrice
de tous les bulletins qui contiennent son nom, et en donnant toujours la
préférence au candidat pour qui cette valeur est moins grande.
\item 
Nous proclamerons élu pour représentant le candidat qui se trouve en tête de
cette liste, et nous réduirons en même temps la force électrice des
bulletins portant son nom, en y soustrayant la valeur de cette force
électrice qui correspond à l'élection.
\item 
Nous répèterons les opérations 2:o et 3:o alternativement jusqu'à ce que le
nombre prescrit de représentants soient proclamés élus.
\item 
S'il arrive, en faisant l'opération 2:o, qu'il y a deux ou plusieurs
candidats pour lesquels la valeur de $k$ devient égale, on déterminera leur
ordre relatif d'après la liste analogue obtenue à l'opération précédente.
Si, même à la première opération, ils ont la même valeur de $k$, leur place
relative est déterminée par la sort (ou par tout autre moyen qu'on y
préfèrerait). 
\end{enumerate}
}

\phragmen{} gave his method an expanded treatment in the book
\cite{Phragmen1895} the following year;
he also gave a different motivation for the method in 1896
\cite{Phragmen1896} (see \refSS{SSopt})
and made further comments,
including comparisons with Thiele's methods in 1899
\cite{Phragmen1899}. Rule 5:o) above, on ties, seems to have been dropped;
otherwise there are only minor variations (without mathematical significance)
in the formulations; for example, in \cite{Phragmen1899} he tries to make
the method more easily understood by talking about the ``load'' 
a ballot receives by the election of a candidate, instead of ``voting
power'' (see \refR{Rload}).

\section{\phragmen's and Thiele's  methods
as formulated
in current Swedish law}

Both \phragmen's ordered method and Thiele's ordered method are used
officially in Sweden for some purposes, see \refApp{Ahistory}.
We give here official formulations (in Swedish) from  current laws.

The methods are not called ``\phragmen's'' and ``Thiele's'' anywhere 
in the laws;  \phragmen's method has no name in the \EA, but 
in the Parliament Act (Riksdagsordningen) \cite[12 kap 8 \S, 12.8.5]{RO}
it is called
``heltals\-metoden'' (the whole 
number method), which otherwise is the Swedish name for D'Hondt's method.
This is really a misnomer; although the method clearly is related to
D'Hondt's method, and \phragmen{} argued that his method was a
generalization of D'Hondt's, it is in several important
ways different from it.
Moreover, from a mathematical point of view, an important feature is the use
of non-integer place numbers (generalizing the integer ones in D'Hondt's
method), which makes the name ``heltalsmetoden'' a bit bizarre.

\subsection{\phragmen's ordered method in the Swedish \EA}
\label{APh-vallag}

\phragmen's ordered method has been used in the Swedish \EA{}
since 1921 for distribution of seats within each party, although since 1998
only as a secondary method that rarely is used,
see \refApp{Ahistory}.

The official formulation %(in Swedish) 
in the current (2016) \EA{}
\cite[14 kap.\ 10 §]{vallag}
%is of the type described above in \refS{SPhro3}, and
is as follows.
(The formulations have been essentially identical since 1921.)
\begin{xquote}{18pt}{0mm}\em
%  10 § 
%Kan inte ett tillräckligt stort antal ledamöter utses på grundval av ett
%personligt röstetal, skall ordningsföljden mellan återstående kandidater
%bestämmas genom att jämförelsetal beräknas enligt följande. 
%
Vid första uträkningen gäller en valsedel för den kandidat som står först på
sedeln varvid bortses från kandidater som redan tagit plats. Valsedlar med
samma första kandidat bildar en grupp. Varje grupps röstetal räknas
fram. Röstetalet är lika med det antal valsedlar som ingår i gruppen. Samma
tal är också jämförelsetal för den kandidat som står först på gruppens
valsedlar. Den kandidat vars jämförelsetal är störst får den första platsen
i ordningen. 

Vid varje följande uträkning gäller en valsedel för den kandidat som står
först på sedeln, men man bortser från kandidater som redan fått plats i
ordningen. Den eller de grupper, vilkas valsedlar vid närmast föregående
uträkning gällde för den kandidat som fick plats i ordningen, upplöses och
ordnas i nya grupper, så att valsedlar som vid den pågående uträkningen
gäller för en och samma kandidat bildar en grupp. Övriga befintliga grupper
behålls däremot oförändrade. För varje nybildad grupp räknas röstetalet
fram. Röstetalet är lika med det antal valsedlar som ingår i gruppen. För
samtliga kandidater som deltar i uträkningen beräknas röstetal och
jämförelsetal. 

Röstetalet för en kandidat är lika med röstetalet för den grupp eller det
sammanlagda röstetalet för de grupper vilkas valsedlar gäller för
kandidaten. Jämförelsetalet för en kandidat är lika med kandidatens
röstetal, om inte den grupp av valsedlar som gäller för kandidaten deltagit
i besättandet av en förut utdelad plats. Om detta är fallet, får man
kandidatens jämförelsetal genom att kandidatens röstetal delas med det tal
som motsvarar den del gruppen tagit i besättandet av plats eller platser som
utdelats (gruppens platstal), ökat med 1, eller, om flera grupper av
valsedlar som gäller för kandidaten deltagit i besättandet av förut utdelad
plats, med dessa gruppers sammanlagda platstal, ökat med 1. Platstalet för
en grupp beräknas genom att gruppens röstetal delas med det största
jämförelsetalet vid uträkningen närmast före gruppens bildande. För kandidat
som redan stod först på någon valsedel beräknas nytt platstal endast för
nytillkomna valsedlar. Bråktal som uppkommer vid delning beräknas med 2
decimaler. Den sista decimalsiffran får inte höjas. 

Den kandidat vars jämförelsetal är störst får nästa plats i ordningen.
%\cite[14 kap.\ 10 §]{vallag}
\end{xquote}

The \EA{} specifies that calculations should be done to two decimal
places, 
rounded downwards. 
\xfootnote{This was obviously of practical importance in 1921.
Today, with computers, it would seem better to use exact calculations with
rational numbers.

Already the commission \cite{bet1913} that suggested the method in 1913
proposed that the calculations should be done with decimal numbers, and that
two decimal places would be enough for practical purposes; they also
for simplicity recommended consistently either rounding up or down, but
favoured rounding up for reasons not further explained. The law introduced
in 1921 chose rounding down, but otherwise followed these recommendations.
}
Otherwise, the formulation is equivalent to the one in \refS{SPhro3}.
%(although stated in slightly different words).

\subsection{Thiele's ordered method in current Swedish law}
\label{Athi-lag}

Thiele's ordered method is used in Sweden for the distribution of seats within
parties at
elections in city and county councils, for example in the election of the
city executive board and other boards, see \refApp{Awithin}.
% även inom kyrkan, \cite[29 kap 8-11 \S]{Kyrkoordning} (ordagrant samma)
The method is formulated as follows in
the Act on Proportional Elections (Lagen om proportionella val)
\cite[15--18 §]{lagprop}:
\begin{xquote}{12pt}{0mm}\em
\begin{itemize}
\item [15 §] 
Ordningen mellan namnen inom varje valsedelsgrupp skall bestämmas genom
särskilda sammanräkningar, i den utsträckning sådana behövs. 

\item[16 §] 
Efter varje sammanräkning skall det namn som enligt 18 § har fått det
högsta röstetalet föras upp på en lista för valsedelsgruppen, det ena under
det andra. Namnen gäller i den ordningsföljd som de har blivit uppförda på
listan. 

\item [17 §] 
Vid varje sammanräkning gäller en valsedel bara för ett namn.

Valsedeln gäller för det namn som står först på valsedeln, så länge detta
namn inte har förts upp på listan. Därefter gäller valsedeln för det namn
som står näst efter det namn som redan har förts upp på listan. 

\item[18 §] 
En valsedel som gäller för sitt första namn räknas som en röst.

När den gäller för sitt andra namn räknas den som en halv röst. En valsedel
som gäller för sitt tredje namn räknas som en tredjedels röst, och så vidare
efter samma grund.
\end{itemize}
\end{xquote}

\section{History and use of \phragmen's and Thiele's methods in 
Sweden\protect
\footnote{%
As far as I know, no version of \phragmen's or Thiele's method has ever been
used outside Sweden.
(Thiele was Danish. I do not know whether his, or \phragmen's, method was
discussed in Denmark.)
}}

\label{Ahistory}

As said above, \citet{Phragmen1894} proposed his method in 1894 
(see \refApp{APhragmen1894}), and \citet{Thiele} as a response proposed his
method 
in 1895.

This was a period when electoral reform was much discussed in Sweden, both
the question of universal suffrage and the election method,
see \eg{} \cite{Carstairs}.
The two questions were linked; the conservatives expected to become a
minority when universal suffrage was introduced, and therefore many of them
wanted a proportional election system; conversely, many liberals
and socialdemocrats wanted a
plurality system with single-member constituencies.

Sweden had 
1867--1970 a parliament with two chambers, where the Second Chamber was
elected in general elections, while the First Chamber was elected
indirectly, by the county councils.
(Until 1909, the county councils were  also elected indirectly, by electors
chosen by city councils and
rural municipality councils;
these councils were elected in local elections,
\xfootnote{Until 1918, in local elections,
each voter had a number of votes proportional to his
  (or in exceptional cases her) tax. There were some limitations, in particular
were the number of votes  limited to at most 100 per voter
in cities and, from 1900, at most 5000 in
  rural municipalities; in 1909--1918 the number of votes per voter were 1--40.
}
and 
thus the First Chamber was elected indirectly in three steps.)
Until 1909, the Second Chamber was elected by the \emph{Block Vote}
(\refApp{ABV}), mainly in single-member constituencies but in larger
cities in constituencies with several members;
only men with a certain income or real estate of a certain value were
allowed to vote.
There were no formal parties, and no registration of candidates before the
election;  every eligible man was a possible candidate, and the voters could
write any names they wanted on the ballots, although parties and other
political organizations recommended certain lists and printed ballots that
were used by most voters. (Even these organized lists often overlapped, see
Examples \ref{E1893a} and \ref{E1893b}.)

In 1896, the government proposed a minor extension of the suffrage (lower
limits for income and real estate), together with introduction of a
proportional method 
in the  constituencies that elected several members of parliament
(\emph{\andrae's method}, a form of STV, see \refApp{ASTV}); however,
the parliament voted against.

Nevertheless, the questions continued to be discussed.
Several bills on universal suffrage for men 
were introduced in parliament and defeated, mainly because of disagreement
on whether the election system should be proportional or not.
A Royal Commission on a Proportional Election Method was appointed in 1902,
with \phragmen{} as one of its members. 
(Nevertheless, the commission 
proposed the cumulative method, see \refApp{ACV}, and not \phragmen's own.)

In 1909, finally, universal suffrage for men was accepted together with 
a proportional election system, based on \emph{D'Hondt's method}
(\refApp{ADHondt}). 
Thus parties were formally introduced; there was still no registration
before the election, but the writer was supposed to write a party name
on the ballot, followed by a list of candidates as before.
\xfootnote{It was until 1921
also possible to omit the party name and just list
  candidates, as previously. These ballots (``the free group'') were treated
  as a separate party, with some special rules. Such ballots were few, and 
they had only rarely (in some local elections) any effect on the outcome.
In the elections to parliament 1911, only 210 votes of 607\,487 were without
a party name. In the county council elections 1910, the free group dominated
in a few constituencies, and in total 7 county councillors (of 1217 in the
whole country) were elected from the free group.
}\,
\xfootnote{
In more recent years, party names and candidates are usually registered
before the election, and ballots are printed by the Election Authority.
However, registration is compulsory only from the general election in 2018;
until now, it has been possible for voters to use blank ballots and fill in
any party name and any candidates.
}
Each party was then, by D'Hondt's method, 
given a number of seats according to the party names on
the ballots.

It seems that parties had now become better organized and dominated the
political scene, and that it was natural to base the election system on
them;
to use a method like \phragmen's or Thiele's to distribute the seats without 
party labels was no longer realistic (as it was when the methods
were proposed some 15 years earlier).
However, the parties were still rather loosely organized and consisted of
various more or less clearly defined factions, and since Sweden kept the old
system of open lists, where the voter could list any candidates on the ballot,
a system was needed for the distribution of seats within the parties.
The method adopted 1909 was \emph{Thiele's unordered method}, but combined with
a special rule, called \emph{rangordningsregeln (The Ranking Rule)}, 
to prevent decapitation (see \refS{Sdecap}):
\xfootnote{
It is easily verified that both the later introduced ordered version of
\phragmen's method and the ordered version of Thiele's method 
obey the Ranking Rule, in the sense that any seat allocated by the Ranking
Rule will be allocated to the same candidate by these nethods too.
}

\begin{metod}{The Ranking Rule}
The ballots are ordered.
If more than half of the ballots for a given party have the same first name,
then this candidate gets the party's first seat.
If further more than $2/3$ of the ballots have the same first and second
names, then the second name gets the second seat, and so on.
I.e., if more than $k/(k+1)$ of the ballots have the same $k$ names first,
in the same order, then these get the first $k$ seats (for any $k\ge1$).  
\end{metod}

The combination chosen in 1909 for distribution of seats within each party
was thus: 

\begin{metodx}
Ballots are ordered. The Ranking Rule is used for as many seats as
possible. 
The remaining seats, if any, are distributed by Thiele's unordered method,
thus ignoring the order of the names on the ballot.
\end{metodx}

In practice, it turned out that most seats were distributed by the Ranking
Rule. For example, in the elections to the parliament in 1911 were only 21
of the 150 members of the First Chamber, and 18 of the 230 in the Second
Chamber, elected by Thiele's method.
\xfootnote{Note that the same election method always was used for the
  general elections to the Second Chamber and for the elections to the First
  Chamber made by the  county councils.}

However, 
the combination was not successful, and there were several examples where
the outcome did not seem fair; all of them coming from the application of
Thiele's method. One important reason seems to have been that the
combination of two rules, where one depended on the order of the names on
the ballot but the other did not, was confusing, and that votes
could have an effect not intended by the voter.
And even when the voter fully understood the method, it was often difficult
to predict how many from the party that
would be elected by the Ranking Rule, and therefore 
whether the addition or
deletion of a name would help or hinder another (more favoured) candidate.
Moreover, the problem was compounded by the fact that
two parties often agreed to an electoral alliance in a
constituency, where they 
would appear under a common party name in the election, but otherwise
promote their own candidates and their own policies; this was possible by
the open nature of the election system, and it was encouraged by D'Hondt's
method, which works in a superadditive way, so that small parties are
somewhat disadvantaged
(see \eg{} \cite{SJ262} for detailed calculations)
and two parties that form an electoral alliance can
together never lose a seat by that.
\xfootnote{
It also happened that a party participated in the election using another
party's name, without any agreement or consent. 
(For example because the first party
did not expect to get enough votes to get a seat on their own.)
}
Since the different names and lists within a ``party'' thus  in
reality often
represented different parties, it was imperative that the
distribution inside each party was fair.
One typical problem is described in \refE{Erank}.

As a result of dissatisfaction with the results of the new method,
a  Royal Commission on the Proportional Election Method
was appointed in 1912 to suggest improvements;
one of the four memebers was Edvard \phragmen.
\xfootnote{The three other members were Sixten von Friesen, a leading liberal
  politician;
Ivar Bendixon, who was \phragmen's successor as professor of
  Mathematics at Stockholm University, and also Vice-Chancellor of the
  University 
  and a local politician in the city council
  of Stockholm; 
Gustaf Appelberg, a civil servant from the department of Justice.
Two of the four members were thus mathematicians.
}
The commission report in 1913 \cite{bet1913}
(with many constructed examples) discusses
several types of errors that could
appear, both because of the combination of the Ranking Rule and Thiele's
method, and because of deficiencies in the latter itself.
They discuss several possible alternatives
\xfootnote{Including two proposed in the Swedish
parliament, a scoring rule (\refApp{Aborda}), and several methods that seem
to have been invented and discussed within the commission.}, 
and come to the
conclusion that the only satisfactory solution is to use ordered ballots.
Moreover, they discuss first the ordered version of Thiele's method (which
had been proposed in parliament in 1912 by Nilson in \"Orebro) but find
serious problems with this method too.
Finally, the commission presents and proposes the ordered version of
\emph{\phragmen's method}.

Nothing came out of this proposal immediately, 
but a new commission repeated the proposal in 1921,
in connection with the reform introducing universal suffrage for both men
and women. This time the method was adopted, %\cite[§ 19]{vallag1921}.
and  \emph{\phragmen's method} has been used for the distribution of seats
within 
parties since 1921, at least to some extent (see below).
\xfootnote{
The seats were still distributed between parties by D'Hondt's method.
In 1952, this was replaced by the \emph{modified Sainte-Lagu\"e's method} 
(\refApp{AStLmod}). In 1971, the two chambers were replaced by a parliament
with a single chamber; furthermore, adjustment seats were introduced for the
distribution between parties. None of this affected the use of \phragmen's
method within parties.
}

In 1924, the well-established practice of forming electoral alliances
was formalized and alleviated by allowing ballots to contain, besides the
party name also a  alliance name and a faction name; the ballot could thus
be labelled in three levels. (Only the party name was compulsory.)
%I vallagen 1924 \cite{vallag1924} kallas nivåerna
%\emph{parti}, \emph{underparti} och \emph{fraktion}, vilket innebar att en
%kartell av olika partier formellt kallades parti och de ingående partierna
%kallades underpartier.
%Terminologin ändrades 1927 till 
%\emph{kartell}, \emph{parti}, \emph{fraktion} \cite{KU1926}, \cite{vallag1927}.
%Samma system infördes 1926 för kommunala val \cite{KU1926}. 
These possibilities disappeared in 1952, when D'Hondt's method was replaced
by the modified Sainte-Lagu\"e's method, and thus the incentive to form
alliances disappeared.
When ballots with alliance and faction names were used, \phragmen's method
was used recursively in up to three steps, as described in \refS{SSrecur}.

\phragmen's method is still used in the distribution of seats within parties
(see \refApp{APh-vallag}), but it was demoted to a secondary role in 1998,
when a system of preference votes was introduced.
Each party has one or several
ballots, with lists of candidates decided by the party, and the voter 
in addition to choosing a ballot, also  can (but
does not have to) vote for one of the  candidates on the ballot.
The seats given to the party are primarily given to its candidates in order
of these preference votes (\ie, by the \emph{SNTV method}, \refApp{ASNTV}), 
but only to candidates that obtain at least 5\% of the total number of votes
for the party. Only when there are not enough candidates that have reached
this threshold, the remaining seats are distributed according to \phragmen's
method.
\xfootnote{Ignoring the candidates that have got seats because of their
  preference votes. This seems to be a mistake in the combination of the two
  methods. Suppose that a party has two lists ABCD and EFGH in a
  constituency, and get 3 seats. Suppose further that of the ballots for the
  party, 55\% are ABCD and 45\% EFGH, and also that 
A and D, but no-one else, each gets preference votes on more than 5\% of the
ballots. Then A and D are elected by their preference votes, and the
remaining seat goes to the largest list, \ie{} to B. Hence all three elected
come from the same list.
It seems that it would be better to calculate place numbers taking into
account also those elected by preference votes.
}

In most cases, the party chooses to have just one list in a constituency, and
then \phragmen's method only means that any seats not assigned by preference
votes are assigned in the order the party has put them on the ballot.
And  if there are several lists but they have disjoint sets of names, 
\phragmen's method is just the same as D'Hondts.
Thus, \phragmen's method really plays a role  
only in the few cases where the party chooses to have more than one list,
with some of the names on more than one list,
and moreover the voters choose to give not too many preference votes.

Local elections have followed essentially the same rules as
parliamentary elections; in particular, they too have for the distribution
of seats inside parties used
 Thiele's unordered method together with the
Ranking Rule in 1909--1921, and after 1921 \phragmen's ordered method, but since
1998 only for seats not assigned by preference votes.
%Phragméns metod infördes också samtidigt (1922) för kommunala val.
%SFS 1922:240--243 som ersatte 1918:1026 med Thieles metod.

\subsection{Elections within political assemblies}\label{Awithin}

\emph{\phragmen's method} is  in Sweden,
besides its use in general elections described above,
 also used for elections of committees in the
parliament
 \cite[12 kap.\ 8 §]{RO}.
Also in this case, the method has in practice a very small role;
the members of parliament can be expected to be loyal to their parties and
vote on a party list, and thus the seats a party gets are
assigned according to this list.
Moreover, parties usually form electoral alliances for these elections;
\xfootnote{These elections are performed using D'Hondt's method
  (\refApp{ADHondt}), which has the 
important property that a party or electoral alliance that has a majority in
the parliament will get a majority in each committee. This method is
superadditive and encourages electoral alliances, as said above.
}
each party in the electoral alliance can be expected to vote on its own
party list,  and then the result of \phragmen's method will be that the
seats are distributed within the electoral alliance by D'Hondt's method.
\xfootnote{
In practice the elections are usually done by 
consent  by voice vote, often unanimous, to a proposal from the
nominating committee, so \phragmen's method is not even formally used. 
However, the nominating committee has of course
calculated what the result would be of an election, so the election method 
still plays the same role, although hidden.
}

 In local councils (cities and other municipalities, and counties),
elections to various committees are instead done by 
\emph{Thiele's ordered method} \cite{lagprop}.
\xfootnote{
I do not know the reason why different methods are used in parliament and in
local councils. One possibility could be that it has been thought that the
parliament is more important and that it is justified to use the more
complicated but usually better \phragmen's method there, while local
councils can do with the simpler Thiele's method. However, this is only a
guess.
In any case, the choice of method very rarely matters, since as said above,
the votes can be expected to be on party lists, and then the result
 for both methods will be the same as by D'Hondt's method.

Note that a political assembly, where everyone can predict how everyone else
will vote, is an ideal setting for tactical voting as in \refE{Etactic} when
Thiele's method is used. However, this only works inside an electoral
alliance, and I do not know that anyone has ever tried it -- it would hurt
your partners and not your political enemies, and would presumably lead to
bad feelings and repercussions. 
}
Again, the method has in practice a very minor role, for the same reasons as
in the parliament.

\section{Some other election methods}\label{Aother}
We give here brief 
(and incomplete)
descriptions of some other election methods, as a
background to
the methods by \phragmen{} and Thiele
and for comparisons with them.
The selection of methods is, by necessity, far from complete; many other
election methods have been proposed and many different ones have actually
been used in different countries and contexts. 
Note also that the same method often is known under  different names.
The methods described below are mainly those of relevance to a discussion 
of \phragmen's and Thiele's methods.
For further methods and further
discussions, from both mathematical and political aspects, see
\eg{}
\cite{Farrell,Politics,Pukelsheim}.
We give some examples (far from exhaustive) of current or earlier use in
elections. 
For the election methods actually used at present
in parliamentary elections around the
world, see
\cite{IPU}.

Note the important difference between elections methods that are 
\emph{proportional}, \ie, methods where different parties get numbers of
seats that are (more or less) proportional to their numbers of votes,
and other methods (that usually favour the largest party).
\xfootnote{
Whether an election method is proportional or not is not
  precisely defined, even when there are formal parties.
%(and much less so without parties).
Obviously, there are necessarily deviations since the number of seats is
an integer for each party.
Moreover, in practice,
whether the outcome of an election is (approximatively)
proportional depends not only on the method, but also on the number of seats
in the contituency, and on other factors such as the number of parties and
their sizes.
}
Proportional election methods are used for political elections
in many countries; in most (but not all) cases using a list method with
party lists, see Secction \ref{Alist} below.
Recall that the idea of both \phragmen's and Thiele's methods is to have a
proportional election method without formal parties.

As above, $\MM\ge1$ is the number to be elected (in a constituency).
We are (as \phragmen{} and Thiele) mainly interested in the case $\MM\ge2$
(multi-member constituencies), but the case $\MM=1$ is often included for
comparison. 

\subsection{Election methods with unordered ballots}\label{Aotheru}

In these method, each ballot contains a list of names, but their order is
ignored. 
The number of candidates on each ballot is, depending on the method,
either arbitrary or limited to at most (or exactly) a given number, for
example $\MM$, the number to be elected.

\subsubsection{Block Vote}\label{ABV}
\emph{Each voter votes for at most $\MM$ candidates. The $\MM$ candidates with the
largest numbers of votes are elected.}
\xfootnote{One version  requires each voter to vote for exactly $\MM$
  candidates. 
This led to a scandalous result in Sweden in the (autumn) election in 1887: 
Stockholm was a single
constituency with 22 seats (the by far largest constituency). After the
election, it was found that one of the 22 that had been elected was not
eligible, because he had unpaid back taxes. As a consequence, all 6206
ballots with his name were deemed to contain only 21 valid names, and were
therefore invalid. As a consequence, all 22 elected (that had received
4898--6749 votes) were replaced by others (that had received 2585--2988
votes)
\cite[pp.~45, 51]{SCB1885-1887}.
(This changed the majority in the parliament to the
protectionists, and as a result the government resigned.)
}

The Block Vote is the case $\MM\ge2$, the multi-member constituency version
of the  case $\MM=1$, which is the widely used \emph{Single-Member Plurality}
system, also called \emph{First-Past-The-Post}, where each voter votes for
one candidate, and the candidate with the largest number of votes wins.
Single-Member Plurality is the perhaps simplest possible
election method and has a long history, and it is still used in many
different  contexts.

Also the Block Vote (with $\MM\ge2$)
is simple and intuitive and has been used for a very long
time. 
(For example, it has been used since the Middle Ages
in English parliamentary elections, where
until 1885 most constituencies had 2 seats.
\cite[p.~158]{Politics})
Until the late 19th century, it seems to have been the dominant
election method for elections in multi-member constituencies.
It was the method used in Sweden when \phragmen{} and Thiele proposed their
methods.

The Block Vote is still used in general elections in some countries.
It is also widely used in non-political elections, for example in various
organizations and associations. (For example, in elections to department
boards at my university).

The Block Vote is well-known for not being proportional; when there are
organized parties, the largest party will get all seats.
Hence, many other election methods have been invented in order to get a
proportional result, among them the methods by \phragmen{} and Thiele
discussed in the present paper.

\subsubsection{Approval Voting}\label{Aapproval}
\emph{Each voter votes for an arbitrary number of candidates. The $\MM$
  candidates with the largest numbers of votes are elected.}

Here $\MM\ge1$.
In a system with well-organized parties, Approval Voting 
ought to give the same result
as the Block Vote (or Single-Member Plurality when $\MM=1$), see
\refApp{ABV}, with each party fielding $\MM$ candidates 
and their voters voting for these. 
In particular, the method is not proportional.
(Without parties, the result may be quite different from the Block Vote.)

This is also an old method, but seems to have been used much less
frequently than the Block Vote.
For example,
it was used (with $\MM>1$)
for parliamentary elections in Greece 1864--1926, and it is
used for distribution of seats within parties in some local elections in
Norway. 
\xfootnote{Versions requiring a qualified majority were used alread to elect
  the Pope 1294--1621, and the Doge (Duke) of Venice 1268--1797, in both
  cases (obviously) with  $\MM=1$.
}
As noted in \refSS{SSthiuopt}, the method is one of the methods
proposed by \citet{Thiele} in 1895,
viz.~ his ``strong method''.
The method was reinvented (and given the name Approval Voting) by Weber 
c.~1976
and was made widely known by for example
\citet{BramsFishburn-book};
see further \citet{Weber} and
\citet{BramsFishburn-mixed}.
Approval Voting
seems to be  used mainly in non-political elections, for example it is used
in several professional
organizations. 

\subsubsection{Limited Vote}\label{ALV}
\emph{Each voter votes for at most $\ell$ candidates, where $\ell$ is some
  given   number with $1<\ell<\MM$. 
The $\MM$ candidates with the
largest numbers of votes are elected.}
\xfootnote{As for the Block Vote, one version  requires each voter to
  vote for exactly $\ell$   candidates.} 

The idea of this modification of the Block Vote
is that a small party can concentrate its votes on (for example)
$\ell$ candidates and obtain at least some seats in the competition with a
larger party that spreads its votes over more candidates. This introduces a
degree of proportionality, but it is far from perfect. 
If $\ell> \MM/2$ (which in practice usually is the case), and the parties
are similar in size, the expected result of optimal voting strategies is
that the largest party gets $\ell$ seats, and the second largest the
remaining $\MM-\ell$ seats, while the other parties (if any) get none.

Moreover, the outcome
of an election depends heavily on how the votes of a party are distributed
inside the party. Hence the method makes it possible, and more or less
necessary, for a party that wants to win many seats,
to use schemes of tactical voting.
See \eg{} \cite[Section A.9]{SJV6} for the optimal strategy, and its result,
in an 
ideal situation when all voters follow the instructions from their parties,
and, moreover, the parties accurately know  in advance  the number of votes
for each party.

\citet{Phragmen1893} mentions Limited Vote as an example of an
unsatisfactory election method because of the possibilities of tactical
voting, saying that by dividing the votes between several lists,
a majority could get all seats, also against a rather strong
minority,
and that
this actually has happened in Spain (where the method was used since
1878).
Another well-known example of this was England,
where 
Limited Vote was used in some constituencies in
1867--1885, and the Liberal Party in Birmingham was very successful
in organising the voters so that they got all seats in each election;
this also gave the method a bad reputation
\cite[p.~193]{Carstairs}, \cite[p.~27]{Farrell}. %pp 192--193., %2.27

Limited Vote 
was proposed by 
\citet[Titre III, %Des Assemblées primaires
Section première, %Organisation des Assemblées primaires.
Article 4]{Condorcet:constitution}
already 1793 (for some elections in
the French republic; the proposal was not adopted), 
and it is still used in a few places.

Note that 
Limited Vote properly means the case when $\ell<\MM$, but if we ignore this
and allow an arbitrary $\ell$, then
Block Vote (\refApp{ABV}) is the special case $\ell=\MM$ and Approval Voting
(\refApp{Aapproval})
is the case $\ell=\infty$.

\subsubsection{Single Non-Transferable Vote (SNTV)}\label{ASNTV}
\emph{Each voter votes for one candidate.
The $\MM$ candidates with the
largest numbers of votes are elected.}
This can be seen as the special case $\ell=1$ of Limited Vote
(\refApp{ALV}),
and it has the same weaknesses as Limited Vote.
In particular, the outcome depends heavily on how the votes are distributed
inside the parties; note that a party can get hurt by too much concentration
of votes as well as by too little concentration.
In an ideal situation where all voters belong to parties 
and vote as instructed by their party, 
and, moreover,
the parties accurately know  in advance  the number of votes
for each party,
it can
be shown that optimal strategies yield the same result as D'Hondt's method,
see \cite[Appendix A.8]{SJV6}.
In a theoretical, and restricted, sense, SNTV can thus be regarded as a
proportional election method. However, just as Limited Vote (\refApp{ALV}),
it requires elaborate voting tactics, and bad tactics, \eg{} based on
incorrect guesses of the party's strength, can lead to
disastrous results.
\citet[p.~67]{Phragmen1895} discusses SNTV and regards it as unsatisfactory.

SNTV is an ancient method, and was mentioned already by Plato
\cite[book VI]{Plato}. %753 c-d
It was analysed mathematically by Charles Dodgson
\xfootnote{Better known by his pseudonym Lewis Carroll used when writing his
  famous children's classics.}
\cite{Dodgson} in 1884.
%\cite{Kopfermann}
It was used already in 1835 in North Carolina, 
% \cite[p.~17]{Flodstr}
and it is still used in a few
countries. However, it seems to have been much
less common than the Block Vote.

\subsubsection{Cumulative Voting}\label{ACV}
The idea is that a voter may divide his or her vote between several
candidates.
The perhaps the most common version is:

 \emph{Each voter has $\ell>1$ votes, that can be given to one or several
   candidates.}
(This seems to be
the oldest version of Cumulative Voting, used \eg{}
in the Cape Colony 1853--1909.)
See \eg{} \citet{Droop}.

Another version of Cumulative Voting, 
sometimes called \emph{Equal and Even Cumulative Voting}
\xfootnote{This version is
called \emph{Satisfaction Approval Voting} by 
\citet{BramsKilgour}.
}
uses unordered ballots:

\emph{Each voter votes for an arbitrary number of candidates.
A ballot with $m$ names is counted as $1/m$ votes for each candidate.}
The number of names on a ballot may be restricted, for example to at most
$\MM$, the number of seats.
This method is at present used in Peoria in Illinois, USA (with 5 seats and
at most 5 
names on each ballot).
In Sweden, it was proposed  by Rosengren in 1896
\cite[pp.~14--16]{Flodstr}, 
\cite[pp.~23--24]{Cassel}
and (for distribution of seats within parties)
% with  D'Hondt's method between parties
by a Royal Commission
%\xfootnote{including \phragmen{} as one of its members} 
1903 \cite{bet1903} and then
in bills by the government in 1904 and 1905,
but it was never adopted by parliament.
% Betänkande 27 okt 1903. Riksdagen protokoll 1904 (se Cassel).
% Prop 1904:51
% Cassel: Rosengrenska metoden.

\citet[p.~68]{Phragmen1895} discusses Cumulative Voting, saying that it
is similar to SNTV and has essentially the same deficiencies, and therefore
``should be seen as a theoretical experiment without practical usefulness''.

\subsection{Election methods with ordered ballots}\label{Aothero}

In these method, each ballot contains an ordered list of names.
The number of candidates on each ballot is
usually arbitrary, but may be limited to at most, or at least, or exactly
some given number, for 
example $\MM$, the number to be elected;
some versions require each voter to rank all candidates.
(In the latter case, 
each ballot can be seen as a permutation of the set of candidates.)
\xfootnote{In elections to the Australian Senate, all candidates have to be
  ranked. This is rather impractical, so
since 1983, each party makes a list ranking all candidates (its own and
all others), and the voter
has the option of voting for one of the party list.
This is used by the vast majority of voters, so the system in reality has
become a 
list system. \cite[pp.~140--141]{Farrell} 
}
%(In Tasmania, at least as many ranked as placed to be filled.
%\cite{Farrell} p 140

\subsubsection{Single Transferable Vote (STV)}\label{ASTV}
{ \em First, a quota is calculated, nowadays almost always the Droop quota
   (rounded to an integer or not), see \refS{Alist} below.
Each ballot is counted for its first name only
(at later stages, ignoring candidates that have been elected or eliminated).
A candidate whose number of
votes is
at least the quota is elected; the surplus, \ie, the ballots
exceeding the quota, are transferred to
the next (remaining) name on the ballot.
This is repeated as long as some unelected candidate reaches the quota.

If there is no such candidate, and not enough candidates have been elected,
then the candidate with the least number of votes is eliminated, and the 
votes for that candidate are transferred to the next name.

At the end, if the number of remaining candidates equals the number of
remaining seats, then all these are elected (regardless of whether  they
have reached the quota or not).}

This description is far from a complete definition; many details are
omitted, and can be filled in in different ways, see \eg{} \cite{Tideman}. 
As a result, STV is a
family of election methods rather than a single method. (There seems to be
hardly no two implementations that agree in all details.)
The most important differences between different version are in the treatment
of the surplus, but also other details, such as the order of transfer when
more than one candidate has reached the quota, 
the treatment of ballots where all names have been elected or eliminated,
and different rounding rules,
can affect the final outcome.
(See  \refE{EPhragmen-stv} for one example.)

The special case $\MM=1$, when only a single candidate is elected, is usually
called \emph{Alternative Vote}. This case is much simpler than the general
case, since no surplus has to be transferred, so the method only involves
eliminations until someone has reached the quota, which in this case means a
majority of the votes.

In the party list case (see \refS{Sparty}), STV yields the same result as a
quota method with the same quota (see \refS{Alist}), i.e.~in almost all
cases Droop's method (\refS{ADroop}).
(This is easily seen using the formulation in \refFn{fQuota}.)
STV is therefore regarded as a proportional election method.

The eliminations are an important part of the method. 
As a result of them,
even if the voters of a party split their votes by voting on the party's
candidates in different orders, so that no-one of them 
gets many votes in the first round (when only the first name is counted), 
as the counting progresses and 
weaker candidates are eliminated the votes will concentrate on the remaining
ones. Hence vote splitting within a party will typically not affect the
outcome, unlike with \phragmen's and Thiele's ordered methods.
\xfootnote{However, eliminations introduce also problems with
  non-monotonicity, see Footnote~\ref{f-monoSTV}.
}

STV is used, in different versions, in some (but rather few) countries, 
for example in all elections in Ireland
and in some in Northern Ireland, Scotland and Australia; 
it is the most common proportional election method in English-speaking
countries. 

There are two main groups of different methods
of  handling the
surplus when a candidate is elected: 
either some ballots (in number equal to the quota) are set aside, and the
others are transferred, or all ballots are transferred but with reduced
values, so that the total voting value is reduced by the quota.
\xfootnote{Methods of the latter type are often called \emph{Gregory
	methods}, after J.B. Gregory who proposed one such method in 1880.}
In the first case it has to be decided how to chose the ballots to be
transferred, and in the second case it has to be decided how to reduce the
values, and there are several versions in both cases.

Some of the methods to do the transfers of surplus from an elected candidate
are as follows
(again omitting some details),
see further \eg{} \cite{Farrell:1983}, \cite{Tideman}. 
In the descriptions, the quota
is denoted by 
$Q$, and we assume that a candidate just has been elected with $v$ votes;
the surplus is thus $v-Q$.
(``The last group of ballots transferred to a candidate'' means all ballots
counted for the candidate if there has not yet been any transfers to the
candidate.) 

\begin{romenumerate}
\item\label{STV-Cincinnati}
\emph{Random selection (the Cincinnati method).
$v-Q$ ballots are selected at random from the ballots counted for the
  elected candidate.
}

\item\label{STV-Ireland} \emph{Random stratified selection (Ireland).
The ballots to be transferred are taken from the last group of ballots
transferred to the elected candidate. 
These ballots are sorted according to
their next name, and a proportional number is selected for each next name.
The transferred ballots are selected as the ones last added to their group,
i.e., in practice they are selected randomly in each group. (The ballots are
 put in random order before the counting starts.)%
}
The surplus is thus transferred proportionally to the next name, but if
also that candidate is elected, the next name to transfer to has been chosen
by random selection.

This method is used in Ireland for the lower house (D\'ail) of the parliament.

\item \emph{The Gregory method.
All ballots in the last group of ballots
transferred to the elected candidate are transferred.
If there are $N$ ballots in this group, each of them is now counted as 
$(v-Q)/N$ votes.%
} 
This version was introduced in Tasmania in 1907 and has been used there
  since then. The Irish Senate uses a minor variation where $N$ is the
  numebr of ballots in the group that are transferable, i.e., have some
  remaining candidate on them.

\item\label{STV-inclusive} 
\emph{The inclusive Gregory method.
All ballots counted for the elected  candidate are transferred.
If there are $N$ ballots in this group, each of them is now counted as 
$(v-Q)/N$ votes.}
This is used \eg{} for the Australian Senate.

\item\label{STV-weighted}
 \emph{The weighted inclusive Gregory method.
Each ballot has an initial value of $1$ vote.
All ballots counted for the elected  candidate are transferred.
Each of them has its current value multiplied by
$(v-Q)/N$.%
}

This method was invented by \phragmen{} in c.~1893
(see also \refFn{fEnestrom}),
although for a version
of STV with unordered ballots and no eliminations, see \refSS{SSPh1}.
It seems to have been reinvented in 1986 
by the Proportional Representation Society of Australia, see
\cite{Farrell:1983}.

The method is used in Western Australia and in local elections in Scotland.
%http://www.prsa.org.au/history.htm#gregory
%http://www.austlii.edu.au/au/legis/wa/consol_act/ea1907103/sch1.html
\end{romenumerate}

%\item\emph{\andrae's method.}
The oldest version of STV, proposed by 
Carl \andrae{} in 1855 and used in some elections in Denmark 1856--1953, is
called \emph{\andrae's method}.
It can be seen as a primitive version; it differs from later versions in
that there are no eliminations.
\xfootnote{Eliminations were introduced in STV by Thomas Hare in 1865,
who had proposed earlier forms of STV in 1857 and 1859.
Thomas Hare was not the first to propose STV, since \andrae's method then
already was in use,
but he developed and promoted STV and is seen as the founder of
STV in its modern forms. For  the early history of STV, see further
\cite{Tideman}.} 

\begin{metod}{\andrae's method}
The Hare quota 
\xfootnote{In this case a misnomer, since \andrae{} invented the method, with
  the quota, before Hare's
  writings.}
(see \refApp{Alist})
is used.
The votes are counted in random order. Any candidate that reaches the quota
is elected, and ignored in the sequel. If not enough are elected when all
ballots are counted, then the
candidates that now have the largest number of votes are elected, provided
they have more than half the quota; if there still are remaining seats, the
votes are counted again, and if there are $m$ remaining seats, then each
ballot is now counted as a vote for the $m$ first candidates that not
already have been elected. The $m$ candidates with most votes are elected.
\end{metod}

This method was proposed for Sweden in 1896 by the government, but it was
not adopted.

\phragmen{} proposed a modification of \andrae's method,
which (unlike all other versions of STV) used \emph{unordered ballots},
see \refSS{SSPh1}.

\subsubsection{Bottoms-up}\label{Abottom}
\emph{%Ordered ballots. Constituencies with 1 or several setas.
Only the first name on each ballot is counted.
The candidate with the least number of votes is eliminated.
This is repeted (with eliminated names ignored) until only the desired
number $\MM$ of candidates remain; these are elected.}

This can be seen as STV (\refS{ASTV}) with the quota $\infty$,
so that no-one reaches the quota. (Hence there  is never any surplus to
consider.) 
When only one is elected, this gives the same result as Alternative Vote
(see \refS{ASTV}).

This very minor method was used in local elections in some
municipalities in South Australia 1985--1999.

Bottoms-up is mentioned here because it might be viewed as an ordered version
of Thiele's elimination method (\refSS{SSthi-}).
However, when $\MM>1$, it will in the party list
case (see \refS{Sparty})
\emph{not} give the same result as D'Hondt's method. 
Actually, in the party list case where everone from the same party votes for
the same ordered list, Bottoms-up will behave
strangely, since if there are at least $\MM$ parties,
everyone except the first candidate from each party will be
eliminated, and thus the seats will go to the $\MM$ largest parties, with one
seat each. This is the same outcome as Thiele's elimination method with the
weak satisfaction function \eqref{fweak}.
On the other hand, if instead the voters for each party vote 
are uniformly split
between all possible orderings of the candidates of their party, then
Bottoms-up will give the same result as Thiele's elimination method (which
has unordered ballots and thus ignores the order).

\subsubsection{Scoring rules (Borda methods)}\label{Aborda}
\emph{Each ballot is counted as $p_1$ votes for the first name, $p_2$ for
  the second name, and so on, according to some given non-increasing
  sequence $(p_k)_{k=1}^\infty$ of non-negative numbers.}

Several different sequences $(p_k)_k$ are used.
The method was suggested 
(for the election of a single person)
by Jean-Charles de Borda in 1770 \cite{Borda} 
with 
$p_k=n-k+1$ where $n$ is the number of candidates, i.e., the sequence
$n,n-1,\dots,1$ (his proposal required each voter to rank all candidates);
the same method was proposed already in 1433 by Nicolas Cusanus for election
of the king (and thus emperor, after being crowned by the pope) of the Holy
Roman Empire \cite{McLean-Borda,Pukelsheim:Cusa}. 
This scoring rule
is called the \emph{Borda method}; more generally, any scoring rule is
sometimes called a Borda method.

Another common choice of scoring rule uses the harmonic series, $p_k=1/k$.
This was
proposed in 1857 by Thomas Hare (who later instead developed and proposed
STV, \refS{ASTV});
it was used for parliamentary elections in Finland 1906--1935 (within parties),
and is currently used 
in Nauru. 
It was one of several methods discussed (and rejected) by the Swedish  Royal
Commission that instead proposed \phragmen's ordered method \cite{bet1913},
see \refApp{Ahistory}.
This version gives in the
party-list case (see \refS{Sparty}) the same result as D'Hondt's method, so
it may be regarded as a proportional method. 
(See further \cite{RamirezPalomares}.)
However, tactical voting where
the voters of a party vote on the same name in different orders may give
quite different results.

An important difference from \phragmen's and Thiele's methods and STV is
that with a scoring rule, also the order of the candidates below the first
unelected may influence the outcome. 
In particular, with a scoring rule,
adding names after the favourite candidate may decrease the
chances of the favourite.

Moreover, scoring rules invite to tactical voting, especially since the
ordering of all of the candidates is important.
\xfootnote{
Borda is said to have
commented this with: ``My scheme is only intended for honest men''.
\cite[pp.~182, 238]{Black}.
}

\subsection{Election methods with party lists}\label{Alist}
As said above, among the proportional election methods, the ones that
are most often used are party list methods, where the voter votes for a
party and the seats are distributed among the parties according to their
numbers of votes.
\xfootnote{
The most important proportional method of another type is STV, Section
\ref{ASTV}.
} 
The seats a party obtains then are distributed to candidates, either
according to a list made in advance by the party (\emph{closed list}), 
or in some way that more
or less is decided by the voters (\emph{open list}), see \refApp{Ahistory}
for examples from Sweden.

Conversely, most (but not all) party list methods that are used in practice,
including the ones below, are proportional, so that each party gets a
proportion of the  seats that approximates its proportion of votes.

Many different list methods are described in detail in \eg{} \citet{Pukelsheim};
see also \citet{BY} for the 
mathematically (but not politically) equivalent problem of
allocating the seats in the US House of Representatives proportionally among
the states.
\xfootnote{There, until 1941, the method was decided after each census,
so the choice of method was heavily influenced by its result.
Moreover, the number of seats was not fixed in advance and thus
also open to negotiations.
}

There are two major types of party list methods: \emph{divisor methods} and
\emph{quota methods}.

In the traditional formulation of divisor methods, 
seats are allocated
sequentially. For each seat, the number of votes of each party is reduced by
division with a \emph{divisor} that depends on the number of seats already
given to the party; the party with the highest quotient gets the seat.
\xfootnote{The quotient may be called  an \jfr; 
the Swedish term is \emph{j\"amf\"orelsetal} (comparative figure).}\,
\xfootnote{\label{fDivisor}%
A different, but equivalent, formulation is that
the number of seats for each party is obtained by selecting a number $D$
and then giving a party with $v_i$ votes $v_i/D$ seats, rounded to an
integer by some general rounding rule (different for different methods),
where $D$ is chosen such that the total number of seats given to the parties
is $\MM$.
See \citet{Pukelsheim} for a details and examples.
Similar formulations have also been used in the United States for
allocating the seats in the House of Representatives to the states
\cite{BY}.
The number $D$ is also called \emph{divisor}; it can be interpreted as roughly
the number 
of votes required for each seat. (Just as the quota in quota methods; the
difference is that the number $D$ is not determined in advance by a simple
formula, but comes out as a result of the calculations.)
}

In a quota method, first a \emph{quota} $Q$ is calculated; this is roughly the
number of votes required for each seat. The two main quotas that are used
are  the \emph{Hare quota} $V/\MM$ and the \emph{Droop quota} $V/(\MM+1)$
\cite{Droop}, where $V$
is the total number of (valid) votes and $\MM$ is the number of seats
(\cf{} \refR{RDroop});
the quota is often rounded to an integer, either up, or  down, or to the nearest
integer (see \eg{} \cite{Pukelsheim} and \cite{SJV6} for examples with
different, or no,  rounding).
The method then gives a party with $v_i$ votes first $\floor{v_i/Q}$ seats
(i.e., the integer part of $v_i/Q$); the remaining seats, if any, are given
to the parties with largest remainder in these divisions.
\xfootnote{\label{fQuota}%
An equivalent definition is that seats are distributed
  sequentially, as just described for divisor methods, but now the number of
  vote for a
  party is reduced by subtracting $Q$ for each seat that the party already
  has got. (This shows the close connection with STV, Section \ref{ASTV}).
}

\subsubsection{D'Hondt's method}\label{ADHondt}
\emph{The divisor method with the sequence of divisors $1,2,3,\dots$.}
Thus, the \jfr{} of a party equals the number of votes
divided by $1+$ the number of seats already obtained.
(The method is sometimes called \emph{highest average}, 
since the \jfr{} is the average number of votes per seat
for the party, if the party receives the next seat. However, this name is
also sometimes used as a term for all divisor methods.)

D'Hondt's method is one of the oldest and most important proportional list
methods. 
It was
%\emph{D'Hondt's method},
proposed by Victor D'Hondt in 1878 \cite{DHondt1878,DHondt1882}.
%\xfootnote{
D'Hondt's method is equivalent to \emph{Jefferson's method},
proposed by Thomas Jefferson in 1792 for 
allocating the seats in the US House of Representatives proportionally among
the states, see \cite{BY}.
\xfootnote{\label{fUSA}%
Jefferson's method is formulated
as in Footnote \ref{fDivisor}, with rounding downwards.
However, for the allocation of seats in the House of Representatives, 
the total number was not fixed in
advance; a suitable divisor $D$ was chosen by congress,
which makes the method somewhat different.
}

It is easy to see that D'Hondt's method slightly favours larger parties
(see \cite{SJ262} for a detailed calculation of the average bias).
In particular, the method is \emph{superadditive} in the sense that if two
parties merge to one (bringing all their voters with them to the new party,
while all other parties keep the same number of votes),
then the combined party will get at least as many seats as the two parties
had together. 
\xfootnote{The method thus discourages party splits, which was seen as an
  important positive feature of it in the discussions in Sweden in the early
  20th century.}
Moreover (and as a corollary), a party that gets a majority of
the votes will get at least half the seats.

In Sweden, D'Hondt's method was used 1909--1951, see \refApp{Ahistory}.

\subsubsection{Sainte-Lagu\"e's method}\label{AStL}
\emph{The divisor method with the sequence of divisors $1, 3, 5,\dots$.}
This makes it easier that D'Hondt's method for small parties to get a seat,
and the method is perfectly proportional in an average sense. (See e.g.\
\cite{SJ262}.) 

The method was proposed in 1910 by \citet{StL}.
It is equivalent to \emph{Webster's method},
proposed by Daniel Webster in 1832 (and used at least in 1842 and 1911) for
allocating the seats in the US House of Representatives proportionally among
the states, see \cite{BY}.
\xfootnote{
Webster's method is formulated
as in Footnote \ref{fDivisor}, with rounding to the nearest integer.
However, Footnote \ref{fUSA} applies here too.}
Sainte-Lagu\"e's method is used in several countries, e.g.\ Germany.

\subsubsection{Modified Sainte-Lagu\"e's method}\label{AStLmod}
\emph{The divisor method with the sequence of divisors $x, 3, 5,\dots$,
  where $x>1$ is a constant, usually $1.4$.}
This differs from
Sainte-Lagu\"e's method only in that the first divisor is larger than 1.
The method was introduced in Sweden in 1952 with the first divisor 1.4, and
this has also been used in a few other countries.
(From the next general election in Sweden, in 2018, the first
  divisor will   be 1.2 instead \cite{vallag}.)

The modification 
makes it more difficult than the standard Sainte-Lagu\"e's method for small
parties to get the first seat,
but it does not affect the distribution of seats between parties that all
have seats.

In Sweden, the modified Sainte-Laguë method has been  used since 1952, see
\refApp{Ahistory}. 

There is no special mathematical reason for the divisor 1.4; it was chosen
in Sweden 1952 because it was expected to give a politically desired result,
see \cite{vSydow}.
In the present Swedish system with adjustment seats, the first divisor does
not affect the total distribution of seats, but it affects the distribution
between constituencies within the parties, and the divisor was changed to
1.2 because simulations shows that this would 
give the best result with the present constituencies and party structure.
%(A result similar to the previously used D'Hondt's method, given the
%existing sizes
%of the constituencies and the party structure, 

\subsubsection{Method of greatest remainder  (Hare's method)}\label{Avalkvot}
\emph{The quota method using the Hare quota, see above}.

The method is
perfectly proportional in an average sense. (See e.g.\ \cite{SJ262}.)

In Sweden, the method has not been used in elections, but it is since 1894
used for the distribution of seats between constituencies (before the
election, based on the
population or on the number of eligible voters).

\subsubsection{Droop's method}\label{ADroop}
\emph{The quota method using the Droop quota, see above.}

The method slightly favours larger parties (but not as much as D'Hondt's
method).
(See e.g.\ \cite{SJ262}.)

\newcommand\vol{\textbf}
\newcommand\jour{\emph}
\newcommand\book{\emph}
\newcommand\inbook{\emph}
\def\no#1#2,{\unskip#2, no. #1,} %(typeset after year) 
\newcommand\toappear{\unskip, to appear}

\newcommand\arxiv[1]{\texttt{arXiv:#1}}
\newcommand\arXiv{\arxiv}

\def\nobibitem#1\par{}

\newcounter{bibrub}
\newcommand\bibrub[1]{\stepcounter{bibrub} \medskip 
  \centerline{\textsc{\Alph{bibrub}. #1}} 
  \medskip}
\newcounter{bibsubrub}[bibrub]
\newcommand\bibsubrub[1]{\stepcounter{bibsubrub} \medskip 
  \centerline{\textsc{\Alph{bibrub}\arabic{bibsubrub}. #1}} 
  \medskip}

\newcommand\urldag[1]{(#1)}
\newcommand\urlq[2]{\url{#1} \urldag{#2}}
\newcommand\xurl{\\\url}
\newcommand\xurlq{\\\urlq}

%\show\thebibliography

\end{document}